\documentclass[a4paper,10pt,twoside]{article}

\usepackage[T1]{fontenc}
\usepackage[utf8]{inputenc}
\usepackage[english]{babel}

\usepackage[top=3cm,bottom=3.2cm,left=3.2cm,right=3.2cm]{geometry}

\usepackage{amscd}
\usepackage{amsthm}
\usepackage{bbm}
\usepackage{bm}
\usepackage{braket}
\usepackage{changepage}
\usepackage[babel]{csquotes}
\usepackage{enumitem}
\usepackage{fancyhdr}
\usepackage[flushmargin,hang,symbol]{footmisc}
\usepackage{graphicx}
\usepackage{indentfirst}
\usepackage[intlimits]{mathtools}
\usepackage{mathrsfs}
\usepackage{microtype}
\usepackage{newtxtext,newtxmath}
\usepackage[english]{varioref}
\usepackage{xcolor}
\usepackage{xurl}

\usepackage{hyperref}
    \hypersetup{
	colorlinks = true,
	citecolor = black,
	filecolor = black,
	linkcolor = black,
	urlcolor = black
    }

\selectfont

\makeatletter
\renewcommand{\@maketitle}{
	\begin{flushleft}\Large\@title\par\vspace{3ex}
	\large\@author\par\vspace{0ex}
	\end{flushleft}}
\makeatother

\makeatletter
\renewcommand{\footnoterule}{
  \kern -3pt
  \vspace{1.6\baselineskip}
  \hrule width 0.4\columnwidth
  \vspace{0.8\baselineskip} 
  \kern -0.4pt
}
\makeatother

\newcommand{\mail}[1]{\href{mailto:#1}{\texttt{#1}}}

\newcommand{\N}{\mathbb{N}}

\newcommand{\R}{\mathbb{R}}

\newcommand{\cA}{\mathcal{A}}
\newcommand{\cB}{\mathcal{B}}
\newcommand{\cD}{\mathcal{D}}
\newcommand{\cE}{\mathcal{E}}
\newcommand{\cI}{\mathcal{I}}
\newcommand{\cO}{\mathcal{O}}
\newcommand{\cR}{\mathcal{R}}
\newcommand{\cU}{\mathcal{U}}

\newcommand{\ccI}{\mathscr{I}}

\newcommand{\bu}{\bm{u}}
\newcommand{\bx}{\bm{x}}
\newcommand{\bX}{\bm{X}}
\newcommand{\by}{\bm{y}}
\newcommand{\bY}{\bm{Y}}
\newcommand{\bz}{\bm{z}}

\renewcommand{\epsilon}{\varepsilon}
\renewcommand{\theta}{\vartheta}
\renewcommand{\rho}{\varrho}
\renewcommand{\phi}{\varphi}

\newcommand{\mo}[1]{\mathopen{#1}}
\newcommand{\mc}[1]{\mathclose{#1}}

\newcommand{\1}{\mathbbm{1}}

\newcommand{\eqdef}{\mathrel{\overset{\mathrm{def}}{=}}}

\DeclarePairedDelimiter{\rounds}{{\mathopen{(}}}{{\mathclose{)}}}
\DeclarePairedDelimiter{\bracks}{\lbrack}{\rbrack}
\DeclarePairedDelimiter{\braces}{\lbrace}{\rbrace}
\DeclarePairedDelimiter{\bracket}{\langle}{\rangle}
\DeclarePairedDelimiter{\abs}{\lvert}{\rvert}
\DeclarePairedDelimiter{\norm}{\lVert}{\rVert}

\newcommand{\Lped}[3]{L_{#3}^{#1}\rounds{#2}}

\DeclareMathOperator{\bP}{\mathbf{P}}
\DeclareMathOperator{\bE}{\mathbf{E}}

\DeclareMathOperator{\Var}{\mathbf{Var}}

\DeclareMathOperator{\Leb}{\bm{\lambda}}
\DeclareMathOperator{\Vol}{\mathbf{Vol}}

\DeclareMathOperator{\cl}{cl}
\DeclareMathOperator{\supp}{supp}

\DeclareMathOperator*{\argmin}{arg\,min}

\DeclareMathOperator*{\essinf}{ess\,inf}
\DeclareMathOperator*{\esssup}{ess\,sup}

\newenvironment{system}
{\left\lbrace\begin{array}{@{}l@{}}}
{\end{array}\right.}

\theoremstyle{plain}

    \newtheorem{proposition}{Proposition}[section]

	\newtheorem{corollary}{Corollary}[section]

\theoremstyle{definition}

\theoremstyle{remark}
	\newtheorem{remark}{Remark}[section]
	
\begin{document}


\title{\textbf{\Large{Integrated expectile-based measures of inequality}}}
\author{\textbf{Ignacio Cascos $\bm{\cdot}$ Marco Tarsia}} 
\date{}

\maketitle

{\let\thefootnote\relax\footnote{
	\textbf{Ignacio Cascos} \newline
	Departamento de Estadística, Universidad Carlos III de Madrid, 28911 Leganés (Madrid), Spain \newline
	E\:\!-mail: \mail{ignacio.cascos@uc3m.es} \newline
	Web page: \href{https://halweb.uc3m.es/esp/Personal/personas/icascos/eng/perso.html}{\texttt{https://halweb.uc3m.es/esp/Personal/personas/icascos/eng/perso.html}}
    \par\smallskip\noindent
    \textbf{Marco Tarsia} \newline
    Researcher, Italy \newline
    E\:\!-mail: \mail{marcotarsia90@gmail.com} \newline
    Web page: \href{https://marcotarsia.github.io/}{\texttt{https://marcotarsia.github.io}}
	}
}

\thispagestyle{empty}

\par\noindent\rule{\linewidth}{0.8pt}\par

\bigskip\bigskip


\begin{description} 

	\item[Abstract.] Expectiles provide a class of asymmetric location functionals that incorporate the magnitude of deviations and admit a natural geometric interpretation.
    Building on their structural consistency with the convex stochastic order, this paper introduces a family of integrated expectile functionals for measuring risk, dispersion, and inequality.
    The proposed functionals admit analytical representations as integrals of expectiles across asymmetry levels and, for a distinguished subclass, geometric representations in terms of weighted areas of star-shaped sets encoding distributional asymmetry.
    This approach yields a new class of expectile-based inequality indices, constituting a natural counterpart to classical Gini-type measures while preserving desirable monotonicity and consistency properties.
    Empirical counterparts are derived in closed form and admit explicit decompositions over finite samples.
    The framework extends naturally to multivariate settings through directional expectile constructions, leading to measures capable of capturing genuinely joint forms of multivariate dispersion and inequality.
	
	\item[Key words.] Expectiles, integrated expectile functionals, convex stochastic order, risk and deviation measures, inequality indices. 

	\item[AMS 2000 subject classifications.] 60E15, 62G30, 62P20, 91B30, 91G70. 

\end{description}

\bigskip\bigskip

\noindent\hfil\rule{0.5\textwidth}{.8pt}


\section*{Introduction}
\label{sec:introduction}

The quantitative assessment of dispersion, variability, and inequality constitutes a central theme across probability theory, statistics, economics, and mathematical finance, and plays a fundamental role in decision making under uncertainty. Classical approaches to these questions have been predominantly based on quantile-derived functionals, most prominently through the Lorenz curve (see~\cite{lorenz_1905}) and the associated Gini index (see~\cite{gini_1914}), which provide canonical summaries of distributional concentration and inequality. These constructions, further developed in statistical and economic analysis (see, e.g.,~\cite{gastwirth_1972,atkinson_1970,shorrocks_1980}), are deeply connected with stochastic orderings and majorization theory; see~\cite{schur_1923,marshall-olkin-arnold_2011,muller-stoyan_2002}. Within this framework, variability and inequality emerge as manifestations of partial ordering relations on probability distributions, most notably through the convex stochastic order.

Parallel developments in the theory of risk measurement have led to the formulation of law-invariant functionals capable of quantifying dispersion and uncertainty in a consistent and axiomatic manner. Foundational contributions establishing coherent and convex risk measures can be found in~\cite{artzner-delbaen-eber-heath_1999,follmer-schied_2002}, while spectral risk measures, defined as integral functionals of quantiles, are introduced in~\cite{acerbi_2002}. These developments highlighted the structural role of stochastic orderings in ensuring consistency and interpretability of dispersion and risk measures; see, among others,~\cite{muller_1997,bauerle-muller_2006,shaked-shanthikumar_2007}. More broadly, these approaches emphasize the importance of order-sensitive and asymmetric functionals as tools for describing heterogeneous and heavy-tailed phenomena.

Within this broader context, expectiles, introduced in~\cite{newey-powell_1987}, have emerged as natural and structurally rich alternatives to quantiles. Defined as minimizers of asymmetric quadratic loss functions, expectiles incorporate the size of deviations in a fundamentally different manner from rank-based functionals. In contrast to quantiles, which depend solely on distributional ordering, expectiles reflect the full geometry of deviations, thereby providing a smooth and intrinsically tail-sensitive characterization of asymmetry and dispersion. In particular, expectiles belong to the class of elicitable functionals, making them especially suitable for statistical inference and decision-theoretic applications. Their interpretation as generalized quantiles, elicitable functionals, and risk measures has been developed and studied in several works (see~\cite{jones_1994,bellini_2012,bellini-klar-muller-r.gianin_2014,bellini-di.bernardino_2017,kratschmer-zahle_2017}). Their asymptotic and statistical properties have been analyzed in~\cite{holzmann-klar_2016}, while further connections between expectiles, stochastic orderings, and measures of variability have been established in subsequent work (see~\cite{bellini-klar-muller_2018,eberl-klar_2023,bellini-fadina-wang-wei_2022}).

Beyond their analytical definition, expectiles admit a natural geometric interpretation. When indexed by asymmetry levels, expectiles generate families of nested regions that encode the distributional structure of random variables across the entire spectrum of asymmetry. This geometric perspective connects expectiles with central regions, depth functions, and multivariate dispersion analysis; see~\cite{mosler_2002,cascos-ochoa_2021}. Multivariate extensions of expectile-based constructions have been investigated in~\cite{herrmann-hofert-mailhot_2018,maume.deschamps-rulliere-said_2017}, further reinforcing the role of expectiles as structurally meaningful descriptors of distributional geometry.

Despite these advances, the systematic use of expectiles as fundamental building blocks for inequality measurement remains comparatively underdeveloped. Classical inequality indices, including the Gini index and its multivariate generalizations (see~\cite{koshevoy-mosler_1997,mosler_2023}), rely on rank-based representations that do not directly capture the extent of deviations. Recent work has emphasized the importance of developing alternative inequality measures capable of incorporating richer structural information; see~\cite{dong-tille-giorgi-guandalini_2024,el.methni-girard-laveur_2025}. In this regard, expectiles provide a natural framework for constructing dispersion and inequality measures grounded in both geometric and probabilistic structure.

The present paper develops a unified framework for constructing risk, dispersion, and inequality measures based on integrated expectile functionals. The central idea is that expectiles, viewed as geometric generators of distributional regions, induce natural aggregation mechanisms whose global properties reflect the underlying dispersion and asymmetry. By integrating suitable expectile-based quantities, we obtain law-invariant functionals satisfying consistency properties with respect to the convex stochastic order. This construction yields, in particular, a new class of expectile-based inequality indices, which admit, in certain cases, a geometric interpretation in terms of the weighted area of associated star-shaped sets. In this sense, the proposed indices constitute a natural counterpart to classical Lorenz–Gini measures, while providing a complementary perspective grounded in the geometry of expectiles rather than distributional ranks. This geometric formulation provides a unified language for asymmetric evaluation and comparison of risks, bridging classical inequality measurement and modern approaches based on stochastic orderings and elicitable functionals.

The framework extends naturally to multivariate settings, where expectile regions defined through directional projections generate inequality measures capable of capturing genuinely multidimensional dispersion. These constructions establish direct connections between expectile-based functionals, multivariate risk measures, and stochastic ordering theory; see also~\cite{molchanov-cascos_2016} and~\cite{mastrogiacomo-tarsia_2026}.

The remainder of the paper is organized as follows. Section~\ref{sec:expectile-dispersion} introduces the expectile framework and its geometric interpretation. Section~\ref{sec:risk-and-deviation-measures} develops the construction of integrated expectile-based functionals and establishes their main properties, including their consistency with the convex stochastic order. Section~\ref{sec:expectile-starshaped-set} introduces the expectile star-shaped set and investigates its geometric and analytical features. Section~\ref{sec:inequality-index} defines the associated inequality indices and studies their structural and ordering properties. Section~\ref{sec:multivariate-generalizations} presents multivariate extensions.
The paper concludes with Section~\ref{sec:conclusions}, where the main findings are summarized and several directions for future research are discussed.
Supplementary empirical results and technical proofs are collected in Appendices~\ref{sec:appendix_empirical-implementation},~\ref{sec:appendix_convexity},~\ref{sec:appendix_maximization}, and~\ref{sec:appendix_proof-of-K_d}.


\section{Expectile geometry and measures of dispersion}
\label{sec:expectile-dispersion}

Expectiles offer a natural and conceptually unifying bridge between stochastic orderings and geometric representations of dispersion.
Their intrinsic structure enables variability and inequality to be articulated through convex order-consistent functionals, endowed with both rigorous analytical meaning and a clear geometric interpretation.
This section lays out the foundational framework, progressing from the underlying stochastic order to the expectile function and its associated regions, establishing the conceptual and methodological groundwork for expectile-based measures of dispersion.


\subsection{Preliminaries and notation}
\label{subsec:preliminaries-notation}

We consider a complete probability space $\rounds{\Omega\:\!,\cA,\bP}$.
For any dimension $d \in \N^{\:\! *}\;\!\! \coloneqq \N \setminus \Set{\! 0 \!}$ and any integrability parameter $p \in \mo{[}\;\!\! 1 ,\:\!\! \infty \mc{[}$, we denote by 
\[
\Lped{p}{\bP}{d} 
\]
the quotient space of $\R^d$-\:\!valued, $\cA$-measurable random variables $\bX = { \bracks{ X_1, \dots, X_d } }^{\top \:\!\!} \!$ defined on $\Omega$ and admitting a finite $p$\:\!-moment with respect to $\bP$\;\!\!, that is,
\[
\bE\:\!\bracks[\big]{ {\norm{\bX}}^p } \equiv \int\nolimits_\Omega \:\!\! { \norm{\bX(\omega)} }^p d\:\!\!\bP(\omega) < \infty \:\! ,
\]
with the usual identification under $\bP$-\:\!a.s. equality.
Here, $\bE \:\! \bracks{\, \bm{\cdot} \,} \equiv \bE_{\;\! \bP \;\!\!} \:\! \bracks{\, \bm{\cdot} \,}$ stands for the expectation operator relative to $\bP$\;\!\!, while $\norm{\, \bm{\cdot} \,} \equiv \norm{\, \bm{\cdot} \,}_{\:\! \R^d \;\!\!}$ refers to the Euclidean norm on $\R^d$\:\!\!.
By convention, any deterministic vector $\bx \in \R^d \;\!\!$ is identified with the corresponding constant element of $\Lped{p}{\bP}{d}$.

We also adopt the following shorthand notation:
\[
\Lped{p \;\!\!}{\bP}{} \coloneqq \Lped{p \;\!\!}{\bP}{1} \:\! ,
\quad
\R_{\:\!+} \;\!\! \coloneqq \mo{[} 0 \:\!,\:\!\! \infty \mc{[} \:\! ,
\quad
\Lped{p}{\bP}{+} \coloneqq \Lped{p}{\bP}{1,+ \:\!\!} \:\! ,
\]
\vspace{-3.5ex}
\[
\R_{\:\!+}^{\:\! d} \coloneqq (\R_{\:\!+})^d \:\!\! \:\! ,
\quad
\Lped{p}{\bP}{d,+ \:\!\!} \coloneqq L^{p} \;\!\! \big{(} \rounds{\Omega\:\!,\cA,\bP} ; \R_{\:\!+}^d \:\! \big{)} \:\! .
\]

We next restrict attention to the scalar setting, corresponding to
\[
d = 1 \:\! .
\]
Accordingly, we henceforth drop boldface notation for scalar quantities, writing, for instance, $x \in \R$ in place of $\bx \in \R^d \;\!\!$, and $X \in \Lped{p \;\!\!}{\bP}{}$ rather than $\bX \in \Lped{p}{\bP}{d}$.
Throughout the sequel, we shall work predominantly with the integrability parameter
\[
p = 1 \:\! .
\]

For convenience, we also recall the standard notation for the positive and negative parts of a scalar $a \in \R$, namely $a_+ \coloneqq \max \:\!\! \Set{\! 0 \:\!,\;\!\! a \!}$ and $a_- \coloneqq (-a)_{\:\!+}$. In particular, one has $a = a_+ \;\!\! - a_-$ and $\abs{\:\! a \:\!} = a_+ \;\!\! + a_-$.



We now introduce the convex stochastic order, which provides a fundamental criterion for comparing integrable random variables sharing the same mean in terms of their dispersion. 

Formally, given any integrable random variables $X , Y \in \Lped{1 \;\!\!}{\bP}{}$, we write
\[
X \le_{\;\!\text{cx}} \;\!\! Y \;\!\!
\quad \Longleftrightarrow \quad
\forall \ \phi \colon \R \to \R \text{ convex} \:\! , \ \bE\:\!\bracks{ \phi(X) } \le \bE\:\!\bracks{ \phi(Y) } \:\! .
\]
Equivalently, the relation $X \le_{\;\!\text{cx}} \;\!\! Y$ holds if and only if $\bE\:\!\bracks{ X } = \bE\:\!\bracks{ Y }$ and, for every $t \in \R$,
\[
\int_{-\infty}^{t} \! F_{\:\! X}(s) \, ds
\:\! \ge \;\!\!
\int_{-\infty}^{t} \! F_{\;\! Y}(s) \, ds \:\! ,
\]
where $F_{\:\! X}$ and $F_{\;\! Y}$ denote the cumulative distribution functions of $X$ and $Y$\;\!\!, respectively; see~\cite{ruschendorf_1981,muller_1996,burgert-ruschendorf_2006}.
In this sense, the convex order formalizes the intuitive notion that $Y$ is more widely dispersed than $X$\;\!\!, while preserving the common mean.

The convex order occupies a prominent position in risk measurement, income inequality, and stochastic dominance theory.
Within the present framework, it serves as a natural benchmark for assessing the consistency (or monotonicity) of expectile-based functionals.
More precisely, a dispersion or inequality measure $\Psi(\;\! \bm{\cdot} \;\!)$ on $\Lped{1 \;\!\!}{\bP}{}$ is said to be consistent with the convex order whenever, for any $X , Y \in \Lped{1 \;\!\!}{\bP}{}$,
\begin{equation}
\label{eq:cx-consistency} 
X \le_{\;\!\text{cx}} \;\!\! Y \;\!\!
\quad \Longrightarrow \quad 
\Psi(X) \le \Psi(Y) \:\! .
\end{equation}
This property ensures that any increase in distributional spread, as assessed through the convex order, is reflected in a larger value of the measure, thereby preserving the intuitive notion of inequality or variability.

The connection between convex orderings and expectiles will be of central relevance in the developments that follow.
Indeed, as will be shown below, inclusions between suitable expectile-based sets turn out to encode the convex stochastic order.
This perspective naturally positions expectiles as fundamental building blocks for the construction of dispersion and inequality measures consistent with~$\le_{\;\!\text{cx}}$.


\subsection{Univariate expectiles}
\label{subsec:univariate-expectiles}

Let $\alpha \in \mo{]} 0 \:\!,\:\!\! 1 \;\!\!\mc{[}$ be fixed.
The $\alpha$-expectile $e_X(\alpha)$ associated with a real-valued random variable $X \in \Lped{1 \;\!\!}{\bP}{}$ is defined as the unique real number satisfying the first-order balance condition
\begin{equation} 
\label{eq:expectiles_1-ord}
(1 - \alpha) \;\!\! \bE\:\!\bracks[\big]{\;\!\! \big{(} X - e_X(\alpha) \big{)}_{\;\!\! - \:\!\!} } = \alpha \;\!\! \bE\:\!\bracks[\big]{\;\!\! \big{(} X - e_X(\alpha) \big{)}_{\;\!\! + \:\!\!} } \:\! .
\end{equation}
This formulation expresses a directional equilibrium between expected positive and negative deviations from $e_X(\alpha)$, weighted respectively by $\alpha$ and $1 - \alpha$.
Since it involves only first moments, \eqref{eq:expectiles_1-ord} is well defined for all integrable random variables and may be viewed as a natural extension of the classical least-squares perspective
originating in~\cite{newey-powell_1987}.

When $X \in \Lped{2}{\bP}{}$, the same value $e_X(\alpha)$ admits an equivalent asymmetric least-squares representation,
\begin{equation} 
\label{eq:expectiles_2-ord}
e_X(\alpha)
=
\argmin_{x \:\! \in \:\! \R} \ \bE\:\!\bracks[\Big]{ \alpha \bracks[\big]{\;\!\! {(X - x)}_+ }^{2}\:\!\! + (1 - \alpha) \bracks[\big]{\;\!\! {(X - x)}_- }^{2} } \:\! ,
\end{equation}
which yields the traditional interpretation of expectiles as solutions to an asymmetric quadratic
optimization problem, while remaining fully consistent with the primary $\Lped{1 \;\!\!}{\bP}{}$\:\!-\:\!based definition given above.
This quadratic formulation embodies a directional penalization of deviations from $x$, assigning weight $\alpha$ to positive excesses and $1 - \alpha$ to negative ones.

Consequently, expectiles retain quantitative information on the magnitude of deviations and are therefore not purely rank-based, unlike quantiles.
This feature will play a central role in the geometric constructions developed throughout the paper.

In this sense, expectiles may be regarded as a generalization of the mean, in that symmetric averaging
is replaced by a weighted equilibrium between positive and negative deviations.
The parameter $\alpha$ governs the degree of asymmetry in this balance: values $\alpha > 1/2$ place greater emphasis on upper-tail outcomes, whereas $\alpha < 1/2$ assign more weight on the lower tail.
When $\alpha = 1/2$, the penalization becomes symmetric and the resulting solution coincides with the mean, namely
\begin{equation}
\label{eq:mean} 
e_X(1/2\:\!) = \bE\:\!\bracks{ X } \:\! .
\end{equation}

A comprehensive account of the fundamental properties of expectiles is firmly established in the literature, starting from the seminal contribution of~\cite{newey-powell_1987} and including, among others,~\cite{bellini-klar-muller-r.gianin_2014}.
We do not attempt to reproduce this material here.
Instead, we refer the reader to these foundational contributions for a detailed and systematic treatment.

Throughout the present work, we deliberately refrain from invoking any general property of expectiles unless it is explicitly stated and used at the point where it becomes relevant.
In this way, the exposition remains entirely self-contained with respect to the specific arguments developed later.

Accordingly, we proceed directly to our first original results.
We begin by recalling the basic monotonicity property of expectiles.
If $X \le Y$ (\:\!$\bP$-\:\!a.s.), then $e_X(\alpha) \le e_Y(\alpha)$ for all $\alpha \in \mo{]} 0 \:\!,\:\!\! 1 \;\!\!\mc{[}$; moreover, if $\bP\:\!\bracks{ X < Y \:\! } > 0$ (that is, $X \ne Y$), then the inequality is strict: $e_X(\alpha) < e_Y(\alpha)$ for all $\alpha \in \mo{]} 0 \:\!,\:\!\! 1 \;\!\!\mc{[}$.

The next proposition gathers fundamental one-sided bounds for the upper and lower expectile functions in the nonnegative case.
These estimates are direct consequences of the first-order condition~\eqref{eq:expectiles_1-ord} and will play a recurring role in the subsequent analysis.
As shown later (cf. Section~\ref{subsec:two-point-expectile}), their structure is essentially sharp; see also Remark~\ref{rem:lower-upper-bounds} below.

For the proof, we employ the notation $a \land b \coloneqq \min \Set{\:\!\! a , \:\!\! b \:\!\!}$ for $a,b \in \R$.

\begin{proposition}
\label{prop:lower-upper-bounds} 
Let $X \in \Lped{1}{\bP}{+}$. Then, for every $\alpha \in \mo{]} 0 \:\!,\:\!\! 1/2 \mc{]}$,
\begin{equation}
\label{eq:lower-upper-bounds} 
e_X(\alpha) \ge \frac{\alpha}{1 - \alpha} \bE\:\!\bracks{ X } \:\! , \quad e_X(1 - \alpha) \le \frac{1 - \alpha}{\alpha} \bE\:\!\bracks{ X } \;\! .
\end{equation}
In addition, if $X \ne 0$ (\:\!$\bP$-\:\!a.s.) and $\alpha \ne 1/2$, then both inequalities in~\eqref{eq:lower-upper-bounds} are strict.
\end{proposition}

\begin{proof}
Fix $\alpha \in \mo{]} 0 \:\!,\:\!\! 1 \;\!\!\mc{[}$. By the elementary decompositions
\[
\begin{system}
\:\! X = \big{(} X - e_X(\alpha) \big{)}_{\;\!\! + \:\!\!} + \:\! e_X(\alpha)\:\!\!\land\:\!\!X \:\! , \\[0.75ex]
\:\! e_X(\alpha) = \big{(} e_X(\alpha) - X \big{)}_{\;\!\! + \:\!\!} + \:\! e_X(\alpha)\:\!\!\land\:\!\!X \:\! ,
\end{system}
\]
we obtain
\[
\begin{system}
\:\! \bE\:\!\bracks[\big]{ \big{(} X - e_X(\alpha) \big{)}_{\;\!\! + \:\!\!} }\;\!\! + \bE\:\!\bracks[\big]{ e_X(\alpha)\:\!\!\land\:\!\!X } = \bE\:\!\bracks{ X } \:\! , \\[1ex]
\:\! \bE\:\!\bracks[\big]{ \big{(} e_X(\alpha) - X \big{)}_{\;\!\! + \:\!\!} }\;\!\! + \bE\:\!\bracks[\big]{ e_X(\alpha)\:\!\!\land\:\!\!X } = e_X(\alpha) \:\! .
\end{system}
\]

Since $\big{(} X - e_X(\alpha) \big{)}_{\;\!\! - \:\!\!} \:\!\! = \big{(} e_X(\alpha) - X \big{)}_{\;\!\! + \:\!\!}$\:\!, the expectile condition~\eqref{eq:expectiles_1-ord} becomes
\[
(1 - \alpha) \:\! \braces[\Big]{\:\!  e_X(\alpha) - \bE\:\!\bracks[\big]{ e_X(\alpha)\:\!\!\land\:\!\!X } }
=
\alpha \:\! \braces[\Big]{\:\!  \bE\:\!\bracks{ X } - \bE\:\!\bracks[\big]{ e_X(\alpha)\:\!\!\land\:\!\!X } } \:\! ,
\]
which yields
\begin{equation}
\label{eq:lower-upper-bounds_proof} 
e_X(\alpha) = \frac{\alpha\:\!\!\bE\:\!\bracks{ X } + (1 - 2\:\!\alpha)\:\!\!\bE\:\!\bracks[\big]{ e_X(\alpha)\:\!\!\land\:\!\!X }}{1 - \alpha} \;\! .
\end{equation}

For the lower bound in~\eqref{eq:lower-upper-bounds}, note that $\alpha \le 1/2$, together with $e_X(\,\bm{\cdot}\,) \ge 0$ and $X \ge 0$ (\:\!$\bP$-\:\!a.s.), ensures
\[
(1 - 2\:\!\alpha)\:\!\!\bE\:\!\bracks[\big]{ e_X(\alpha)\:\!\!\land\:\!\!X } \ge 0 \:\! .
\]

Similarly, to derive the upper bound for $e_X(1 - \alpha)$, we apply~\eqref{eq:lower-upper-bounds_proof} at level $1 - \alpha$:
\[
e_X(1 - \alpha) = \frac{(1 - \alpha)\:\!\!\bE\:\!\bracks{ X } + (2\:\!\alpha - 1)\:\!\!\bE\:\!\bracks[\big]{ e_X(1 - \alpha)\:\!\!\land\:\!\!X }}{\alpha} \;\! .
\]
Here, when $\alpha \le 1/2$,
\[
(2\:\!\alpha - 1)\:\!\!\bE\:\!\bracks[\big]{ e_X(1 - \alpha)\:\!\!\land\:\!\!X } \le 0 \:\! ,
\]
and the conclusion follows.

If $X \ne 0$ (\:\!$\bP$-\:\!a.s.) and $\alpha \neq 1/2$, then both $\bE\:\!\bracks[\big]{ e_X(\alpha)\:\!\!\land\:\!\!X }$ and $\bE\:\!\bracks[\big]{ e_X(1 - \alpha)\:\!\!\land\:\!\!X }$ are strictly positive, and the coefficients $1 - 2\:\!\alpha$ and $2\:\!\alpha - 1$ above are nonzero. 
Hence the discarded terms in the two estimates contribute with a definite sign, and the two inequalities in~\eqref{eq:lower-upper-bounds} become strict. \qedhere
\end{proof}

\begin{remark}
\label{rem:lower-upper-bounds} 
Since $e_X(\,\bm{\cdot}\,)$ is nondecreasing and satisfies~\eqref{eq:mean}, the bounds of Proposition~\ref{prop:lower-upper-bounds} admit the following equivalent reformulation: for every $\alpha \in \mo{]} 0 \:\!,\:\!\! 1/2 \mc{]}$,
\[
\frac{\alpha}{1 - \alpha} \bE\:\!\bracks{ X } \le e_X(\alpha) \le \bE\:\!\bracks{ X } \le e_X(1 - \alpha) \le \frac{1 - \alpha}{\alpha} \bE\:\!\bracks{ X } \:\! .
\]
If, in addition, $X$ is non-degenerate (\:\!$\bP$-\:\!a.s.), then the monotonicity of $e_X(\,\bm{\cdot}\,)$ is strict, and all the above inequalities become narrow as soon as $\alpha \ne 1/2$.
\end{remark}

The next corollary extends the upper bound of Proposition~\ref{prop:lower-upper-bounds} to arbitrary integrable random variables through comparison with their positive part.
In particular, no nonnegativity assumption is required.

\begin{corollary}
\label{cor:upper-bound} 
Let $X \in \Lped{1 \;\!\!}{\bP}{}$. Then, for every $\alpha \in \mo{]} 0 \:\!,\:\!\! 1/2 \mc{]}$,
\begin{equation}
\label{eq:upper-bound_1} 
e_X(1 - \alpha) \le \frac{1 - \alpha}{\alpha} \bE\:\!\bracks{ X_{\:\! +} } \:\! .
\end{equation}
In addition, if $\bP\:\!\bracks{ X < 0 \:\! } > 0$, then the inequality in~\eqref{eq:upper-bound_1} is strict. In this case, in particular,
\begin{equation}
\label{eq:upper-bound_2} 
e_X(1 - \alpha) < \frac{1 - \alpha}{\alpha} \bE\:\!\bracks[\big]{ \abs{\:\! X \:\!} } \:\! .
\end{equation}
\end{corollary}

\begin{proof}
Since $X \le X_{\:\! +}$ (\:\!$\bP$-\:\!a.s.), the monotonicity of expectiles yields $e_X(1 - \alpha) \le e_{X_{\:\! +}}(1 - \alpha)$ for all $\alpha \in \mo{]} 0 \:\!,\:\!\! 1/2 \mc{]}$.
By the upper inequality in~\eqref{eq:lower-upper-bounds}, applied to the nonnegative random variable $X_{\:\! +}$, we obtain
\[
e_{X_{\:\! +}}(1 - \alpha)
\le
\frac{1 - \alpha}{\alpha} \bE\:\!\bracks{ X_{\:\! +} } ,
\]
which ultimately proves~\eqref{eq:upper-bound_1}.

If $\bP\:\!\bracks{ X < 0 \:\! } > 0$ (i.e. $X \notin \Lped{1}{\bP}{+}$), then $X < X_{\:\! +}$ on a positive probability set and $X_{\:\! +} < \abs{\:\! X \:\!}$ on the same set. Strict monotonicity of expectiles makes the inequality in~\eqref{eq:upper-bound_1} strict, and combining this with the strictness of the respective expectations yields~\eqref{eq:upper-bound_2}. \qedhere
\end{proof}

For any $X \in \Lped{1 \;\!\!}{\bP}{}$, the expectile function $e_X \colon \mo{]} 0 \:\!,\:\!\! 1 \;\!\!\mc{[} \:\!\! \to \R$ is well defined, strictly increasing, and continuous. It therefore admits a continuous extension to the closed interval $\mo{[} 0 \:\!,\:\!\! 1 \mc{]}$, determined by
\[
e_X(0) \coloneqq \essinf X \:\! , \quad e_X(1) \coloneqq \esssup X \:\! .
\]

As a consequence, $e_X$ defines a strictly increasing homeomorphism from $\mo{[} 0 \:\!,\:\!\! 1 \mc{]}$ onto
\[
\mo{[} \essinf X ,\;\!\! \esssup X \mc{]}
\]
(the latter interval being possibly unbounded).
The canonical extension of its inverse to $\R$, denoted by $e_X^{-1} \colon \R \to \mo{[} 0 \:\!,\:\!\! 1 \mc{]}$ and referred to as the inverse expectile function, admits the explicit representation
\begin{equation} 
\label{eq:inverse-expectile}
e_X^{-1}(x)
=
\begin{cases}
\:\! 0 \:\! , & \text{if $x \in \mo{]} -\infty \:\!,\:\!\! \essinf X \mc{]}$,} \\[0.75ex]
\:\! {\displaystyle \frac{ \bE\:\!\bracks[\big]{\;\!\! {(X - x)}_- } }{ \bE\:\!\bracks[\big]{ \abs{\:\! X - x \:\! } } } } , & \text{if $x \in \mo{]} \essinf X ,\;\!\! \esssup X \mc{[}$,} \\[0.75ex]
\:\! 1 \:\! , & \text{if $x \in \mo{[} \esssup X ,\:\!\! \infty \mc{[}$.}
\end{cases}
\end{equation}

The mapping $e_X^{-1}$ is continuous and strictly increasing on $\mo{]} \essinf X ,\;\!\! \esssup X \mc{[}$, and thus establishes a one-to-one correspondence between levels $\alpha \in \mo{]} 0 \:\!,\:\!\! 1 \;\!\!\mc{[}$ and thresholds $x = e_X(\alpha) \in \R$.

The convex stochastic order admits a complete description in terms of expectiles: for any $X , Y \in \Lped{1 \;\!\!}{\bP}{}$,
\begin{equation} 
\label{eq:cx-expectiles}
X \le_{\;\!\text{cx}} \;\!\! Y \;\!\!
\quad \Longleftrightarrow \quad
\begin{system}
\:\! \forall \ \alpha \in \mo{]} 0 \:\!,\:\!\! 1/2 \mc{]} , \ e_Y(\alpha) \le e_X(\alpha) \:\! , \\[0.5ex]
\:\! \forall \ \alpha \in \mo{[}\;\!\! 1/2 ,\:\!\! 1 \;\!\!\mc{[} , \ e_X(\alpha) \le e_Y(\alpha) \:\! .
\end{system}
\end{equation}
In particular, this characterization entails that $\bE\:\!\bracks{ X } = \bE\:\!\bracks{ Y }$ (see~\eqref{eq:mean}).

We further note that, for any $X , Y \in \Lped{1 \;\!\!}{\bP}{}$ satisfying $X \le_{\;\!\text{cx}} \;\!\! Y$\;\!\!, one has
$e_X(1/2\:\!) = e_Y(1/2\:\!)$ and
\[
\begin{system}
\:\! e_X^{-1}(x) \le e_Y^{-1}(x) \quad \text{for all } x \le e_X(1/2\:\!) \:\! , \\[0.75ex]
\:\! e_Y^{-1}(x) \le e_X^{-1}(x) \quad \text{for all } x \ge e_X(1/2\:\!) \:\! .
\end{system}
\]
Conversely, the above inequalities for the inverse expectile functions $e_X^{-1}$ and $e_Y^{-1}$ imply that $X \le_{\;\!\text{cx}} \;\!\! Y$\:\!\!.

Beyond these formal properties, a conceptual distinction with respect to quantiles deserves emphasis.
The computation of each $e_X(\alpha)$, $\alpha \in \mo{]} 0 \:\!,\:\!\! 1 \;\!\!\mc{[}$, depends on the entire distribution of $X$: both sides of~\eqref{eq:expectiles_1-ord} involve expectations of the positive and negative parts of $X - e_X(\alpha)$, thus aggregating information over the full support of $X$\;\!\!.
In contrast, quantiles are local functionals, determined solely by the distribution function at a prescribed probability level.
Expectiles thus encode a global, distribution-wide balance of deviations, reflecting asymmetries across the entire tail structure rather than at a single probability threshold.


\subsection{Expectile regions}
\label{subsec:expectile-regions}

Let $X \in \Lped{1 \;\!\!}{\bP}{}$. For any $\alpha \in \mo{]} 0 \:\!,\:\!\! 1/2 \mc{]}$, the \emph{expectile region} of $X$ at level $\alpha$ is defined as the interval
\begin{equation}
\label{eq:expectiles-regions} 
E_X(\alpha) \eqdef \mo{[} e_X(\alpha) \:\!,\;\!\! e_X(1 - \alpha) \mc{]} \:\! .
\end{equation}

Each interval $E_X(\alpha)$ comprises all values of $x \in \R$ lying between the lower and upper expectiles corresponding to the symmetric levels $\alpha$ and $1 - \alpha$.
Intuitively, the region provides a geometric representation of how the distribution extends in opposite directions from its center once asymmetry is taken into account.
The two endpoints therefore encode a weighted notion of centrality around the mean.
These regions provide a multiscale description of distributional imbalance as the expectile level varies.
Thus, the passage from individual expectiles to expectile regions marks the first step from a pointwise description toward a geometric representation of the distribution.

As $\alpha$ decreases from $1/2$ to $0$, the regions $E_X(\alpha)$ expand monotonically, in a symmetric fashion about $\bE\:\!\bracks{ X }$, capturing the progressive inclusion of tail behavior. 
In other words, the family
\[
\big\{\:\! E_X(\alpha) \big\}_{\:\!\! \alpha \in \mo{]} 0 \:\!,\:\!\! 1/2 \mc{]}}
\] 
forms a nested sequence of intervals conveying the dispersion structure of $X$\;\!\!.

It is therefore natural to compare random variables through the inclusion of their expectile regions.
This intuition is formalized by~\eqref{eq:cx-expectiles}: for any $Y \in \Lped{1 \;\!\!}{\bP}{}$,
\begin{equation}
\label{eq:cx-expectiles-regions} 
X \le_{\;\!\text{cx}} \;\!\! Y \;\!\!
\quad \Longleftrightarrow \quad
\forall \ \alpha \in \mo{]} 0 \:\!,\:\!\! 1/2 \mc{]} \:\! , \ E_X(\alpha) \subseteq E_{\:\! Y}(\alpha) \:\! .
\end{equation}
Hence, the convex stochastic order admits a geometric characterization through the nesting of expectile regions.
This provides a geometric reformulation of convex-order comparisons, whereby probabilistic dominance is expressed through set inclusion.

From this perspective, the convex stochastic order compares random variables through the relative extent of their expectile regions.
The inclusion $E_X(\alpha) \subseteq E_{\:\! Y}(\alpha)$ for all $\alpha \in \mo{]} 0 \:\!,\:\!\! 1/2 \mc{]}$ indicates that, at every asymmetry level, the distribution of $Y$ assigns relatively more probability mass to extreme outcomes than that of $X$\;\!\!. 
Accordingly, $Y$ exhibits greater dispersion (while preserving the same mean), in full agreement with the interpretation of the convex order $X \le_{\;\!\text{cx}} \;\!\! Y$\:\!\!.

As a direct consequence of Proposition~\ref{prop:lower-upper-bounds}, we obtain the following universal upper bound on the length of the one-dimensional expectile region in the nonnegative case.
The estimate provides a simple quantitative link between the size of expectile regions and the mean of the underlying distribution.

For this purpose, we write
\[
\Leb\:\!\bracks{\, \bm{\cdot} \,} \equiv \Leb^{\;\!\! 1 \;\!\!} \bracks{\, \bm{\cdot} \,}
\]
for the one-dimensional Lebesgue measure on the Borel $\sigma$-\:\!algebra of $\R$.

\begin{proposition}
\label{cor:lower-upper-bounds} 
Let $X \in \Lped{1}{\bP}{+} \setminus \Set{\! 0 \!}$. Then, for every $\alpha \in \mo{]} 0 \:\!,\:\!\! 1/2 \mc{[}$,
\begin{equation*}
\Leb\:\!\bracks[\big]{ E_X(\alpha) }
\equiv \;\!
e_X(1 - \alpha) - e_X(\alpha)
<
\frac{1 - 2\:\!\alpha}{\alpha\:\!(1 - \alpha)} \bE\:\!\bracks{ X } \:\! .
\end{equation*}
\end{proposition}

Beyond this setting, the length of $E_X(\alpha)$ can be controlled, without any sign assumption, by the first absolute moment around an appropriately chosen centre within the region, revealing that the geometric extent of expectile regions is intrinsically constrained by familiar measures of statistical dispersion.

This observation is made precise in the following result; see also Remark~\ref{rem:median-vs-mean}.

\begin{proposition}
\label{prop:lenght-expectiles-regions} 
Let $X \in \Lped{1 \;\!\!}{\bP}{}$. Then, for every $\alpha \in \mo{]} 0 \:\!,\:\!\! 1/2 \mc{]}$,
\begin{equation}
\label{eq:lenght-expectiles-regions} 
\Leb\:\!\bracks[\big]{ E_X(\alpha) }
\le
\frac{1 - 2\:\!\alpha}{\alpha} \! \inf_{m \:\! \in \:\! E_X\;\!\!(\alpha) } \:\!\! \bE\:\!\bracks[\big]{\:\! \abs{\:\! X \:\!\! - m \:\!} \:\!} .
\end{equation}
In addition, if $X$ is non-degenerate and $\alpha \ne 1/2$, then the inequality in~\eqref{eq:lenght-expectiles-regions} is strict.
\end{proposition}

\begin{proof}
Fix $\alpha \in \mo{]} 0 \:\!,\:\!\! 1 \;\!\!\mc{[}$. By~\eqref{eq:expectiles_1-ord}, we get
\[
\bE\:\!\bracks[\big]{\;\!\! \big{(} X - e_X(\alpha) \big{)}_{\;\!\! + \:\!\!} }
=
\frac{1 - \alpha}{\alpha} \bE\:\!\bracks[\big]{\;\!\! \big{(} X - e_X(\alpha) \big{)}_{\;\!\! - \:\!\!} } \:\! ,
\]
from which
\begin{equation}
\label{eq:proof-expectiles_1-ord} 
\bE\:\!\bracks[\big]{\;\!\! \big{(} X - e_X(\alpha) \big{)} } \equiv
\bE\:\!\bracks[\big]{\;\!\! \big{(} X - e_X(\alpha) \big{)}_{\;\!\! + \:\!\!} }\;\!\! - \bE\:\!\bracks[\big]{\;\!\! \big{(} X - e_X(\alpha) \big{)}_{\;\!\! - \:\!\!} } = \frac{1 - 2\:\!\alpha}{\alpha} \bE\:\!\bracks[\big]{\;\!\! \big{(} X - e_X(\alpha) \big{)}_{\;\!\! - \:\!\!} } \:\! .
\end{equation}
Equivalently,
\[
\bE\:\!\bracks[\big]{\;\!\! \big{(} X - e_X(1 - \alpha) \big{)} } = \frac{2\:\!\alpha - 1}{\alpha} \bE\:\!\bracks[\big]{\;\!\! \big{(} X - e_X(1 - \alpha) \big{)}_{\;\!\! + \:\!\!} } \:\! .
\]
Therefore,
\begin{equation}
\label{eq:lenght-expectiles-regions_proof} 
\begin{split}
e_X(1 - \alpha) - e_X(\alpha) &\equiv \bE\:\!\bracks[\big]{\;\!\! \big{(} X - e_X(\alpha) \big{)} } - \bE\:\!\bracks[\big]{\;\!\! \big{(} X - e_X(1 - \alpha) \big{)} } \\[1ex]
&= \frac{1 - 2\:\!\alpha}{\alpha} \:\! \braces[\Big]{\;\! \bE\:\!\bracks[\big]{\;\!\! \big{(} X - e_X(\alpha) \big{)}_{\;\!\! - \:\!\!} }\;\!\! + \bE\:\!\bracks[\big]{\;\!\! \big{(} X - e_X(1 - \alpha) \big{)}_{\;\!\! + \:\!\!} } } .
\end{split}
\end{equation}

Now, let $\alpha \le 1/2$. Then, for any $m \in \R$ such that $e_X(\alpha) \le m \le e_X(1 - \alpha)$, that is, $m \in E_X(\alpha)$, one has
\begin{equation}
\label{eq:conclusion} 
\begin{system}
\:\! \big{(} X - e_X(\alpha) \big{)}_{\;\!\! - \:\!\!} \le \big{(} X \:\!\! - m \big{)}_{\;\!\! - \:\!\!} \:\! , \\[0.75ex]
\:\! \big{(} X - e_X(1 - \alpha) \big{)}_{\;\!\! + \:\!\!} \le \big{(} X \:\!\! - m \big{)}_{\;\!\! + \:\!\!} \:\! ,
\end{system}
\end{equation}
thus
\[
\begin{system}
\:\! \bE\:\!\bracks[\big]{\;\!\! \big{(} X - e_X(\alpha) \big{)}_{\;\!\! - \:\!\!} } \le \bE\:\!\bracks[\big]{\;\!\! \big{(} X \:\!\! - m \big{)}_{\;\!\! - \:\!\!} } \:\! , \\[1ex]
\:\! \bE\:\!\bracks[\big]{\;\!\! \big{(} X - e_X(1 - \alpha) \big{)}_{\;\!\! + \:\!\!} } \le \bE\:\!\bracks[\big]{\;\!\! \big{(} X \:\!\! - m \big{)}_{\;\!\! + \:\!\!} } \:\! ,
\end{system}
\]
and, by inserting these bounds into~\eqref{eq:lenght-expectiles-regions_proof}, we obtain
\[
e_X(1 - \alpha) - e_X(\alpha) \le \frac{1 - 2\:\!\alpha}{\alpha} \:\! \bE\:\!\bracks[\big]{\:\! \abs{\:\! X \:\!\! - m \:\!} \:\!} .
\]
Taking the infimum over all such $m$ yields the first claim.

If $X$ is non-degenerate (\:\!$\bP$-\:\!a.s.) and $\alpha \neq 1/2$, then for every $m \in E_X(\alpha)$ at least one of the two inequalities in~\eqref{eq:conclusion} is necessarily strict, and the strict version of~\eqref{eq:lenght-expectiles-regions} follows. \qedhere
\end{proof}

\begin{remark}
It is useful to record that, combining~\eqref{eq:expectiles_1-ord} with~\eqref{eq:proof-expectiles_1-ord} yields, for any $\alpha \in \mo{]} 0 \:\!,\:\!\! 1 \;\!\!\mc{[}$,
\begin{equation}
\label{eq:rem-expectiles_1-ord} 
e_X(\alpha) - \bE\:\!\bracks{ X }
=
\frac{2\:\!\alpha - 1}{1 - \alpha} \bE\:\!\bracks[\big]{\;\!\! \big{(} X - e_X(\alpha) \big{)}_{\;\!\! + \:\!\!} } \:\! .
\end{equation}
\end{remark}

Because $\bE\:\!\bracks{ X } \in E_X(\alpha)$ whenever $\alpha \le 1/2$, Proposition~\ref{prop:lenght-expectiles-regions} directly implies the next corollary. Its statement coincides with the upper bound in~\cite[Theorem~14]{eberl-klar_2023} for inter-expectile ranges.

\begin{corollary}
\label{cor:lenght-expectiles-regions_0} 
Let $X \in \Lped{1 \;\!\!}{\bP}{}$ non-degenerate. Then, for every $\alpha \in \mo{]} 0 \:\!,\:\!\! 1/2 \mc{[}$,
\begin{equation*}
\Leb\:\!\bracks[\big]{ E_X(\alpha) } < \frac{1 - 2\:\!\alpha}{\alpha} \:\! \bE\:\!\bracks[\Big]{\:\! \abs[\big]{\:\! X \:\!\! - \bE\:\!\bracks{ X } \:\!} \:\!} .
\end{equation*}
\end{corollary}

Proposition~\ref{cor:lower-upper-bounds} and Corollary~\ref{cor:lenght-expectiles-regions_0} yield two complementary upper bounds for the length of the expectile region $E_X(\alpha)$.
While the first relies on the nonnegativity of~$X$, the second holds without sign restrictions and is expressed in terms of the first absolute central moment.
When $X$ is nonnegative, both estimates apply simultaneously, leading to the following combined bound.

\begin{corollary}
\label{cor:lenght-expectiles-regions} 
Let $X \in \Lped{1}{\bP}{+}$ non-degenerate. Then, for every $\alpha \in \mo{]} 0 \:\!,\:\!\! 1/2 \mc{[}$,
\begin{equation}
\label{eq:lenght-expectiles-regions_cor} 
\Leb\:\!\bracks[\big]{ E_X(\alpha) } < \frac{1 - 2\:\!\alpha}{\alpha} \:\! \min \;\! \braces[\bigg]{ \frac{1}{1 - \alpha} \bE\:\!\bracks{ X } \;\! , \bE\:\!\bracks[\Big]{\:\! \abs[\big]{\:\! X \:\!\! - \bE\:\!\bracks{ X } \:\!} \:\!} } \:\! .
\end{equation}
\end{corollary}

A useful observation is that, in the nonnegative case considered here, the active bound in~\eqref{eq:lenght-expectiles-regions_cor} depends on the dispersion of $X$\;\!\!. 
When $X$ is only mildly spread around its mean, the term involving the first absolute central moment typically yields the sharper estimate, whereas for more widely dispersed random variables the linear bound $\bE\:\!\bracks{ X }/(1 - \alpha)$ may prevail. 
In this way, the effective upper bound on $\Leb\:\!\bracks[\big]{ E_X(\alpha) }$ reflects the interplay between the scale of $X$ and its variability.

Finally, in connection with the above results, note also that the prefactor $(1 - 2\:\!\alpha)/\alpha$ is positive on $\mo{]} 0 \:\!,\:\!\! 1/2 \mc{[}$, diverges as $\alpha \downarrow 0$, and decreases strictly to $0$ as $\alpha \uparrow 1/2$.

\begin{remark}
\label{rem:median-vs-mean} 

The effectiveness of the bound in Proposition~\ref{prop:lenght-expectiles-regions} depends crucially on the choice of the centre $m \in E_X(\alpha)$.
Since the function
\[
m \longmapsto \bE\:\!\bracks[\big]{\:\! \abs{\:\! X \:\!\! - m \:\!} \:\!}
\]
is continuous, convex, and coercive (i.e., it diverges as $\abs{\:\! m \:\!} \to \infty$), the infimum in~\eqref{eq:lenght-expectiles-regions} is in fact attained.
As this map is minimized precisely at the medians of $X$\;\!\!, the
best choice within $E_X(\alpha)$ is obtained by projecting a median onto this interval.
Accordingly, if a median lies strictly inside $E_X(\alpha)$, it delivers the sharpest possible bound; if it lies outside, the optimum is reached at the nearest endpoint.

This explains why the mean $\bE\:\!\bracks{ X }$, although natural from several perspectives, is not typically optimal here: it does not in general minimize $\bE\:\!\bracks[\big]{\:\! \abs{\:\! X \:\!\! - m \:\!} \:\!}$, especially for skewed or heavy-tailed distributions.
In such cases, a median, or its projection onto $E_X(\alpha)$, provides a markedly superior centre for the inequality.

Nevertheless, in Sections~\ref{sec:inequality-index} and~\ref{sec:multivariate-generalizations} we shall adopt the choice
\[
m = \bE\:\!\bracks{ X }
\]
in keeping with the conceptual and technical structure underlying those sections.

\end{remark}


\section{Integrated expectile functionals as risk and deviation measures}
\label{sec:risk-and-deviation-measures}




We consider a class of law-invariant scalar functionals obtained by integrating expectiles with respect to an auxiliary probability measure on the unit interval, in the spirit of the approach underlying expected shortfalls in~\cite{acerbi_2002}.
With a suitable sign convention, this scheme gives rise to positively homogeneous risk measures, while coherence is ensured whenever the integrating measure is supported on $\mo{[} 0 \:\!,\:\!\! 1/2 \mc{]}$.

Beyond risk assessment, expectiles also serve as a natural basis for generalized deviation measures, defined as integrals of the distance between expectiles and the mean of the underlying random variable, taken  with respect to a nontrivial measure on $\mo{[} 0 \:\!,\:\!\! 1 \mc{]}$.
Particular attention is devoted to measures that are symmetric about $1/2$, which lead to balanced and interpretable deviations.

In both settings, the efficient computation of the resulting integrals is essential for practical applicability.
For the risk-measure framework, explicit empirical implementations are provided later in the appendix.


\subsection{Integrated expectile risk measures}
\label{subsec:risk-measures}


In line with standard terminology (e.g.,~\cite{artzner-delbaen-eber-heath_1999,follmer-schied_2002}), a \emph{risk measure} associated with $\bP$ and $p \in \mo{[}\;\!\! 1 ,\:\!\! \infty \mc{[}$ is a functional $\cR \colon \Lped{p \;\!\!}{\bP}{} \to \R$ satisfying the following three properties.
\begin{description} 

\item[\textbf{($\cR$-\:\!nor)}] \hypertarget{item:Rnor}{\label{item:Rnor} --- \emph{Normalization}.
\begin{equation*} 
\label{eq:Rnor}
\cR(0) = 0 \;\! .
\end{equation*}
} 

\item[\textbf{($\cR$-\:\!inv)}] \hypertarget{item:Rinv}{\label{item:Rinv} --- \emph{Translation invariance}.
For any $X \in \Lped{p \;\!\!}{\bP}{}$ and $c \in \R$,
\begin{equation*} 
\label{eq:Rinv}
\cR(X + c) = \cR(X) - c \:\! .
\end{equation*}
} 

\item[\textbf{($\cR$-\:\!mon)}] \hypertarget{item:Rmon}{\label{item:Rmon} --- \emph{Monotonicity}.
For any $X , Y \in \Lped{p \;\!\!}{\bP}{}$,
\begin{equation*} 
\label{eq:Rmon}
X \le Y \text{ (\:\!$\bP$-\:\!a.s.)}
\quad \Longrightarrow \quad
\cR(Y) \le \cR(X) \:\! .
\end{equation*}
} 

\end{description}

A risk measure $\cR(\;\! \bm{\cdot} \;\!)$ is \emph{positively homogeneous} if:
\begin{description} 
\item[\textbf{($\cR$-\:\!hom)}] \hypertarget{item:Rhom}{\label{item:Rhom} --- \emph{Positive homogeneity}.
For any $X \in \Lped{p \;\!\!}{\bP}{}$ and $t \in \mo{]} 0 \:\!,\:\!\! \infty \mc{[}$,
\begin{equation*} 
\label{eq:Rhom}
\cR(t X) = t \;\! \cR(X) \:\! .
\end{equation*}
} 
\end{description}
It is called \emph{subadditive} when:
\begin{description} 
\item[\textbf{($\cR$-\:\!sub)}] \hypertarget{item:Rsub}{\label{item:Rsub} --- \emph{Subadditivity}.
For any $X , Y \in \Lped{p \;\!\!}{\bP}{}$,
\begin{equation*} 
\label{eq:Rsub}
\cR(X + Y) \le \cR(X) + \cR(Y) \:\! .
\end{equation*}
} 
\end{description}
Any risk measure that is both positively homogeneous and subadditive is said to be \emph{coherent}.


For any probability measure $\bm{\nu}$ on the Borel $\sigma$-\:\!algebra $\cB \big{(} \mo{[} 0 \:\!,\:\!\! 1 \mc{]} \big{)}$, we define the \emph{$\bm{\nu}$\:\!-\:\!integrated expectile} on $\Lped{p \;\!\!}{\bP}{}$ by 
\begin{equation}
\label{eq:integrated-expectile} 
R^{\:\! e}_{\bm{\nu}}(X)
\eqdef
- \:\!\! \int_{0}^{1} \:\!\!\! e_X(\alpha) \, d\bm{\nu}(\alpha) \:\! , \quad X \in \Lped{p \;\!\!}{\bP}{} \:\! .
\end{equation}
Recall that $\alpha \mapsto e_X(\alpha)$ is nondecreasing and hence Borel-measurable on $\mo{[} 0 \:\!,\:\!\! 1 \mc{]}$.
In the sequel, we restrict attention to choices of probability measures $\bm{\nu}$ for which $R^{\:\! e}_{\bm{\nu}}(X)$ is finite for every $X \in \Lped{p \;\!\!}{\bP}{}$.

The functional~\eqref{eq:integrated-expectile} is then a law-invariant risk measure satisfying positive homogeneity.
Here, law invariance is understood in the usual sense, namely dependence on probability laws under $\bP$ only.

Moreover, whenever $\bm{\nu}$ is supported on $\mo{[} 0 \:\!,\:\!\! 1/2 \mc{]}$, $R^{\:\! e}_{\bm{\nu}}$ is subadditive, and hence coherent.
In this case, by~\eqref{eq:cx-expectiles}, it is also consistent with the convex order (in the sense of~\eqref{eq:cx-consistency}):
\[
X \le_{\;\!\text{cx}} \;\!\! Y \;\!\!
\quad \Longrightarrow \quad
R^{\:\! e}_{\bm{\nu}}(X)
\le
R^{\:\! e}_{\bm{\nu}}(Y) \:\! .
\]
Thus, integrating expectiles over $\mo{[} 0 \:\!,\:\!\! 1/2 \mc{]}$ does not alter the convex-order ranking induced by the individual expectile levels, but merely aggregates it according to the preference structure specified by~$\bm{\nu}$\;\!\!.
The restriction
\[
\supp \:\! (\bm{\nu})
\subseteq
\mo{[} 0 \:\!,\:\!\! 1/2 \mc{]} 
\]
is generally essential for both the preceding coherence and consistency properties. 

Since the expectile function $e_X$ may exhibit unbounded behavior as $\alpha \downarrow 0$ or $\alpha \uparrow 1$ for unbounded random variables, the finiteness of the integral is typically influenced by both the tail behavior of $X$ and the mass of $\bm{\nu}$ near $0$ and $1$.
On the other hand, such issues do not arise, for instance, if $X$ is essentially bounded or if $\bm{\nu}$ is supported in a compact subinterval of $\mo{]} 0 \:\!,\:\!\! 1 \;\!\!\mc{[}$.

This aggregation mechanism is structurally analogous to spectral constructions based on quantiles, while retaining information on the magnitude of deviations through expectiles.

From an interpretative viewpoint, the $\bm{\nu}$-integrated expectile may be regarded as an aggregation of tail-sensitive location functionals associated with different degrees of asymmetry.
Each expectile level $\alpha$ reflects a specific balance between upper and lower tail contributions, while the integrating measure $\bm{\nu}$ determines how these local asymmetry assessments are combined into a global risk evaluation.
In this perspective, $\bm{\nu}$ acts not on scenarios themselves, but on asymmetry levels, and therefore serves as a preference parameter governing which portions of the expectile profile receive greater emphasis.
This provides a flexible mechanism for encoding different attitudes towards risk and tail sensitivity.

We discuss below some representative choices of the integrating measure $\bm{\nu}$\;\!\!.



\paragraph{\emph{Discrete measures.}} \hspace{-2ex}
Let $n \in \N^{\:\! *}\;\!\!$ and assume that $\bm{\nu}$ is a discrete probability measure with finite support $\Set{\! \alpha_1 , \dots , \alpha_n \!} \subset \mo{[} 0 \:\!,\:\!\! 1 \mc{]}$ and associated weights $p_i \;\!\! \coloneqq \bm{\nu}\:\!\big{(}\! \Set{\! \alpha_{\:\! i} \!} \!\big{)}$, $i = 1 , \dots , n$. In this case, the $\bm{\nu}$\:\!-\:\!integrated expectile~\eqref{eq:integrated-expectile} reduces to a finite convex combination of expectiles: for any $X \in \Lped{p \;\!\!}{\bP}{}$,
\[
R^{\:\! e}_{\bm{\nu}}(X)
=
- \;\!\! \sum_{i \:\! = \:\! 1}^{n} \:\! p_i \;\! e_X(\alpha_{\:\! i}) \:\! .
\]
This construction yields a coherent risk measure on $\Lped{p \;\!\!}{\bP}{}$ provided that $\alpha_{\:\! i} \le 1/2$ for all $i = 1 , \dots , n$.

\paragraph{\emph{Lebesgue measure.}} \hspace{-2ex}
Assume that $\bm{\nu} \equiv \Leb\;\!\!|_{\cB (\;\!\! \mo{[} 0 ,\:\!\! 1 \mc{]} \;\!\!)}$ is the Lebesgue measure on $\mo{[} 0 \:\!,\:\!\! 1 \mc{]}$.
We show that, if $p > 1$, then, for any $X \in \Lped{p \;\!\!}{\bP}{}$,
\begin{equation}
\label{eq:lebesgue-claim} 
\int_{0}^{1} \:\! \abs[\big]{\;\! e_X(\alpha) \:\!} \, d\alpha < \infty \:\! .
\end{equation}
Accordingly, $R^{\:\! e}_{\bm{\nu}}$ is well defined on the whole space $\Lped{p \;\!\!}{\bP}{}$ for any $p > 1$ (see~\eqref{eq:integrated-expectile}) and, as a purely analytic consequence of a change-of-variables
argument for monotone functions, $R^{\:\! e}_{\bm{\nu}}$ admits the following layer-cake (Cavalieri-type) representation:
\[
\int_{0}^{1} \! e_X(\alpha) \, d\alpha
=
- \;\!\! \int_{-\infty}^{0} \;\!\!\! e_X^{-1}(x) \, dx
\:\! + \:\!\!
\int_0^{\infty} \! \big{(} 1 - e_X^{-1}(x) \big{)} \;\! dx \;\! .
\]
Whenever the mapping $e_X$ can be identified with the quantile function of an integrable random variable, this formulation allows for an intuitive interpretation of $R^{\:\! e}_{\bm{\nu}}(X)$ as the mean of a random variable whose distribution is encoded by the expectiles of $X$\;\!\!.
Despite this analogy with the expectation, the induced risk measure $R^{\:\! e}_{\bm{\nu}}$ is, in general, not coherent.

Let us now establish~\eqref{eq:lebesgue-claim} for a fixed $X \in \Lped{p \;\!\!}{\bP}{}$, with $p > 1$. Since the expectile function is nondecreasing and thus any potential lack of integrability may only arise from a (positive) divergence as $\alpha \uparrow 1$, we restrict ourselves to the case
\[
\sup_{\alpha \:\! \in \:\! \mo{]} 0 ,\:\!\! 1 \;\!\!\mc{[}} \;\!\!\! e_X(\alpha) = \infty \:\! ,
\]
where it is enough to prove that there exists $\bar{\alpha} \in \mo{]} 0 \:\!,\:\!\! 1 \;\!\!\mc{[}$ such that
\begin{equation}
\label{eq:lebesgue-claim*} 
\int_{\bar{\alpha}}^{1} \:\! \abs[\big]{\;\! e_X(\alpha) \:\!} \, d\alpha < \infty \:\! .
\end{equation}
In fact, for any $\delta \in \mo{]} 0 \:\!,\:\!\! 1 \;\!\!\mc{[}$, it suffices to choose $\bar{\alpha} \equiv \bar{\alpha}_\delta > 1/2$ such that
\[
e_X(\bar{\alpha})
>
\frac{\bE\:\!\bracks{ X_{\:\! +} }}{1 - \delta} \:\! .
\]
Indeed, in this case, by~\eqref{eq:rem-expectiles_1-ord} and by the monotonicity of $e_X$, for any $\alpha \in \mo{]} \bar{\alpha} \:\!,\:\!\! 1 \;\!\!\mc{[}$ we have
\[
\delta \:\! e_X(\alpha)
<
\frac{2\:\!\alpha - 1}{1 - \alpha} \bE\:\!\bracks[\big]{\;\!\! \big{(} X - e_X(\alpha) \big{)}_{\;\!\! + \:\!\!} } \:\! .
\]
Moreover, since $e_X(\alpha) > 0$,
\[
\bE\:\!\bracks[\big]{\;\!\! \big{(} X - e_X(\alpha) \big{)}_{\;\!\! + \:\!\!} }
= \;\!\!
\int_{e_X(\alpha)}^{\infty} \:\!\!\! x \, dF_{\:\! X}(x)
\le \;\!\!
\int_{e_X(\alpha)}^{\infty} \:\!\!\! \frac{x^{\:\! p}}{{e_X(\alpha)}^{\:\! p - 1}} \, dF_{\:\! X}(x)
\:\! \le \:\!
\frac{\bE\:\!\bracks[\big]{ {\abs{\:\! X \:\!}}^p }}{{e_X(\alpha)}^{\:\! p - 1}} \:\! ,
\]
where $F_{\:\! X}$ denotes the cumulative distribution function of $X$\;\!\!. In conclusion, for any $\alpha \in \mo{]} \bar{\alpha} \:\!,\:\!\! 1 \;\!\!\mc{[}$,
\[
\abs[\big]{\;\! e_X(\alpha) \:\!}
\equiv
e_X(\alpha)
<
\bigg{(} \frac{2\:\!\alpha - 1}{\delta\:\!(1 - \alpha)} \bigg{)}^{\:\!\! 1/p}\:\!\!{ \bE\:\!\bracks[\big]{ {\abs{\:\! X \:\!}}^p } }^{\:\! 1/p} \;\!\! .
\]
Since $p > 1$, the right-hand side is integrable on $\mo{]} \bar{\alpha} \:\!,\:\!\! 1 \;\!\!\mc{[}$, which yields~\eqref{eq:lebesgue-claim*} and completes the argument.

\paragraph{\emph{Truncated Lebesgue measures.}} \hspace{-2ex}
Let $\alpha , \beta \in \mo{[} 0 \:\!,\:\!\! 1 \mc{]}$ with $\alpha < \beta$ and assume that $\bm{\nu}$ is the rescaled Lebesgue measure on $\mo{[} \alpha \:\!, \beta \mc{]}$. As before, in this setting one can write a layer-cake formula of the form
\[
\frac{1}{\beta - \alpha} \;\! \int_{\alpha}^{\beta} \! e_X(t) \, dt
\:\! = \:\!
- \:\! \frac{1}{\beta - \alpha} \:\!\! \int_{-\:\!{e_X(\alpha)}_{-}}^{-\:\!{e_X(\beta)}_{-}} \!\! \big{(}\:\! e_X^{-1}(x) - \alpha \:\!\big{)} \;\! dx
\:\! + \:\!
\frac{1}{\beta - \alpha} \;\!\! \int_{{e_X(\alpha)}_+}^{{e_X(\beta)}_+} \:\!\!\! \big{(}\;\! \beta - e_X^{-1}(x) \big{)} \;\! dx \;\! ,
\]
where the conventions $e_X(0) = -\infty$, $e_X(1) = \infty$ are adopted. In general, this construction induces a risk measure on $\Lped{p \;\!\!}{\bP}{}$ for every $p > 1$. If $\alpha \neq 0$ and $\beta \neq 1$, then it extends to a risk measure on $\Lped{1 \;\!\!}{\bP}{}$. 
Moreover, if $\alpha,\beta \le 1/2$, the associated risk measure $R^{\:\! e}_{\bm{\nu}}$ is coherent.

\bigskip


Further computational aspects, including empirical implementations based on inverse expectile representations, are developed in Appendix~\ref{sec:appendix_empirical-implementation}.


\subsection{Generalized deviation measures}
\label{subsec:deviation-measures}

Following the framework developed in~\cite{rockafellar-uryasev-zabarankin_2006}, a \emph{generalized deviation measure} (on $\Lped{1 \;\!\!}{\bP}{}$) is a functional $\cD \colon \Lped{1 \;\!\!}{\bP}{} \to \R$ that fulfills the three axioms below.
\begin{description} 

\item[\textbf{($\cD$\:\!-\:\!inv)}] \hypertarget{item:Dinv}{\label{item:Dinv} --- \emph{Scale invariance}.
For any $X \in \Lped{1 \;\!\!}{\bP}{}$ and $c \in \R$,
\begin{equation*} 
\label{eq:Dinv}
\cD(X + c) = \cD(X) \:\! .
\end{equation*}
} 

\item[\textbf{($\cD$\:\!-\:\!hom)}] \hypertarget{item:Dhom}{\label{item:Dhom} --- \emph{Positive homogeneity}.
For any $X \in \Lped{1 \;\!\!}{\bP}{}$ and $t \in \mo{[} 0 \:\!,\:\!\! \infty \mc{[}$,
\begin{equation*} 
\label{eq:Dhom}
\cD(t X) = t \:\! \cD(X) \:\! .
\end{equation*}
} 

\item[\textbf{($\cD$\:\!-\:\!pos)}] \hypertarget{item:Dpos}{\label{item:Dpos} --- \emph{Nonnegativity}.
For any $X \in \Lped{1 \;\!\!}{\bP}{}$,
\begin{equation*} 
\label{eq:Dpos}
\cD(X) \ge 0 \;\! .
\end{equation*}
} 

\end{description}

A generalized deviation measure $\cD(\;\! \bm{\cdot} \;\!)$ is \emph{subadditive} if:
\begin{description}
\item[\textbf{($\cD$\:\!-\:\!sub)}] \hypertarget{item:Dsub}{\label{item:Dsub} --- \emph{Subadditivity}.
For any $X , Y \in \Lped{1 \;\!\!}{\bP}{}$,
\begin{equation*} 
\label{eq:Dsub}
\cD(X + Y) \le \cD(X) + \cD(Y) \:\! .
\end{equation*}
} 
\end{description}

For any nontrivial measure $\bm{\mu}$ on $\cB \big{(} \mo{[} 0 \:\!,\:\!\! 1 \mc{]} \big{)}$, we define the \emph{$\bm{\mu}$-integrated expectile deviation} on $\Lped{1 \;\!\!}{\bP}{}$ by
\begin{equation}
\label{eq:integrated-expectile-deviation} 
D^{\:\! e}_{\:\!\! \bm{\mu}}(X)
\eqdef \;\!\!
\int_{0}^{1} \:\! \abs[\big]{\;\! e_X(\alpha) - e_X(1/2\:\!) \:\!} \, d\bm{\mu}(\alpha) \:\! , \quad X \in \Lped{1 \;\!\!}{\bP}{} \:\! 
\end{equation}
(see~\eqref{eq:mean}). Henceforth, we focus on measures $\bm{\mu}$ for which $D^{\:\! e}_{\:\!\! \bm{\mu}}$ is finite on the whole space $\Lped{1 \;\!\!}{\bP}{}$.
Under this assumption,~\eqref{eq:integrated-expectile-deviation} is a (law-invariant) generalized deviation measure.

The issue of subadditivity of $D^{\:\! e}_{\:\!\! \bm{\mu}}$ is more delicate.
While it holds on $\Lped{2}{\bP}{}$ by the results of~\cite{rockafellar-uryasev-zabarankin_2006}, its extension to $\Lped{1 \;\!\!}{\bP}{}$ is not automatic and depends on $\bm{\mu}$.
In particular, whenever $D^{\:\! e}_{\:\!\! \bm{\mu}}$ is sufficiently regular with respect to the $\Lped{1 \;\!\!}{\bP}{}$\:\!-\:\!topology, subadditivity on $\Lped{1 \;\!\!}{\bP}{}$ can be obtained by continuity, by extending it from $\Lped{2}{\bP}{}$ by density. Accordingly, the possibility of preserving this property depends on the behavior of the measure $\bm{\mu}$ near the endpoints of the unit interval.

Heuristically, $D^{\:\! e}_{\:\!\! \bm{\mu}}$ can be interpreted as a measure of dispersion around the central expectile $e_X(1/2\:\!)$.
For each level $\alpha \in \mo{]} 0 \:\!,\:\!\! 1 \;\!\!\mc{[}$, the quantity $\abs{ e_X(\alpha) - e_X(1/2\:\!) }$ captures the magnitude of an asymmetric deviation from the center, indicating how far the distribution extends in the direction corresponding to $\alpha$.
The measure $\bm{\mu}$ then determines how such deviations across different levels are aggregated, assigning relative importance to various parts of the distribution.
In this sense, $D^{\:\! e}_{\:\!\! \bm{\mu}}$ provides a flexible expectile-based dispersion measure, whose sensitivity to tail behavior and asymmetry is governed by the form of $\bm{\mu}$.

In the special case where $\bm{\mu}$ is symmetric about $1/2$, that is, for any $B \in \cB \big{(} \mo{[} 0 \:\!,\:\!\! 1 \mc{]} \big{)}$,
\[
\bm{\mu}\:\!\bracks{ B \:\!}
=
\bm{\mu}\:\!\bracks[\big]{\;\!\! \Set{\! 1 - x : x \in B \!} \;\!\!} \:\! ,
\]
by exploiting the symmetry of $\bm{\mu}$ together with the monotonicity of the expectile function, the $\bm{\mu}$-integrated expectile deviation~\eqref{eq:integrated-expectile-deviation} is entirely determined by the widths of the expectile regions $E_X(\alpha)$ for $\alpha \in \mo{]} 0 \:\!,\:\!\! 1/2 \mc{]}$ (see~\eqref{eq:expectiles-regions}): for any $X \in \Lped{1 \;\!\!}{\bP}{}$,
\[
D^{\:\! e}_{\:\!\! \bm{\mu}}(X)
= \;\!\!
\int_{0}^{1/2} \:\! \bracks[\big]{ e_X(1 - \alpha) - e_X(\alpha) } \;\! d\bm{\mu}(\alpha)
\equiv \;\!\!
\int_{0}^{1/2} \:\!\! \Leb\:\!\bracks[\big]{ E_X(\alpha) } \;\! d\bm{\mu}(\alpha) \:\! .
\]

Relevant instances of $\bm{\mu}$-integrated expectile deviations arise from structured choices of the underlying measure.
Discrete specifications lead to finite combinations of expectile widths, while absolutely continuous ones allow for natural generalizations of classical dispersion measures.
Among these, the expectile analogue of the mean absolute deviation about the median (MAD) plays a prominent role and is discussed below.
Further variants, obtained by weighting the expectile widths by the level $\alpha$, will also be considered, as they naturally connect with the convex body constructions introduced later on.
Suitable bounds required for this purpose are addressed accordingly.

In close analogy with the mean absolute deviation, a scalar measure of dispersion can be obtained by aggregating the widths of the expectile regions across all levels.
Replacing quantiles with expectiles in this construction yields the convex-order-consistent dispersion functional
\begin{equation}
\label{eq:expectile-dispersion} 
\Delta^{\;\!\! e}_{\:\! \bm{\mu}}(X)
\eqdef \;\!\!
\int_{0}^{1/2} \:\!\! \Leb\:\!\bracks[\big]{ E_X(\alpha) } \;\! d\bm{\mu}(\alpha) \:\! , \quad X \in \Lped{1 \;\!\!}{\bP}{} \:\! .
\end{equation}
Beyond its intrinsic interpretation as a measure of variability, this quantity will serve as the fundamental building block for the expectile-based inequality indices introduced in the subsequent sections.

As a reference case, for $\bm{\mu} = \Leb\;\!\!|_{\cB (\;\!\! \mo{[} 0 ,\:\!\! 1 \mc{]} \;\!\!)}$, the functional $\Delta^{\;\!\! e}_{\:\! \bm{\mu}}$ admits an alternative formulation, in terms of the inverse expectile function, as in Section~\ref{subsec:risk-measures}: for any $X \in \Lped{1 \;\!\!}{\bP}{}$,
\begin{equation}
\label{eq:expectile-dispersion-inverse} 
\Delta^{\;\!\! e}_{\:\! \bm{\mu}}(X)
= \;\!\!
\int_{-\infty}^{\bE\:\!\bracks{ X }} \;\!\!\! e_X^{-1}(x) \, dx
\:\! + \:\!\!
\int_{\:\! \bE\:\!\bracks{ X }}^{\infty} \! \big{(} 1 - e_X^{-1}(x) \big{)} \;\! dx \:\! .
\end{equation}
In particular, the quantity in~\eqref{eq:expectile-dispersion-inverse} can be interpreted as the mean absolute deviation of the distribution induced by $e_X^{-1}$ from the central expectile $e_X(1/2\:\!) = \bE\:\!\bracks{ X }$.
By invoking~\eqref{eq:inverse-expectile}, this representation yields explicit expressions in terms of the tail functionals $\bE\:\!\bracks[\big]{\;\!\! {(X - x)}_{\pm} }$ and of the absolute deviation $\bE\:\!\bracks[\big]{ \abs{\:\! X - x \:\! } }$, $x \in \R$, which will be exploited in the analysis.

More generally, let $\bm{\mu}$ be absolutely continuous with respect to the Lebesgue measure on $\mo{[} 0 \:\!,\:\!\! 1/2 \mc{]}$, with Radon--Nikodym derivative
$\psi \coloneqq \frac{d\bm{\mu}}{d\:\!\!\Leb} \colon \mo{[} 0 \:\!,\:\!\! 1/2 \mc{]} \:\!\! \to \R_{\:\! +}$ (defined $\Leb$-\:\!a.e.).
Then~\eqref{eq:expectile-dispersion} reads
\begin{equation}
\label{eq:inequality-index} 
\Delta^{\;\!\! e}_{\:\! \bm{\mu}}(X)
= \;\!\!
\int_{0}^{1/2} \:\!\! \psi(\alpha) \;\!\! \Leb\:\!\bracks[\big]{ E_X(\alpha) } \;\! d\alpha \:\! , \quad X \in \Lped{1 \;\!\!}{\bP}{} \:\! .
\end{equation}
The dispersion functional \eqref{eq:inequality-index} immediately inherits the upper bounds established in Corollary~\ref{cor:lenght-expectiles-regions_0} and Proposition~\ref{cor:lower-upper-bounds} for the widths of the expectile regions.

\begin{proposition}
\label{prop:kappa_psi} 
Under the standing assumptions of this subsection, suppose further that
\begin{equation}
\label{eq:kappa_psi} 
\kappa_{\;\!\! \psi}
\coloneqq \;\!\!
\int_{0}^{1/2} \:\!\! \psi(\alpha) \;\! \frac{1 - 2\:\!\alpha}{\alpha} \, d\alpha < \infty \:\! .
\end{equation}
Then, for any non-degenerate $X \in \Lped{1 \;\!\!}{\bP}{}$,
\begin{equation}
\label{eq:kappa_psi-ineq} 
\Delta^{\;\!\! e}_{\:\! \bm{\mu}}(X)
\:\! < \:\!
\kappa_{\;\!\! \psi} \bE\:\!\bracks[\Big]{\:\! \abs[\big]{\:\! X \:\!\! - \bE\:\!\bracks{ X } \:\!} \:\!} .
\end{equation}
Moreover, if
\begin{equation}
\label{eq:kappa_psi^+} 
\kappa_{\;\!\! \psi}^{+}
\coloneqq \;\!\!
\int_{0}^{1/2} \:\!\! \psi(\alpha) \;\! \frac{1 - 2\:\!\alpha}{\alpha\:\!(1 - \alpha)} \, d\alpha < \infty \:\! ,
\end{equation}
then, for any $X \in \Lped{1}{\bP}{+} \setminus \Set{\! 0 \!}$,
\begin{equation}
\label{eq:kappa_psi^+-ineq} 
\Delta^{\;\!\! e}_{\:\! \bm{\mu}}(X)
\:\! < \:\!
\kappa_{\;\!\! \psi}^{+} \bE\:\!\bracks{ X } \:\! .
\end{equation}
\end{proposition}


\begin{remark}
It is worth observing that
\[
\kappa_{\;\!\! \psi}
\le
\kappa_{\;\!\! \psi}^{+}
\le
2 \:\! \kappa_{\;\!\! \psi} \:\! ,
\]
so that the finiteness requirements
in~\eqref{eq:kappa_psi} and~\eqref{eq:kappa_psi^+}, imposed on $\psi$ through $\kappa_{\;\!\! \psi}$ and $\kappa_{\;\!\! \psi}^{+}$, are in fact equivalent.
They are nevertheless stated explicitly for ease of reference.

On the other hand, these integrability conditions do not, in general, guarantee that $\psi$ is integrable on its own over $\mo{]} 0 \:\!,\:\!\! 1/2 \mc{]}$ (equivalently, that the associated measure $\bm{\mu}$ is finite on $\mo{[} 0 \:\!,\:\!\! 1/2 \mc{]}$).

\end{remark}

Natural examples are provided by power-type densities of the form
\begin{equation}
\label{eq:power-density} 
\alpha \longmapsto \psi(\alpha) \propto \alpha^{\:\! q} \;\!\! , \quad q \in \mo{]} 0 \:\!,\:\!\! \infty \mc{[} \:\! ,
\end{equation}
along with suitable linear combinations thereof.
Among these alternatives, the triangular weighting density
\[
\psi(\alpha) = \alpha
\]
($q = 1$) occupies a distinguished role: it assigns progressively larger weight to less extreme expectile levels, while still retaining sensitivity to asymmetry across the whole range $\mo{[} 0 \:\!,\:\!\! 1/2 \mc{]}$.

Moreover, it admits particularly transparent geometric interpretations and will serve as the canonical weighting scheme in the sequel, leading to an expectile-based measure for which inequality can be interpreted directly in terms of geometric size.

The triangular weighting scheme also admits explicit bounds relating the associated geometric functional to classical measures of variability.
Besides providing quantitative control, these estimates anticipate the multivariate bounds established later in the paper.

\begin{corollary}
\label{cor:power-density} 
Let $X \in \Lped{1 \;\!\!}{\bP}{}$ non-degenerate. Then
\begin{equation}
\label{eq:area-expectile-starshaped-set_bound} 
\int_{0}^{1/2} \:\!\! \alpha \:\! \bracks[\big]{ e_X(1 - \alpha) - e_X(\alpha) } \;\! d\alpha
\:\! < \:\!
\frac{1}{4} \;\!\! \bE\:\!\bracks[\Big]{\:\! \abs[\big]{\:\! X \:\!\! - \bE\:\!\bracks{ X } \:\!} \:\!} \:\! .
\end{equation}
Furthermore, if $X \in \Lped{1}{\bP}{+} \setminus \Set{\! 0 \!}$, then
\begin{equation}
\label{eq:area-expectile-starshaped-set_nor-bound} 
\int_{0}^{1/2} \:\!\! \alpha \:\! \bracks[\big]{ e_X(1 - \alpha) - e_X(\alpha) } \;\! d\alpha
\:\! < \:\!
( 1 - \ln 2 \;\!)\;\!\!\bE\:\!\bracks{ X }
\approx
0.307 \cdot \bE\:\!\bracks{ X } \:\! . 
\end{equation}
\end{corollary}



For further insight into these estimates, and in particular into the role played by the choice of the centre, we refer again to Remark~\ref{rem:median-vs-mean}.
The sharpness of the constant appearing in~\eqref{eq:area-expectile-starshaped-set_nor-bound} is discussed in Section~\ref{subsec:two-point-expectile}.

\begin{remark}
Whenever $X \in \Lped{2}{\bP}{}$, an immediate application of~\eqref{eq:kappa_psi-ineq} and the classical Cauchy--Schwarz inequality yields the additional estimate
\[
\Delta^{\;\!\! e}_{\:\! \bm{\mu}}(X)
\le
\kappa_{\;\!\! \psi} \sqrt{\:\! \Var \bracks{ X } } \:\! .
\]
\end{remark}


\section{The expectile star-shaped set and its geometric measure}
\label{sec:expectile-starshaped-set}

The expectile function admits a natural geometric interpretation in the plane.
By embedding expectile levels into a two-dimensional radial structure, one obtains a star-shaped set whose shape encodes the dispersion and asymmetry of the underlying distribution.
This construction provides a bridge between the analytical and geometric viewpoints in expectile-based analysis.

Beyond furnishing a visual representation of expectile asymmetry, the resulting geometric object captures global distributional features that are not immediately apparent from individual expectile levels.
As will become clear below, its geometry provides a natural foundation for the development of expectile-based measures of dispersion and inequality.


\subsection{Definition and interpretation}
\label{subsec:definition-interpretation_starshaped}

For any $X \in \Lped{1 \;\!\!}{\bP}{}$, we introduce the \emph{expectile star-shaped set} of $X$ as
\begin{equation}
\label{eq:expectile-starshaped-set} 
\cE_{\:\! X \;\!\!} \eqdef \:\! \cl \:\! \Set{\:\!\! {\bracks{t , x}}^{\top \:\!\!} \! \in \R^{2 \:\!\!} \:\!\ \Big{|} \:\!\ t \in \mo{]} 0 \:\!,\:\!\! 1/2 \mc{]} \:\! , \ t^{\:\! -1 \:\!} x \in E_X(t) \:\!\!} \:\!\! .
\end{equation}

The set $\cE_{\:\! X \;\!\!}$ provides a geometric representation of the expectile structure of $X$\;\!\!. 
Each level $\alpha \in \mo{]} 0 \:\!,\:\!\! 1/2 \mc{]}$ corresponds to the line segment
\[
\Set{\:\!\! {\bracks{t , x}}^{\top \:\!\!} \! \in \R^{2 \:\!\!} \:\!\ \Big{|} \:\!\ t = \alpha \:\! , \ \alpha^{\:\!\! -1 \:\!} x \in E_X(\alpha) \:\!\!} \;\!\! ,
\]
so that $\cE_{\:\! X \;\!\!}$ collects all such scaled intervals and their closure. 
By construction, $\cE_{\:\! X \;\!\!}$ is star-shaped with respect to the origin, in the sense that, for any ${\bracks{t , x}}^{\top \:\!\!} \! \in \cE_{\:\! X \;\!\!}$ and any scalar $\theta \in \mo{[} 0 \:\!,\:\!\! 1 \mc{]}$, the point $\theta \:\! {\bracks{t , x}}^{\top \:\!\!} \!$ also belongs to $\cE_{\:\! X \;\!\!}$. 
Equivalently, $\cE_{\:\! X \;\!\!}$ is invariant under radial contractions about the origin.

Geometrically, the boundary of $\cE_{\:\! X \;\!\!}$ is traced by the curve 
\[
\alpha \longmapsto {\bracks{ \alpha \:\! , \alpha \:\! e_X(\alpha) }}^{\top \:\!\!} \:\!\! , \quad \alpha \in \mo{]} 0 \:\!,\:\!\! 1/2 \mc{]} \:\! ,
\]
together with its counterpart
\[
\alpha \longmapsto {\bracks{ \alpha \:\! , \alpha \:\! e_X(1 - \alpha) }}^{\top \:\!\!} \:\!\! , \quad \alpha \in \mo{]} 0 \:\!,\:\!\! 1/2 \mc{]} \:\! .
\]
Accordingly, the set $\cE_{\:\! X \;\!\!}$ can be viewed as the planar region enclosed between these two branches, encoding in geometric form the dispersion and asymmetry of $X$ through its expectile function.

Rather than analyzing the expectile regions individually, the construction encapsulates the entire expectile profile of $X$ within a single geometric object.
This viewpoint enables probabilistic features such as dispersion and asymmetry to be studied through geometric descriptors including area, symmetry, curvature, and radial behavior, while promoting a genuinely geometric approach to distributional analysis.
In particular, the geometry of $\cE_{\:\! X \;\!\!}$ reflects distributional features directly: symmetric laws tend to generate balanced shapes, whereas increasing asymmetry manifests itself through progressively unbalanced geometries.

In general, these branches are not symmetric with respect to the horizontal axis. 
This asymmetry is a distinctive feature of expectile geometry and stands in contrast to Lorenz-type constructions, where symmetry is inherent. 
It reflects the directional nature of the expectile weighting and provides a finer decomposition of distributional imbalance between the lower and upper tails.

Furthermore, contrary to the Lorenz curve framework, the expectile set $\cE_{\:\! X \;\!\!}$ is, in general, non-convex, although it remains star-shaped.
This feature is reflected in the geometry of its radial boundary and endows the construction with a richer geometric structure, with direct implications for the interpretation of the associated inequality measure presented in Section~\ref{sec:inequality-index}.
In this respect, a characterization of convexity is provided in Appendix~\ref{sec:appendix_convexity} under mild regularity assumptions.


The geometric characterization of the convex order provided by~\eqref{eq:cx-expectiles-regions} extends naturally to the expectile star-shaped sets.
Indeed, by definition~\eqref{eq:expectile-starshaped-set},
\[
X \le_{\;\! \mathrm{cx}} Y
\quad \Longleftrightarrow \quad
\cE_{\:\! X \;\!\!} \subseteq \cE_{\:\! Y \;\!\!} \:\! .
\]

Thus, the convex order admits a geometric interpretation in terms of set inclusion.
Within this framework, $\cE_{\:\! X \;\!\!} \subseteq \cE_{\:\! Y \;\!\!}$ precisely expresses the fact that $Y$ is more dispersed than $X$ while preserving the same mean.
The family of expectile star-shaped sets therefore furnishes a geometric realization of the convex order, extending the inclusion structure of expectile regions to a global geometric representation.

The area of $\cE_{\:\! X \;\!\!}$ in the plane is obtained by integrating, over $\alpha \in \mo{]} 0 \:\!,\:\!\! 1/2 \mc{]}$, 
the width of the corresponding scaled expectile segment, namely
\begin{equation}
\label{eq:area-expectile-starshaped-set} 
\Leb^{2}\bracks[\big]{ \cE_{\:\! X \;\!\!} }
= \;\!\!
\int_{0}^{1/2} \:\!\! \alpha \;\!\! \Leb\:\!\bracks[\big]{ E_X(\alpha) } \;\! d\alpha
\;\! \equiv
\int_{0}^{1/2} \:\!\! \alpha \:\! \bracks[\big]{ e_X(1 - \alpha) - e_X(\alpha) } \;\! d\alpha \:\! .
\end{equation}
An explicit upper bound for $\Leb^{2}\bracks[\big]{ \cE_{\:\! X \;\!\!} }$ is provided in~\eqref{eq:area-expectile-starshaped-set_bound}.
In this sense, the expectile-based dispersion functional introduced in~\eqref{eq:expectile-dispersion}, defined in terms of the unscaled regions $E_X(\alpha)$, is directly connected to the geometric area of the star-shaped set $\cE_{\:\! X \;\!\!}$ once the scaling along the first coordinate is taken into account.

Finally, the radial function associated with $\cE_{\:\! X \;\!\!}$,
\[
\bm{\rho}_{\:\! \cE_{\:\! X \;\!\!}}(1 , x)
\eqdef
\sup \Set{\:\!\! \theta \in \R_{\:\!+} : \theta \:\! {\bracks{1 , x}}^{\top \:\!\!} \! \in \cE_{\:\! X \;\!\!} \:\!\!} \:\!\! , \quad x \in \R
\]
is well defined and, for any $x \in \R$, satisfies
\begin{equation} 
\label{eq:radial-function}
\bm{\rho}_{\:\! \cE_{\:\! X \;\!\!}}(1 , x)
=
\min \braces[\big]{\:\! e_X^{-1}(x) \:\!,\;\!\! 1 - e_X^{-1}(x) \:\!}
\equiv
\begin{cases}
\:\! e_X^{-1}(x) \:\! , & \text{if $x \le \bE\:\!\bracks{ X }$\:\!,} \\[0.5ex]
\:\! 1 - e_X^{-1}(x) \:\! , & \text{if $x \ge \bE\:\!\bracks{ X }$\:\!.}
\end{cases}
\end{equation}
Moreover, $\bm{\rho}_{\:\! \cE_{\:\! X \;\!\!}}(1 , x)$ takes values in $\mo{[} 0 \:\!,\:\!\! 1/2 \mc{]}$ for all $x \in \R$.

The characterization of the radial function $\bm{\rho}_{\:\! \cE_{\:\! X \;\!\!}}$\;\!\! through the inverse expectile function facilitates both the analysis and the computation of $\Leb^{2}\bracks[\big]{ \cE_{\:\! X \;\!\!} }$, since the structure of $e_X^{-1}$ directly encodes the geometry of $\cE_{\:\! X \;\!\!}$ along its radial directions.
The resulting inverse-expectile representation also underlies the empirical implementation described in Appendix~\ref{sec:appendix_empirical-implementation}.
Beyond its computational role, the area of the expectile star-shaped set captures structural information on the distribution of $X$\;\!\!, and may be interpreted as an expectile-based analogue of the classical Gini index.

The following proposition provides an exact expression for the area of $\cE_{\:\! X \;\!\!}$ based on its radial representation through the inverse expectile function.

\begin{proposition}
\label{prop:area-expectile-starshaped-set_rho} 
Let $X \in \Lped{1 \;\!\!}{\bP}{}$. Then
\begin{equation}
\label{eq:area-expectile-starshaped-set_rho} 
\Leb^{2}\bracks[\big]{ \cE_{\:\! X \;\!\!} }
=
\frac{1}{2} \:\! \int\nolimits_{\:\! \R} \bm{\rho}^{\;\! 2}_{\:\! \cE_{\:\! X \;\!\!}}(1 , x) \, dx
=
\frac{1}{2} \, \braces[\Bigg]{\, \int_{-\infty}^{\bE\:\!\bracks{ X }} \;\!\!\! { \big{(} e_X^{-1}(x) \big{)} }^{2} dx
\:\! + \:\!\! \int_{\:\! \bE\:\!\bracks{ X }}^{\infty} \:\!\!\! { \big{(} 1 - e_X^{-1}(x) \big{)} }^{2} dx \:\!} \:\! .
\end{equation}
\end{proposition}

\begin{proof}
By the Fubini\:\!--Tonelli theorem, one obtains
\[
\begin{split}
\int_{0}^{1/2} \:\!\! \alpha \:\! \bracks[\big]{ e_X(1 - \alpha) - e_X(\alpha) } \;\! d\alpha &= \int_{0}^{1/2} \:\!\! \alpha \;\! \bracks[\Bigg]{ \:\! \int\nolimits_{\:\! \R} \1_{\mo{[} e_X(\alpha) , e_X(1 - \alpha) \mc{]}} (x) \, dx } \:\! d\alpha \\
&= \int\nolimits_{\:\! \R} \:\!\ \bracks[\Bigg]{\, \int_{0}^{1/2} \alpha \:\! \1_{\mo{[} e_X(\alpha) , e_X(1 - \alpha) \mc{]}} (x) \, d\alpha \:\! } \:\! dx \\
&= \int\nolimits_{\:\! \R} \:\!\ \bracks[\Bigg]{\:\!\! \int_{0}^{\bm{\rho}_{\:\! \cE_{\:\! X \;\!\!}}(1 , x)} \!\!\! \alpha \, d\alpha \:\!} \:\! dx \\[1.25ex]
&= \frac{1}{2} \:\! \int\nolimits_{\:\! \R} \bm{\rho}^{\;\! 2}_{\:\! \cE_{\:\! X \;\!\!}}(1 , x) \, dx \:\! .
\end{split}
\]
This identity, combined with~\eqref{eq:area-expectile-starshaped-set} and the characterization of the radial function in~\eqref{eq:radial-function},
establishes~\eqref{eq:area-expectile-starshaped-set_rho}. \qedhere
\end{proof}

\begin{remark}
\label{rem:power-density} 
The preceding representation reveals a broader geometric structure behind the power-type weighting schemes introduced in~\eqref{eq:power-density}.
Indeed, for any $X \in \Lped{1 \;\!\!}{\bP}{}$, each choice of the power exponent $q \in \mo{]} 0 \:\!,\:\!\! \infty \mc{[}$ determines a different way of quantifying the geometric size of the expectile star-shaped set $\cE_{\:\! X \;\!\!}$.

More precisely, Proposition~\ref{prop:area-expectile-starshaped-set_rho} admits the following extension:
\begin{equation}
\label{eq:power-density_measure} 
\begin{split}
\int_{0}^{1/2} \:\!\! \alpha^{\:\! q} \:\!\! \Leb\:\!\bracks[\big]{ E_X(\alpha) } \;\! d\alpha &= \frac{1}{q + 1} \:\! \int\nolimits_{\:\! \R} \bm{\rho}^{\;\! q \:\! + 1}_{\:\! \cE_{\:\! X \;\!\!}}(1 , x) \, dx \\[-0.25ex]
&= \frac{1}{q + 1} \;\! \braces[\Bigg]{\, \int_{-\infty}^{\bE\:\!\bracks{ X }} \;\!\!\! { \big{(} e_X^{-1}(x) \big{)} }^{q \:\! + 1} dx
\:\! + \:\!\! \int_{\:\! \bE\:\!\bracks{ X }}^{\infty} \:\!\!\! { \big{(} 1 - e_X^{-1}(x) \big{)} }^{q \:\! + 1} dx \:\!} \:\! .
\end{split}
\end{equation}

Thus, power-type weights give rise to a hierarchy of radial moments associated with $\cE_{\:\! X \;\!\!}$.
The triangular case corresponds to $q = 1$, for which~\eqref{eq:power-density_measure} reduces precisely to~\eqref{eq:area-expectile-starshaped-set_rho}.
For $q \neq 1$, the resulting quantity is no longer an area in the ordinary Lebesgue sense, but rather a generalized geometric size of $\cE_{\:\! X \;\!\!}$ determined by its $(q + 1)$\:\!-\:\!th radial moment.
The triangular weighting scheme thus occupies a distinguished position within the power family.
Finally, the limiting case $q \downarrow 0$ recovers the identity~\eqref{eq:expectile-dispersion-inverse}, thereby connecting this family with the expectile-based functional~\eqref{eq:expectile-dispersion}.

In this sense, the parameter $q > 0$ controls the geometric sensitivity of the corresponding functional: larger values of $q$ place greater emphasis on the most extended portions of the expectile star-shaped set, whereas smaller positive values assign relatively more importance to its thinner peripheral regions.
\end{remark}


An intriguing geometric analogy emerges between the boundary curves of $\cE_{\:\! X \;\!\!}$ and Lorenz-type representations of inequality. 
In particular, the lower convex envelope of the mapping 
\[
\alpha \longmapsto (\alpha/2) \;\! e_X(\alpha/2)
\]
may be related to the generalized Lorenz curve of a suitable distribution, whereas the upper concave envelope of 
\[
\alpha \longmapsto (\alpha/2) \;\! e_X(1 - \alpha/2 \:\!)
\]
could be associated with the dual of the generalized Lorenz curve of another distribution. 
This observation suggests that the geometry of $\cE_{\:\! X \;\!\!}$ implicitly encodes a collection of Lorenz-like profiles induced by expectiles, potentially linking the expectile-based area to alternative representations of inequality. 
A systematic investigation of such associated distributions therefore appears promising and may reveal deeper structural connections between expectile geometry and classical measures of economic disparity.

We further refer to Appendix~\ref{sec:appendix_convexity} for a detailed convexity analysis of expectile-based mappings and their relation to distributional asymmetry.



\section{Expectile-based inequality indices}
\label{sec:inequality-index}

From a geometric perspective, the measure of dispersion encoded by expectiles admits a faithful representation through the planar regions introduced in the previous section.
The corresponding area, considered in~\eqref{eq:area-expectile-starshaped-set} and made more explicit in Proposition~\ref{prop:area-expectile-starshaped-set_rho}, condenses this information into a single scalar quantity, capturing the global scale of the expectile spread.
Yet, in its unweighted formulation, this descriptor treats all levels of asymmetry uniformly, without distinguishing their relative position within the distribution.
While this uniformity already yields a meaningful scale of inequality once suitably standardized, it may prove too restrictive whenever one seeks to privilege selected regions, most notably the tails of the distribution, where inequality typically manifests with greater intensity; see also~\cite{bonaccolto.topfer-bonaccolto_2023}.

It is therefore natural to enrich the construction by means of suitable weighting schemes, allowing the contribution of distinct asymmetry regimes to be modulated in a controlled and transparent manner.
This step leads, in a direct and systematic fashion, to the family of expectile-based functionals developed below, in the spirit of the classical approach to inequality measurement as formulated in~\cite{atkinson_1970, shorrocks_1980} and its more recent extensions such as~\cite{inoua_2021, dong-tille-giorgi-guandalini_2024, el.methni-girard-laveur_2025}.


\subsection{Axiomatic framework}
\label{subsec:axiomatic-framework_index}

A \emph{generalized inequality index} associated with $\bP$ and $p \in \mo{[}\;\!\! 1 ,\:\!\! \infty \mc{[}$ is a functional $\cI \colon \Lped{p}{\bP}{+} \to \R$ that satisfies the following three conditions.
\begin{description} 

\item[\textbf{($\cI$\;\!\!-\;\!nor)}] \hypertarget{item:Inor}{\label{item:Inor} --- \emph{Normalization}.
For any $X \in \Lped{p}{\bP}{+}$,
\begin{equation*} 
\cI(X) \in \mo{[} 0 \:\!,\:\!\! 1 \mc{]} \:\! .
\end{equation*}
} 

\item[\textbf{($\cI$\;\!\!-\;\!inv)}] \hypertarget{item:Iinv}{\label{item:Iinv} --- \emph{Scale invariance}.
For any $X \in \Lped{p}{\bP}{+}$ and $t \in \mo{]} 0 \:\!,\:\!\! \infty \mc{[}$,
\begin{equation*} 
\cI(t X) = \cI(X) \:\! .
\end{equation*}
} 


\item[\textbf{($\cI$\;\!\!-\;\!con)}] \hypertarget{item:Icon}{\label{item:Icon} --- \emph{Convex-order consistency}.
For any $X , Y \in \Lped{p}{\bP}{+}$,
\begin{equation*} 
X \le_{\;\!\text{cx}} \;\!\! Y \;\!\!
\quad \Longrightarrow \quad 
\cI(X) \le \cI(Y) \:\! .
\end{equation*}
} 

\end{description}

Normalization fixes the range of the inequality index and endows it with an interpretable scale, with the value $0$ corresponding to a state of perfect equality and the value $1$ to maximal inequality.
Scale invariance (or \emph{income scale independence}) expresses the requirement that the assessment of inequality remain unaffected by uniform rescalings of the underlying quantity.

The role traditionally assigned to the \emph{transfer} or \emph{Pigou–Dalton principle} (PDP), particularly in empirical contexts, is here assumed by consistency with respect to the convex stochastic order.
The Pigou–Dalton condition prescribes a reduction in inequality following a progressive redistribution that preserves the aggregate level of the underlying quantity.
When working with empirical distributions, this requirement is well known to be equivalent to monotonicity with respect to majorization, which in turn coincides with the convex order of the associated empirical distributions.
From this perspective, consistency with the convex order may be regarded as the natural population-level counterpart of the PDP.

A generalized inequality index $\cI(\;\! \bm{\cdot} \;\!)$ is \emph{decomposable} if:
\begin{description} 
\item[\textbf{($\cI$\;\!\!-\;\!dec)}] \hypertarget{item:Idec}{\label{item:Idec} --- \emph{Decomposability}.
There exists a mapping $\ccI \colon \mo{]} 0 \:\!,\:\!\! 1 \;\!\!\mc{[} \times \R_{\:\!+}^{\:\! 2} \;\!\! \times \mo{[} 0 \:\!,\:\!\! 1 \mc{]}^2 \;\!\! \to \mo{[} 0 \:\!,\:\!\! 1 \mc{]}$ such that, for any $X , Y \in \Lped{p}{\bP}{+}$ and $\lambda \in \mo{]} 0 \:\!,\:\!\! 1 \;\!\!\mc{[}$, if $Z \equiv Z_{\;\! X,Y\!,\:\!\lambda} \;\!\! \in \Lped{p}{\bP}{+}$ has cumulative distribution function
\[
F_{\:\! Z} = \lambda \:\! F_{\:\! X} + ( 1 - \lambda \:\!) F_{\;\! Y} ,
\]
then
\begin{equation*} 
\cI(Z)
=
\ccI \big{(} \:\! \lambda \:\!,\;\!\! \bE\:\!\bracks{ X } \:\!,\;\!\! \bE\:\!\bracks{ Y } \:\!,\;\!\! \cI(X) \:\!,\;\!\! \cI(Y) \big{)} \:\! .
\end{equation*}
} 

\end{description}

The decomposability or \emph{aggregativity} axiom formalizes the idea that the inequality of a mixture can be assessed from aggregate information on its constituents, without requiring full knowledge of their underlying distributions.
In the binary case considered here, the resulting value depends solely on the mixing weight, the means of the components, and their respective inequality levels.
This requirement neither presupposes nor entails law invariance, nor does it impose any particular functional form or regularity condition on the aggregation mechanism. 

\begin{remark}
The axioms introduced above do not aim at exhaustively covering all properties commonly required of inequality indices in the literature.
Rather, they are deliberately chosen so as to isolate a minimal yet structurally meaningful framework, suitable for the subsequent developments.

In particular, sample-level conditions such as the \emph{anonymity principle} and the \emph{population principle} are not included explicitly. The anonymity principle requires the value of an inequality index to be invariant under permutations of individual outcomes, whereas the population principle postulates invariance under replications of the population.

Although these principles are classically motivated by empirical approaches to inequality measurement, they do not constitute independent structural requirements when inequality indices are defined and studied at the level of distributions or random variables. For distribution-based indices depending only on the underlying law, both anonymity and population principles are inherently embedded in the framework, since neither permutations nor population replications affect the associated distribution. 
\end{remark}

The geometric perspective developed in the previous sections naturally leads to the study of inequality indices that extend the classical notion of dispersion by incorporating suitable weighting schemes.
Within this framework, we introduce a class of expectile-based inequality functionals that generalize deviation measures through a distribution-sensitive aggregation mechanism (cf. Section~\ref{subsec:deviation-measures}).

More precisely, let $\bm{\mu}$ be a measure on $\mo{[} 0 \:\!,\:\!\! 1/2 \mc{]}$ that is absolutely continuous with respect to the Lebesgue measure, with Radon--Nikodym derivative $\psi = \frac{d\bm{\mu}}{d\:\!\!\Leb} \colon \mo{[} 0 \:\!,\:\!\! 1/2 \mc{]} \:\!\! \to \R_{\:\! +}$, and assume that the integrability condition~\eqref{eq:kappa_psi^+} holds.
Building upon the expectile-dispersion functional~\eqref{eq:expectile-dispersion} and its normalization provided by~\eqref{eq:kappa_psi^+-ineq}, we define the \emph{$\psi$-\:\!integrated expectile inequality index} on $\Lped{p}{\bP}{+}$ by
\begin{equation}
\label{eq:expectile-inequality-index} 
\begin{split}
I^{\:\! e}_{\;\!\! \psi}(X)
\eqdef
\frac{1}{\kappa_{\;\!\! \psi}^{+}} \;\! \Delta^{\;\!\! e}_{\:\! \bm{\mu}} \big{(}\:\! \widetilde{X} \:\!\big{)}
&=
\frac{1}{\kappa_{\;\!\! \psi}^{+}} \:\! \int_{0}^{1/2} \:\!\! \Leb\:\!\bracks[\big]{ E_{\widetilde{X}\:\!}(\alpha) } \;\! d\bm{\mu}(\alpha) \\[0.25ex]
&\equiv
\frac{1}{\kappa_{\;\!\! \psi}^{+}} \:\! \int_{0}^{1/2} \:\!\! \psi(\alpha) \;\!\! \Leb\:\!\bracks[\big]{ E_{\widetilde{X}\:\!}(\alpha) } \;\! d\alpha \:\! , \quad X \in \Lped{p}{\bP}{+} \setminus \Set{\! 0 \!} \:\!\! ,
\end{split}
\end{equation}
where
\begin{equation}
\label{eq:X-std} 
\widetilde{X}
\coloneqq
\frac{X}{\bE\:\!\bracks{ X }} \:\! ,
\end{equation}
the definition being naturally extended to the degenerate case $X \equiv 0$ by setting
\[
I^{\:\! e}_{\;\!\! \psi}(0) \coloneqq 0 \:\! .
\]
Thus, $I^{\:\! e}_{\;\!\! \psi}$ is a law-invariant generalized inequality index, whose convex-order consistency follows from~\eqref{eq:cx-expectiles-regions}.

\begin{remark}
Unlike classical rank-based inequality indices, such as the Gini index and its generalizations, the expectile-based indices introduced here incorporate information on the magnitude of deviations, rather than relying solely on their relative ordering. Rank-based measures are determined entirely by the ordering structure of the distribution and remain invariant under monotone increasing transformations preserving ranks. In contrast, the indices proposed in this paper reflect the metric scale of deviations through the underlying expectile regions and their associated geometric structure.

As a consequence, distributions sharing similar Lorenz curves or rank profiles may nonetheless be distinguished by these indices whenever they exhibit different tail behavior, asymmetry, or localized dispersion. In more homogeneous or symmetric settings, expectile-based indices remain closely aligned with classical measures, while providing a complementary perspective that captures genuinely geometric and magnitude-sensitive aspects of inequality.
\end{remark}

The function $\psi$ acts in~\eqref{eq:expectile-inequality-index} as a weighting scheme that modulates the contribution of the standardized expectile regions $E_{\widetilde{X}\:\!}(\alpha)$ across levels $\alpha$, thereby emphasizing specific portions of the distribution.
The choice of $\psi$ plays a central role in the construction of the index, as it governs the sensitivity of $I^{\:\! e}_{\;\!\! \psi}$ to tail behavior and hence determines how inequality is quantified in expectile-based terms.

As a canonical illustrative instance, consider the triangular weight
\[
\psi(\alpha) = \alpha \:\! , \quad \alpha \in \mo{[} 0 \:\!,\:\!\! 1/2 \mc{]} \:\! .
\]
Under this choice, the resulting functional~\eqref{eq:expectile-inequality-index} progressively emphasizes upper-tail deviations and distributional asymmetry around the mean, moving beyond a purely dispersion-based assessment.

Perhaps more importantly, it admits a direct geometric interpretation in terms of size.
Specifically, the resulting quantity is proportional to the area of the normalized expectile star-shaped set:
\begin{equation}
\label{eq:I_1} 
I^{\:\! e}_{\;\!\! 1}(X)
\coloneqq
\frac{1}{1 - \ln 2} \Leb^{2}\bracks[\big]{ \cE_{\:\! {\widetilde{X}\:\!} \;\!\!} } \;\!\!
\equiv
\frac{1}{1 - \ln 2} \;\! \int_{0}^{1/2} \:\!\! \alpha \:\! \bracks[\big]{ e_{\widetilde{X}\:\!}(1 - \alpha) - e_{\widetilde{X}\:\!}(\alpha) } \;\! d\alpha \:\! , \quad X \in \Lped{p}{\bP}{+} \setminus \Set{\! 0 \!} 
\end{equation}
(see~\eqref{eq:area-expectile-starshaped-set_bound} and cf. Section~\ref{sec:expectile-starshaped-set}).
For nonnegative random variables with unit mean, $I^{\:\! e}_{\;\!\! 1}$ acquires a purely geometric interpretation as a measure of inequality.

Rather than assessing inequality exclusively through ranks, the proposed construction evaluates it through the geometric size of an expectile-generated set, furnishing an expectile-based counterpart to classical Gini-type measures that remains sensitive to the magnitude of distributional disparities.

Alternative weighting profiles allow one to construct heterogeneity functionals that emphasize selected ranges of asymmetry levels, capturing structural features of distributions that are not reducible to dispersion or inequality alone.
From this perspective, varying $\psi$ generates a unified family $\{ I^{\:\! e}_{\;\!\! \psi} \}$ of expectile-based indices, providing a common geometric language for dispersion, inequality, and radial heterogeneity.

These indices rely on global features of the expectile regions associated with the underlying distribution.
In particular, quantities such as $\Leb^{2}\bracks[\big]{ \cE_{\:\! {\widetilde{X}\:\!} \;\!\!} }$ encode nonlocal information on the joint behavior of upper and lower expectiles across the entire range of levels $\alpha$.

As a consequence, the inequality of a mixture cannot, in general, be reconstructed from the corresponding indices of its components through a simple aggregation rule.
Indeed, decomposability is, by design, well adapted to indices that capture dispersion in an additive or averaged sense, but it becomes inherently problematic as soon as one seeks to quantify genuinely shape-driven aspects of heterogeneity.

\begin{remark}
From this perspective, the lack of decomposability should be understood as a structural feature of shape-sensitive indices, rather than as a limitation of the construction.

Only in highly structured or degenerate situations can decomposability be recovered, typically at the price of losing part of the geometric information carried by the expectile regions.
Weaker notions of decomposability, allowing for partial, approximate, or conditional aggregation, may nevertheless be compatible with expectile-based constructions; a thorough examination of these possibilities, however, lies beyond the scope of the present work.

This reflects a fundamental trade-off between geometric sensitivity and aggregative simplicity: the richer the geometric content captured by an inequality index, the less compatible it becomes with strong decomposability requirements.
\end{remark}

A particularly natural subfamily of expectile-based inequality indices is obtained from the power-type weighting profiles introduced in Section~\ref{subsec:deviation-measures}, namely~\eqref{eq:power-density}.
This construction gives rise to a one-parameter collection of indices, parametrized by $q > 0$ and encompassing~\eqref{eq:I_1} as the canonical case $q = 1$:
\begin{equation*}
\begin{split}
I^{\:\! e}_{\;\!\! q}(X)
&\coloneqq
\frac{2^{\:\! q} \;\!\!}{
\frac{2}{q}
-
\Phi \:\!\! \left( \frac{1}{2} , \:\!\! 1 , \:\!\! q \right)
} \;\!
\int_{0}^{1/2} \:\!\! \alpha^{\:\! q} \:\!\! \Leb\:\!\bracks[\big]{ E_{\widetilde{X}\:\!}(\alpha) } \;\! d\alpha \\[0.25ex]
&\equiv
\frac{2^{\:\! q} \;\!\!}{
(q + 1) \bracks*{
\frac{2}{q}
-
\Phi \:\!\! \left( \frac{1}{2} , \:\!\! 1 , \:\!\! q \right)
\;\!\! }
} \:\!
\int\nolimits_{\:\! \R} \bm{\rho}^{\;\! q \:\! + 1}_{\:\! \cE_{\:\! {\widetilde{X}\:\!} \;\!\!}}(1 , x) \, dx \:\! , \quad X \in \Lped{p}{\bP}{+} \setminus \Set{\! 0 \!} \:\!\! ,
\end{split}
\end{equation*}
where $\Phi$ denotes the classical Lerch transcendent.

This family provides a continuum of expectile-based inequality indices with varying degrees of sensitivity to tail asymmetry and distributional heterogeneity, as reflected also by the geometric representation~\eqref{eq:power-density_measure} in Remark~\ref{rem:power-density}.
The use of parametric families of inequality measures is fairly common in the literature; see, for instance,~\cite{el.methni-girard-laveur_2025} and the references therein.


To gain further insight into the proposed construction, we next consider a setting in which expectiles, expectile regions, and the resulting inequality indices admit explicit representations.
This analysis is carried out in the next subsection for a suitable two-point distribution.



\subsection{Expectile geometry for a two-point distribution}
\label{subsec:two-point-expectile}

The index~\eqref{eq:I_1} lends itself to comparison with distributions exhibiting extreme concentration under minimal structural complexity.
Among these, two-point (Bernoulli-type) distributions play a distinguished role, as they constitute the simplest nondegenerate configurations compatible with a prescribed mean, while simultaneously attaining extremal geometric profiles.

The analysis of the expectile geometry associated with such distributions thus provides a canonical benchmark for the geometric framework developed above.
In this setting, expectiles, star-shaped sets, and their induced areas admit fully explicit representations, allowing one to assess sharpness of bounds, asymmetry effects, and normalization procedures in a transparent manner.
As will be seen, the two-point case furnishes not only an extremal reference for this framework, but also a natural calibration scale for interpreting the magnitude of expectile-based inequality indices.

To make these ideas explicit, consider the random variable
\[
X \coloneqq
\begin{cases}
\:\! 0 \:\! , & \text{with probability } 1 - p , \\
\:\! 1/p , & \text{with probability } p ,
\end{cases}
\qquad p \in \mo{]} 0 \:\!,\:\!\! 1 \;\!\!\mc{[},
\]
which is nonnegative and normalized so as to satisfy $\bE\:\!\bracks{ X } = 1$.
By~\eqref{eq:expectiles_1-ord}, for any $\alpha \in \mo{]} 0 \:\!,\:\!\! 1 \;\!\!\mc{[}$, and noting that $e_X(\alpha) \in \mo{]} 0 \:\!,\:\!\! 1/p \mc{[}$, one readily obtains
\begin{equation} 
\label{eq:two-point-expectile}
e_X(\alpha) = \frac{\alpha}{(1 - p) + \alpha \:\! (2\:\!p - 1)} \;\! , \quad e_X(1 - \alpha) = \frac{1 - \alpha}{p - \alpha \:\! (2\:\!p - 1)} \,\! .
\end{equation}

The egalitarian line is given by the mapping $\alpha \mapsto \alpha \bE\:\!\bracks{ X } \equiv \alpha$. For $\alpha \in \mo{]} 0 \:\!,\:\!\! 1/2 \mc{]}$, the vertical gaps between this line and the two branches of the expectile star-shaped set $\cE_{\:\! X \;\!\!}$ (cf. Section~\ref{sec:expectile-starshaped-set}) may then be computed by a straightforward algebraic manipulation of~\eqref{eq:two-point-expectile}, yielding
\[
\alpha - \alpha \:\! e_X(\alpha) = \frac{\alpha \:\! (1 - p)(1 - 2\:\!\alpha)}{(1 - p) + \alpha \:\! (2\:\!p - 1)} \:\! ,
\quad
\alpha \:\! e_X(1 - \alpha) - \alpha
= \frac{\alpha \:\! (1 - p)(1 - 2\:\!\alpha)}{p - \alpha \:\! (2\:\!p - 1)} \:\! .
\]

In the typical case of a right-skewed distribution, namely for $p < 1/2$, one has
\[
p - \alpha \:\! (2\:\!p - 1) < (1 - p) + \alpha \:\! (2\:\!p - 1) \:\! , \quad \alpha \in \mo{]} 0 \:\!,\:\!\! 1/2 \mc{[} \:\! .
\]
It follows that
\[
0 < \alpha - \alpha \:\! e_X(\alpha) < \alpha \:\! e_X(1 - \alpha) - \alpha \:\! , \quad \alpha \in \mo{]} 0 \:\!,\:\!\! 1/2 \mc{[} \:\! .
\]

This example further illustrates the sharpness of Proposition~\ref{prop:lower-upper-bounds}.
Indeed, as $p \downarrow 0$, formula~\eqref{eq:two-point-expectile} implies
\[
e_X(\alpha) \downarrow \frac{\alpha}{1-\alpha} \;\! ,
\quad
e_X(1 - \alpha) \uparrow \frac{1 - \alpha}{\alpha}\;\! ,
\]
thereby saturating the two one-sided estimates in~\eqref{eq:lower-upper-bounds}.

Integrating over $\alpha \in \mo{]} 0 \:\!,\:\!\! 1/2 \mc{]}$, one therefore finds that the area of the lower portion of~$\cE_{\:\! X \;\!\!}$, that is, the region between the egalitarian line and the branch
$\alpha \mapsto \alpha \:\! e_X(\alpha)$, is strictly smaller than the area of its upper portion, lying between the egalitarian line and the branch $\alpha \mapsto \alpha \:\! e_X(1 - \alpha)$.

More precisely, in the regime $p \downarrow 0$, corresponding to an increasingly extreme upper-tail asymmetry, the two areas admit the asymptotic limits
\[
\int_0^{1/2} \:\!\! \bracks[\big]{ \alpha - \alpha \:\! e_X(\alpha) } \:\! d\alpha
\; \longrightarrow \;
\frac{3}{4} - \ln 2 \; \approx \; 0.057 \:\! ,
\qquad
\int_0^{1/2} \:\!\! \bracks[\big]{ \alpha \:\! e_X(1 - \alpha) - \alpha } \:\! d\alpha
\; \longrightarrow \;
\frac{1}{4} \:\! .
\]
Accordingly, the upper area is asymptotically about $4.4$ times larger than the lower one.
This quantifies the extent to which an increasingly concentrated upper tail induces a pronounced, yet still quantitatively controlled, geometric imbalance of~$\cE_{\:\! X \;\!\!}$.

In particular, as $p \downarrow 0$, one finds
\begin{equation}
\label{eq:area-limit} 
\Leb^{2}\bracks[\big]{ \cE_{\:\! X \;\!\!} }
\equiv
\int_{0}^{1/2} \:\!\! \alpha \:\! \bracks[\big]{ e_X(1 - \alpha) - e_X(\alpha) } \;\! d\alpha
\; \longrightarrow \;
1 - \ln 2 \;\! ,
\end{equation}
showing that the constant in~\eqref{eq:area-expectile-starshaped-set_nor-bound} is indeed sharp within the normalized nonnegative class.
This also confirms that the adopted normalization is naturally aligned with the extremal geometry of the model.

Moreover, in the same limit $p \downarrow 0$, one observes that
\[
\bE\:\!\bracks[\Big]{\:\! \abs[\big]{\:\! X \:\!\! - \bE\:\!\bracks{ X } \:\!} \:\!} \uparrow \:\! 2 \;\! .
\]
Hence, the limit in~\eqref{eq:area-limit} is also fully consistent with the upper estimate in~\eqref{eq:area-expectile-starshaped-set_bound} and, taken together, these considerations make clear that the latter cannot be optimal within the normalized nonnegative class.

Finally, the two-point model exhibits a further geometric feature: the convexity of the boundary of the associated star-shaped set.
Indeed, the scaled mapping $\alpha \mapsto \alpha \:\! e_X(\alpha)$, $\alpha \in \mo{]} 0 \:\!,\:\!\! 1 \;\!\!\mc{[}$, has second derivative
\[
\alpha \longmapsto \frac{2 \:\! (1 - p)^2 \:\!\!}{\bracks[\big]{ (1 - p) + \alpha(2\:\!p - 1) }^3 \:\!\!} \:\! , \quad \alpha \in \mo{]} 0 \:\!,\:\!\! 1 \;\!\!\mc{[} \:\! ,
\]
which is strictly positive for every $p \in \mo{]} 0 \:\!,\:\!\! 1 \;\!\!\mc{[}$.

We remark that this two-point specification may be regarded as an extremal configuration within the class of nonnegative random variables with unit mean.
Indeed, despite its minimal support, it induces a markedly asymmetric expectile geometry while preserving convexity.
As such, it serves as a particularly transparent illustration of how a highly concentrated upper-tail mass may shape the overall geometry of the associated expectile star-shaped set.

The explicit two-point characterization of the area functional naturally leads to a calibration procedure.
More precisely, for any $p \in \mo{]} 0 \:\!,\:\!\! 1 \;\!\!\mc{[}$, let
\[
\begin{split}
\cI(p) &\coloneqq \Leb^{2}\bracks[\big]{ \cE_{\:\! X \;\!\!} } \\[1.5ex]
&=
\begin{cases}
\:\! \dfrac{(1 - p) \bracks[\big]{ (2\:\!p - 1)(1 - \ln 2) - p \ln p + (1 - p)\ln\:\!(1 - p) }}{(2\:\!p - 1)^{\:\! 3}\;\!\! } \:\! , & \text{if $p \neq \dfrac{1}{2}$} \:\! , \\[2.25ex]
\:\! \dfrac{1}{12} \:\! , & \text{if $p = \dfrac{1}{2}$} \:\! .
\end{cases}
\end{split}
\]
It follows that the mapping $\cI(\;\! \bm{\cdot} \;\!)$ is strictly decreasing on the whole interval $\mo{]} 0 \:\!,\:\!\! 1 \;\!\!\mc{[}$ and therefore invertible on its image.
This allows one to introduce a normalized percentage or degree of inequality associated with the distribution of~$X$\;\!\!, given by
\[
1 - \cI^{\;\! -1 \;\!\!} \big{(} \Leb^{2}\bracks[\big]{ \cE_{\:\! X \;\!\!} } \:\! \big{)} \:\! .
\]
From a conceptual perspective, this construction may be viewed as a calibration procedure whereby the geometric magnitude $\Leb^{2}\bracks[\big]{ \cE_{\:\! X \;\!\!} }$ is mapped back to a one-dimensional scale of asymmetry, thus yielding an interpretable index that increases as the distribution departs further from symmetry.

\begin{remark}
\label{rem:robin-hood} 

As an expectile-based counterpart of the classical Hoover (or Robin Hood) measure of dispersion, one may contemplate the index
\[
\sup_{\alpha \in \mo{]} 0 \:\!, 1/2 \mc{]}} \:\!\! \alpha \:\! \bracks[\big]{ e_X(1 - \alpha) - e_X(\alpha) } \:\! .
\]
In the particular instance of a two-point distribution, the stationary points of the smooth mapping
\[
\alpha \longmapsto \alpha \:\! \bracks[\big]{ e_X(1 - \alpha) - e_X(\alpha) }
=
\frac{\alpha \:\! (1 - p)(1 - 2\:\!\alpha)}{\bracks[\big]{p - \alpha \:\! (2\:\!p - 1)}\bracks[\big]{(1 - p) + \alpha \:\! (2\:\!p - 1)}} \:\! , \quad \alpha \in \mo{]} 0 \:\!,\:\!\! 1/2 \mc{]} \;\! ,
\]
are identified by the vanishing of the quadratic polynomial
\[
(2\:\!p - 1)^2 \:\! \alpha^2 \:\!\! + 4\:\!p\:\!(1 - p)\alpha - p\:\!(1 - p) \:\! .
\]

In the symmetric configuration $p = 1/2$, the above stationarity condition simplifies considerably, reducing to a linear equation whose unique solution is
\[
\alpha^\ast\;\!\!(1/2\:\!) = 1/4 \:\! .
\]

When $p \neq 1/2$, by contrast, the admissible root takes the explicit form
\[
\alpha^\ast\;\!\!(p)
=
\frac{\sqrt{p\:\!(1 - p)} - 2\:\!p\:\!(1 - p)}{(2\:\!p - 1)^2\:\!\!} \:\! .
\]
This value furnishes the unique maximizer of the index and, for every $p \neq 1/2$, satisfies
\[
\alpha^\ast\;\!\!(p) < 1/4 \:\! .
\]

Consequently, the location of the maximizing expectile level itself becomes informative, providing an additional descriptor of the asymmetry structure of the distribution.

Finally, one notes that $\alpha^\ast\;\!\!(p) \downarrow 0$ as $p \downarrow 0$, indicating that the maximally asymmetric level is attained at progressively smaller~$\alpha$ as the underlying distribution becomes increasingly skewed.

Further details and complementary arguments are discussed in Appendix~\ref{sec:appendix_maximization}.

\end{remark}







\section{Multivariate generalizations} 
\label{sec:multivariate-generalizations}

We now move to a multivariate framework, where the expectile regions introduced so far are no longer intervals, but higher-dimensional geometric objects.
The focus is not on a further axiomatic abstraction of expectiles, but rather on the dimensional enlargement of the regions they generate, from one-dimensional segments to subsets of $\R^d$\;\!\!, with $d \in \N^{\:\! *}\;\!\!$.

This transition naturally brings star-shaped sets into the picture, offering a geometric language that is particularly well suited to describe dispersion, centrality, and stochastic orderings beyond the scalar case.
Moreover, this geometric viewpoint allows us to extend, in a natural and systematic way, all the bounds derived in the preceding sections to both multivariate regions and their associated star-shaped sets.

The present section is devoted to the development of this perspective. To this end, $\Vol_{\;\! d}$ denotes the $d$-dimensional Lebesgue measure, that is,
\[
\Vol_{\;\! d}\bracks{\, \bm{\cdot} \,} \coloneqq \Leb^{d} \bracks{\, \bm{\cdot} \,} \:\! ,
\]
for every $d \ge 2$.
The Euclidean inner product is written as $\bracket{\, \bm{\cdot} \, , \textbf{-} \,} \equiv \bracket{\, \bm{\cdot} \, , \textbf{-} \,}_{\;\! \R^d \;\!\!}$.
No ambiguity will arise from omitting the explicit reference to~$\R^d \;\!\!$.


\subsection{Multivariate expectile regions and volume bounds} 
\label{subsec:multivariate-expectile-regions}

Unlike the scalar case, the multivariate setting does not admit a canonical expectile construction.
We therefore follow a projection-based approach, reconstructing multivariate information from the family of expectile regions associated with all one-dimensional projections of the random vector.

Let $\bX \in \Lped{1}{\bP}{d}$. For any $\alpha \in \mo{]} 0 \:\!,\:\!\! 1/2 \mc{]}$, we define the \emph{expectile region} of $\bX$ at level $\alpha$ as the set
\begin{equation}
\label{eq:multivariate-expectile-regions_def} 
\begin{split}
E_{\:\!\! \bX}^{\;\! d}\;\!\!(\alpha) &\eqdef \Set{\:\!\! \bz \in \R^{d \:\!\!} \:\!\ \Big{|} \:\!\ \forall \ \bu \in \mathbb{S}^{\:\! d - 1 \:\!\!} , \ \bracket{\bz\:\!,\:\!\!\bu} \in E_{\bracket{\bX\:\!\!,\bu}}\:\!\!(\alpha) \:\!\!} \\[0.5ex]
&\equiv \Set{\:\!\! \bz \in \R^{d \:\!\!} \:\!\ \Big{|} \:\!\ \forall \ \bu \in \mathbb{S}^{\:\! d - 1 \:\!\!} , \ e_{\bracket{\bX\:\!\!,\bu}}\:\!\!(\alpha) \le \bracket{\bz\:\!,\:\!\!\bu} \le e_{\bracket{\bX\:\!\!,\bu}\:\!\!}(1 - \alpha) \:\!\!} \:\!\! ,
\end{split}
\end{equation}
in accordance with~\eqref{eq:expectiles-regions}, where
\[
\mathbb{S}^{\:\! d - 1 \:\!\!} \coloneqq \Set{\:\!\! \bu \in \R^{d \:\!\!} \:\!\ \big{|} \:\!\ \norm{\:\! \bu \:\!} = 1 \:\!\!} 
\]
denotes the unit sphere in $\R^d$\;\!\!. Since $e_{-Y}(\beta) = - \:\! e_Y(1 - \beta\:\!)$ for every $Y \in \Lped{1 \;\!\!}{\bP}{}$ and $\beta \in \mo{]} 0 \:\!,\:\!\! 1 \;\!\!\mc{[}$, the defining condition~\eqref{eq:multivariate-expectile-regions_def} can equivalently be expressed solely in terms of upper half-spaces:
\begin{equation}
\label{eq:multivariate-expectile-regions} 
E_{\:\!\! \bX}^{\;\! d}\;\!\!(\alpha)
\:\! = \:\!\!
\bigcap_{\bu \:\!  \in \;\! \mathbb{S}^{\:\! d - 1 \:\!\!}} \:\!\! \Set{\:\!\! \bz \in \R^{d \:\!\!} \:\!\ \Big{|} \:\!\ \bracket{\bz\:\!,\:\!\!\bu} \le e_{\bracket{\bX\:\!\!,\bu}\:\!\!}(1 - \alpha) \:\!\!} \:\!\! .
\end{equation}

For any $\alpha \in \mo{]} 0 \:\!,\:\!\! 1/2 \mc{]}$, the set $E_{\:\!\! \bX}^{\;\! d}\;\!\!(\alpha)$ is a closed and convex subset of $\R^d$\;\!\!, whose support function is given by the directional expectile
\[
\bu \longmapsto e_{\bracket{\bX\:\!\!,\bu}\:\!\!}(1 - \alpha) \:\! , \quad \bu \in \mathbb{S}^{\:\! d - 1 \:\!\!} ,
\] 
and its interior is nonempty whenever the law of $\bX$ is not concentrated on a proper affine subspace of $\R^d$\;\!\!.

Identity~\eqref{eq:cx-expectiles-regions} immediately yields the following multivariate characterization: for any $\bY \in \Lped{1}{\bP}{d}$,
\begin{equation}
\label{eq:cx-multivariate-expectiles-regions} 
\bX \le_{\;\!\text{lcx}} \;\!\! \bY \;\!\!
\quad \Longleftrightarrow \quad
\forall \ \alpha \in \mo{]} 0 \:\!,\:\!\! 1/2 \mc{]} \:\! , \ E_{\:\!\! \bX}^{\;\! d}\;\!\!(\alpha) \subseteq E_{\bY}^{\:\! d}(\alpha) \:\! .
\end{equation}

For convenience, we recall that the relation $\bX \le_{\;\!\text{lcx}} \;\!\! \bY$ denotes the linear-convex order, that is,
\[
\bE\:\!\bracks[\big]{ (\phi \circ \ell\:\!)(\bX) }
\le
\bE\:\!\bracks[\big]{ (\phi \circ \ell\:\!)(\bY) }
\]
for every linear map $\ell \colon \R^d\:\!\! \to \R$ and every convex and nondecreasing function $\phi \colon \R \to \R$.
Equivalently, all one-dimensional projections $\bracket{\bX,\:\!\!\bu}$ and $\bracket{\bY\:\!\!,\:\!\!\bu}$, $\bu \in \mathbb{S}^{\:\! d - 1 \:\!\!}$, are ordered in the usual convex order.

A first geometric control on multivariate expectile regions follows from their behavior under marginal projections and from the basic containment structure induced along each coordinate direction. This preliminary estimate is provided in the next result.

\begin{proposition}
\label{prop:ball-rectangle} 
Let $\bX \in \Lped{1}{\bP}{d}$. Then, for every $\alpha \in \mo{]} 0 \:\!,\:\!\! 1/2 \mc{]}$,
\begin{equation}
\label{eq:ball-rectangle} 
E_{\:\!\! \bX}^{\;\! d}\;\!\!(\alpha)
\subseteq 
\prod_{i = 1}^d E_{X_{\:\! i}}\:\!\!(\alpha)
\;\! \bigcap \;\!
B\:\!\bigg{(} 0 \;\! , \:\!\! \frac{1 - \alpha}{\alpha} \bE\:\!\bracks[\big]{ \norm{\:\! X \:\!} } \bigg{)} \:\! .
\end{equation}
\end{proposition}

\begin{proof}
First observe that, upon selecting the canonical basis vectors of $\R^d \;\!\!$ in~\eqref{eq:multivariate-expectile-regions_def}, one readily obtains
\[
E_{\:\!\! \bX}^{\;\! d}\;\!\!(\alpha)
\subseteq
\prod_{i = 1}^d E_{X_{\:\! i}}\:\!\!(\alpha) \:\! .
\]

Regarding the spherical constraint, fix $\bu \in \mathbb{S}^{\:\! d - 1 \:\!\!}$. By Corollary~\ref{cor:upper-bound} applied to $\bracket{\bX,\:\!\!\bu}$, together with the Cauchy--Schwarz inequality, we find 
\[
e_{\bracket{\bX\:\!\!,\bu}\:\!\!}(1 - \alpha)
\le 
\frac{1 - \alpha}{\alpha} \bE\:\!\bracks[\Big]{\:\! \abs[\big]{ \bracket{\bX,\:\!\!\bu} } \:\!}
\le
\frac{1 - \alpha}{\alpha} \bE\:\!\bracks[\big]{ \norm{\:\! \bX \:\!} } \:\! .
\]
Combining this estimate with the representation~\eqref{eq:multivariate-expectile-regions}, and recalling that every $\bz \in \R^d \;\!\!$ satisfies
\[
\norm{\:\! \bz \:\!}
= \:\!\!
\sup_{\bu \:\! \in \:\! \mathbb{S}^{\:\! d - 1 \:\!\!}} \;\!\! \bracket{\bz\:\!,\:\!\!\bu} \:\! ,
\]
one directly deduces the ball inclusion in~\eqref{eq:ball-rectangle}. \qedhere 
\end{proof}

As will emerge from the subsequent analysis, the enclosure~\eqref{eq:ball-rectangle} should be regarded merely as a preliminary geometric estimate, to be considerably refined in the developments that follow.



Against this background, we complement the one-dimensional analysis developed in Proposition~\ref{prop:lenght-expectiles-regions}, where deviation bounds for univariate expectile regions were expressed in terms of their length.
Our aim is to extend this geometric control to the multivariate setting by quantifying the deviation of multivariate expectile regions through their $d$-dimensional volume.

To prepare the ground, we recall the definition of the Gamma function on the positive real axis:
\[
\Gamma(z)
\coloneqq \;\!\!
\int_0^\infty \:\!\! \tau^{\;\! z - 1} \:\! e^{-\tau} \:\! d\tau \:\! , \quad z \in \mo{]} 0 \:\!,\:\!\! \infty \mc{[} \:\! .
\]
This function naturally arises in the constants governing geometric inequalities in $\R^d \;\!\!$, and it will play a central role in the volume estimates developed below.

An elementary control on the geometry of multivariate expectile regions stems from the set inclusion furnished by Proposition~\ref{prop:ball-rectangle}, stated below in the nonnegative setting.
These bounds illustrate a recurring theme of the paper: geometric size is quantitatively controlled by classical measures of statistical dispersion.
We also refer to Proposition~\ref{cor:lower-upper-bounds} for comparison.

\begin{proposition}
\label{cor:ball-rectangle} 
Let $\bX \in \Lped{1}{\bP}{d,+ \:\!\!}$. Then, for every $\alpha \in \mo{]} 0 \:\!,\:\!\! 1/2 \mc{]}$,
\begin{equation}
\label{eq:volume-multivariate-expectiles-regions_ball-rectangle} 
\Vol_{\;\! d}\bracks[\big]{ E_{\:\!\! \bX}^{\;\! d}\;\!\!(\alpha) }
\le
{ \bracks[\bigg]{\;\!\! \frac{1 - 2\:\!\alpha}{\alpha\:\!(1 - \alpha)} \;\!\!} }^d \:\!\! \prod_{i = 1}^{d \:\!\!} \bE\:\!\bracks{ X_{\:\! i} } \:\! .
\end{equation}
\end{proposition}

\begin{proof}
Fix $j \in \Set{\! 1, \dots, d \:\!\!}$. Since $X_j \ge 0$ (\:\!$\bP$-\:\!a.s.), Proposition~\ref{prop:lower-upper-bounds} provides the inclusion
\[
E_{X_j}\:\!\!(\alpha) \equiv \mo{[} e_{X_j \:\!\!}(\alpha) \:\!,\;\!\! e_{X_j \:\!\!}(1 - \alpha) \mc{]} \subseteq \mo{\Big[} \frac{\alpha}{1 - \alpha} \bE\:\!\bracks{ X_j } \:\!,\:\!\! \frac{1 - \alpha}{\alpha} \bE\:\!\bracks{ X_j } \mc{\Big]} \:\! .
\]
In this nonnegative framework, Proposition~\ref{prop:ball-rectangle} therefore yields
\[
E_{\:\!\! \bX}^{\;\! d}\;\!\!(\alpha)
\subseteq
\prod_{i = 1}^d \;\! \mo{\Big[} \frac{\alpha}{1 - \alpha} \bE\:\!\bracks{ X_{\:\! i} } \:\!,\:\!\! \frac{1 - \alpha}{\alpha} \bE\:\!\bracks{ X_{\:\! i} } \mc{\Big]} \:\! .
\]
The right-hand side is a hyper-rectangle whose volume equals the product of its side-lengths; this quantity coincides exactly with the expression appearing on the right-hand side of~\eqref{eq:volume-multivariate-expectiles-regions_ball-rectangle}. \qedhere
\end{proof}

The following result may be viewed as the multivariate counterpart of Proposition~\ref{prop:lenght-expectiles-regions}.
Without imposing any sign condition on the components of $\bX$, it establishes an explicit volume estimate for $E_{\:\!\! \bX}^{\;\! d}\;\!\!(\alpha)$ in terms of $d$, $\alpha$, and the $d$-th moment of the deviation of $\bX$ from its mean, reflecting the underlying geometry of multivariate expectile regions.
Further discussion is postponed to Remarks~\ref{rem:volume-multivariate-expectiles-regions_comparison} and~\ref{rem:volume-multivariate-expectiles-regions}.

\begin{proposition}
\label{prop:volume-multivariate-expectiles-regions} 
Let $\bX \in \Lped{1}{\bP}{d}$. Then, for every $\alpha \in \mo{]} 0 \:\!,\:\!\! 1/2 \mc{]}$,
\begin{equation}
\label{eq:volume-multivariate-expectiles-regions} 
\Vol_{\;\! d}\bracks[\big]{ E_{\:\!\! \bX}^{\;\! d}\;\!\!(\alpha) }
\le
\frac{2 \:\! \Gamma^{\;\! d-1}\big{(}\frac{d}{2}\big{)}}{d \:\! (d - 1)^d \:\! \Gamma^{\;\! d}\big{(}\frac{d - 1}{2}\big{)}} \;\! \bigg{(} \frac{1 - 2\:\!\alpha}{\alpha} \bigg{)}^{\:\!\! d \:\!\!} \:\!\! \bE\:\!\bracks[\Big]{\:\! \norm[\big]{\:\! \bX \:\!\! - \bE\:\!\bracks{ \bX } \:\!} \:\!} ^ {d \:\!\!} \:\!\! .
\end{equation}
\end{proposition}

\begin{proof}
Fix $\bu \in \mathbb{S}^{\:\! d - 1 \:\!\!}$.
Since $\bracket{\bX,\:\!\!\bu} - \bE\:\!\bracks[\big]{ \bracket{\bX,\:\!\!\bu} } = \bracket[\big]{\:\!\! \bX \:\!\! - \bE\:\!\bracks{ X } ,\:\!\! \bu }$, Proposition~\ref{prop:lenght-expectiles-regions} immediately provides a one-dimensional control on the width of $E_{\:\!\! \bX}^{\;\! d}\;\!\!(\alpha)$, $\alpha \in \mo{]} 0 \:\!,\:\!\! 1/2 \mc{]}$, along $\bu$.
Specifically,
\begin{equation}
\label{eq:width} 
w_{\:\!\!E_{\:\!\! \bX}^{\;\! d}\;\!\!(\alpha)}\:\!\!(\bu)
\coloneqq
\Leb\:\!\bracks[\big]{ E_{\bracket{\bX\:\!\!,\bu}}\:\!\!(\alpha) } \le \frac{1 - 2\:\!\alpha}{\alpha} \:\! \bE\:\!\bracks[\Big]{\:\! \abs[\big]{\:\! \bracket[\big]{\bX \:\!\! - \bE\:\!\bracks{ X } ,\:\!\! \bu\:\!} \:\!} \:\!} \:\! .
\end{equation}

Averaging~\eqref{eq:width} over the unit sphere and applying the Fubini\:\!--Tonelli theorem yields the following representation of the mean width:
\[
\begin{split}
\bar{w}_{\:\!\!E_{\:\!\! \bX}^{\;\! d}\;\!\!(\alpha)} &\coloneqq \frac{1}{\omega_{d-1}} \:\! \int\nolimits_{\:\! \mathbb{S}^{\:\! d - 1 \:\!\!}} \:\!\! w_{\:\!\!E_{\:\!\! \bX}^{\;\! d}\;\!\!(\alpha)}\:\!\!(\bu) \, d\sigma(\bu) \\[2ex]
&\le \frac{1 - 2\:\!\alpha}{\alpha} \:\! \frac{1}{\omega_{d-1}} \:\! \int\nolimits_{\:\! \mathbb{S}^{\:\! d - 1 \:\!\!}} \:\!\! \bE\:\!\bracks[\Big]{\:\! \abs[\big]{\:\! \bracket[\big]{\bX \:\!\! - \bE\:\!\bracks{ X } ,\:\!\! \bu\:\!} \:\!} \:\!} \:\! d\sigma(\bu) \\[1.5ex]
&= \frac{1 - 2\:\!\alpha}{\alpha} \:\! \bE\:\!\bracks[\Bigg]{ \frac{1}{\omega_{d-1}} \:\! \int\nolimits_{\:\! \mathbb{S}^{\:\! d - 1 \:\!\!}} \:\! \abs[\big]{\:\! \bracket[\big]{\bX \:\!\! - \bE\:\!\bracks{ X } ,\:\!\! \bu\:\!} \:\!} \, d\sigma(\bu) } \:\! ,
\end{split}
\]
where $\omega_{d-1} \coloneqq 2\:\!\pi^{\:\! d/2 \:\!\!}/\Gamma\big{(}\frac{d}{2}\big{)}$ denotes the surface measure of $\mathbb{S}^{\:\! d - 1 \:\!\!}$.
The inner spherical integral is controlled by the classical mean absolute projection inequality, which implies
\[
\bar{w}_{\:\!\!E_{\:\!\! \bX}^{\;\! d}\;\!\!(\alpha)}
\le
\frac{1 - 2\:\!\alpha}{\alpha} \:\! \frac{2 \:\! \Gamma\big{(}\frac{d}{2}\big{)}}{\sqrt{\pi} \:\! (d - 1) \:\! \Gamma\big{(}\frac{d - 1}{2}\big{)}} \:\! \bE\:\!\bracks[\Big]{\:\! \norm[\big]{\:\! \bX \:\!\! - \bE\:\!\bracks{ \bX } \:\!} \:\!} \:\! .
\]

Finally, Urysohn’s inequality relates the $d$-dimensional volume of any convex body to its mean width. Applying it to $E_{\:\!\! \bX}^{\;\! d}\;\!\!(\alpha)$ gives
\[
\Vol_{\;\! d}\bracks[\big]{ E_{\:\!\! \bX}^{\;\! d}\;\!\!(\alpha) }
\le
\frac{2\:\!\pi^{\:\! d/2 \:\!\!}}{d\:\!\Gamma\big{(} \frac{d}{2} \big{)}} \;\! \bigg{(} \frac{\bar{w}_{\:\!\!E_{\:\!\! \bX}^{\;\! d}\;\!\!(\alpha)}}{2} \bigg{)}^{\:\!\! d \:\!\!} \! ,
\]
which concludes the proof. \qedhere
\end{proof}

\begin{remark}
\label{rem:volume-multivariate-expectiles-regions_comparison} 
Although the enclosure bound in~\eqref{eq:volume-multivariate-expectiles-regions_ball-rectangle} is, in principle, far less refined than the moment-based estimate~\eqref{eq:volume-multivariate-expectiles-regions}, it may nonetheless outperform it in concrete instances.
This occurs whenever $\bX$ displays a sufficiently sparse or anisotropic structure, so that the rectangular constraint induced by the marginal expectile regions becomes substantially tighter than the isotropic control provided by the $d$-th absolute central moment.
In this regard, the two bounds are intrinsically complementary: the first is sensitive to the coordinate-wise concentration of $\bX$\:\!\!, whereas the latter reflects its overall geometric dispersion.
\end{remark}

The estimate in Proposition~\ref{prop:volume-multivariate-expectiles-regions} naturally raises the question of how the dimensional constant
\begin{equation}
\label{eq:K_d} 
\mathsf{K}_{\;\! \Gamma \:\!}(d)
\coloneqq
\frac{2 \:\! \Gamma^{\;\! d - 1}\big{(}\frac{d}{2}\big{)}}{d \:\! (d - 1)^d \:\! \Gamma^{\;\! d}\big{(}\frac{d - 1}{2}\big{)}}
\end{equation}
varies as $d$ increases. 
This dependence is not merely technical: it reflects the intrinsic geometric contraction of multivariate expectile regions in high dimensions and isolates the geometric contribution in Proposition~\ref{prop:volume-multivariate-expectiles-regions} from the distributional factor $\bE\:\!\bracks[\big]{\:\! \norm[\big]{\:\! \bX \:\!\! - \bE\:\!\bracks{ \bX } \:\!} \:\!} ^ {d \:\!\!}$\:\!\!.

The following result elucidates this phenomenon; its proof is deferred to Appendix~\ref{sec:appendix_proof-of-K_d}.

\begin{proposition}
\label{prop:K_d} 
Let $\mathsf{K} \colon \mo{]} 1 ,\:\!\! \infty \mc{[} \to \mo{]} 0 \:\!,\:\!\! \infty \mc{[}$ be the mapping defined by
\begin{equation}
\label{eq:K_z} 
\mathsf{K}_{\;\! \Gamma \:\!}(z)
\coloneqq
\frac{2 \:\! \Gamma^{\;\! z - 1}\big{(}\frac{z}{2}\big{)}}{z \:\! (z - 1)^z \:\! \Gamma^{\;\! z}\big{(}\frac{z - 1}{2}\big{)}} \;\! , \quad z \in \mo{]} 1 ,\:\!\! \infty \mc{[} \:\! .
\end{equation}
Then $\mathsf{K}_{\;\! \Gamma}$ strictly decreases in $\mo{[} 2 ,\:\!\! \infty \mc{[}$. Moreover, as $z \to \infty$, it decays to zero at a super-exponential rate; more precisely,
\begin{equation}
\label{eq:lim_infty_K_d} 
\mathsf{K}_{\;\! \Gamma \:\!}(z)
=
\cO\Big{(} e^{\:\! z/2} \;\! z^{\;\! -(z \:\! + \:\! 1/2)\:\!\!} \Big{)} \:\! , \quad z \to \infty \:\! .
\end{equation}
In the opposite direction,
\begin{equation}
\label{eq:lim_1_K_d} 
\mathsf{K}_{\;\! \Gamma \:\!}(z) \uparrow 1 \quad \text{as} \quad z \downarrow 1 \;\! .
\end{equation}
\end{proposition}

Since $\Gamma(1) = 1$ and $\Gamma\big{(}\tfrac{1}{2}\big{)} = \sqrt{\pi}$, the decreasing monotonicity of $\mathsf{K}_{\;\! \Gamma \:\!}$ entails the strict estimate
\[
\mathsf{K}_{\;\! \Gamma \:\!}(d) < \mathsf{K}_{\;\! \Gamma \:\!}(2) = \frac{1}{\pi} \approx 0.318 \:\! , \quad d > 2 \:\! .
\]
Furthermore, the limit~\eqref{eq:lim_1_K_d} reveals that Proposition~\ref{prop:volume-multivariate-expectiles-regions} may be regarded as a multivariate counterpart of the bound~\eqref{eq:lenght-expectiles-regions} established in Proposition~\ref{prop:lenght-expectiles-regions}, corresponding to the choice $m = \bE\:\!\bracks{ X }$.

\begin{remark}
\label{rem:volume-multivariate-expectiles-regions} 
A different route to bounding the volume of $E_{\:\!\! \bX}^{\;\! d}\;\!\!(\alpha)$ proceeds directly from~\eqref{eq:width} by applying the Cauchy--Schwarz inequality, which yields, for every $\bu \in \mathbb{S}^{\:\! d - 1 \:\!\!}$,
\[
\bE\:\!\bracks[\Big]{\:\! \abs[\big]{\:\! \bracket[\big]{\bX \:\!\! - \bE\:\!\bracks{ X } ,\:\!\! \bu\:\!} \:\!} \:\!} \:\! \le \;\! \bE\:\!\bracks[\Big]{\:\! \norm[\big]{\:\! \bX \:\!\! - \bE\:\!\bracks{ \bX } \:\!} \:\!} \:\! .
\]
Replacing the mean width by the maximal width (the diameter) of
$E_{\:\!\! \bX}^{\;\! d}\;\!\!(\alpha)$ and invoking the classical isodiametric inequality, one obtains the bound
\[
\Vol_{\;\! d}\bracks[\big]{ E_{\:\!\! \bX}^{\;\! d}\;\!\!(\alpha) }
\le
\frac{2\:\!\pi^{\:\! d/2 \:\!\!}}{d\:\!\Gamma\big{(} \frac{d}{2} \big{)}} \;\! \bigg{(} \frac{1 - 2\:\!\alpha}{\alpha} \bigg{)}^{\:\!\! d \:\!\!} \:\!\! \bE\:\!\bracks[\Big]{\:\! \norm[\big]{\:\! \bX \:\!\! - \bE\:\!\bracks{ \bX } \:\!} \:\!} ^ {d \:\!\!} \:\!\! .
\]

The constant $\mathsf{K}_{\;\! \Gamma \:\!}(d)$ appearing in
Proposition~\ref{prop:volume-multivariate-expectiles-regions}
(see~\eqref{eq:K_d}), however, is uniformly sharper for every $d\ge 2$ and also asymptotically superior. More precisely,
\begin{equation}
\label{eq:log-convex} 
\mathsf{K}_{\;\! \Gamma \:\!}(d) < \frac{1}{2^{\:\! d - 1 \:\!\!}} \:\! \frac{\pi^{\:\! d/2 \:\!\!}}{d\:\!\Gamma\big{(} \frac{d}{2} \big{)}} \;\! , \quad d \ge 2 \:\! ,
\end{equation}
showing that the mean-width approach systematically improves upon the diameter-based estimate across all dimensions greater than one.

To establish~\eqref{eq:log-convex}, it is convenient to rewrite it in the equivalent form
\begin{equation}
\label{eq:log-convex*} 
\Gamma\big{(} \tfrac{d}{2} \big{)} < \sqrt{\pi} \, \Gamma\big{(} \tfrac{d + 1}{2} \big{)} \:\! , \quad d \ge 2 \:\! ,
\end{equation}
(see also~\eqref{eq:gamma-function}).
For $d=2$ the verification is immediate, since $\Gamma\big{(} \tfrac{3}{2} \big{)} = \frac{\sqrt{\pi}}{2}$, and therefore~\eqref{eq:log-convex*} simply reduces to $2 < \pi$. In dimension $d \ge 3$,~\eqref{eq:log-convex*} follows from the (strict) logarithmic convexity of the Gamma function on $\mo{]} 0 \:\!,\:\!\! \infty \mc{[}$: for any $z \in \mo{]} 0 \:\!,\:\!\! \infty \mc{[}$ and $\lambda \in \mo{]} 0 \:\!,\:\!\! 1 \mc{]}$,
\[
\Gamma(\lambda\:\!z + 1/2 \:\!) \le \Gamma^{\;\!\lambda}(z + 1/2 \:\!) \;\! \Gamma^{\;\!1 - \lambda}(1/2\:\!) \equiv \Gamma^{\;\!\lambda}(z + 1/2 \:\!) \;\! \pi^{(1 - \lambda)/2} < \sqrt{\pi} \, \Gamma^{\;\!\lambda}(z + 1/2 \:\!) \;\! .
\]
Taking $z = \tfrac{d}{2}$ and $\lambda = 1 - 1/d$, and noting that $\Gamma\big{(} \tfrac{d + 1}{2} \big{)} > 1$ as $d \ge 3$, gives
\[
\Gamma\big{(} \tfrac{d}{2} \big{)} < \sqrt{\pi} \, \Gamma^{\;\!1 - 1/d\;\!\!}\big{(} \tfrac{d + 1}{2} \big{)} < \sqrt{\pi} \, \Gamma\big{(} \tfrac{d + 1}{2} \big{)} \:\! ,
\]
which is exactly~\eqref{eq:log-convex*}.

From a geometric perspective, this phenomenon is entirely natural: the diameter captures the maximal extent of a convex body, whereas the mean width corresponds to an average taken over all directions.
Inequality~\eqref{eq:log-convex} provides the analytic counterpart of this observation, and explains why the resulting volume estimate is unavoidably sharper.
\end{remark}

We conclude the present subsection by recording a brief but important observation regarding the possible passage from our scalar notions of risk and generalized deviation measures, introduced in~\eqref{eq:integrated-expectile} and~\eqref{eq:integrated-expectile-deviation} respectively (cf. Section~\ref{sec:risk-and-deviation-measures}), to the multivariate setting.

The key point is that no canonical scalar-valued notion of risk or deviation arises naturally in the multivariate framework, nor does there exist a vector-valued expectile capable of playing a genuinely analogous role while adhering to the spirit of the present approach.
Rather, the relevant objects in this context are inherently set-valued and convey directional and geometric information that cannot be faithfully reduced to a single number without incurring a substantial loss of structure.

Accordingly, multivariate expectile regions~\eqref{eq:multivariate-expectile-regions_def} themselves should be considered the primary carriers of risk and dispersion information.
This insight applies with particular force to deviation, as direct extensions of univariate integrated expectile deviations do not admit meaningful multivariate representations but instead lead to fundamentally different, set-valued objects.

Scalar quantities obtained by integration, such as volume-based functionals or related geometric summaries, can nevertheless be defined in a coherent and systematic manner.
These should therefore be interpreted as deviation \emph{indices} derived from the underlying geometric objects, rather than as primitive deviation measures; see Section~\ref{subsec:multivariate-inequality-indices} below.
In this sense, quantitative comparison and ranking remain possible, while the richer geometric structure intrinsic to multivariate deviation is preserved.


\subsection{Multivariate expectile star-shaped set}
\label{subsec:multivariate-expectile-set}

The analysis of multivariate expectile regions naturally leads to the study of a unified geometric object encoding their behavior across all confidence levels $\alpha \in \mo{]} 0 \:\!,\:\!\! 1/2 \mc{]}$.
In analogy with the univariate case, and consistently with the multivariate deviation bounds provided in Proposition~\ref{prop:volume-multivariate-expectiles-regions}, we introduce a star-shaped set that aggregates the entire family of expectile regions into a single $(d + 1)$\:\!-\:\!dimensional construction.

For any $\bX \in \Lped{1}{\bP}{d}$, we define the \emph{expectile star-shaped set} of $\bX$ as
\begin{equation}
\label{eq:multivariate-expectile-starshaped-set} 
\cE_{\:\!\! \bX}^{\:\! d} \eqdef \:\! \cl \:\! \Set{\:\!\! {\bracks{t , \bx}}^{\top \:\!\!} \! \in \R^{\;\!\! 1 + d \:\!\!} \:\!\ \Big{|} \:\!\ t \in \mo{]} 0 \:\!,\:\!\! 1/2 \mc{]} \:\! , \ \bx \in t \:\! E_{\:\!\! \bX}^{\;\! d}(t) \:\!\!} \:\!\! .
\end{equation}
This set may be interpreted as a Minkowski-type product, in which each section at level $t$ corresponds to a scaled multivariate expectile region $t \:\! E_{\:\!\! \bX}^{\;\! d}(t)$; in this sense, it conveys the evolution of expectile regions as the level parameter varies, and furnishes a compact geometric summary of their shape and dispersion.

Since each region $E_{\:\!\! \bX}^{\;\! d}(t)$ is determined by directional expectiles, the star-shaped set $\cE_{\:\!\! \bX}^{\:\! d}$ simultaneously encodes the behavior of all one-dimensional projections across all levels.
As a consequence, it furnishes a global geometric representation of multivariate dispersion and asymmetry, while naturally extending the intuition developed in the scalar case.

A key structural property, inherited directly from the geometric characterization~\eqref{eq:cx-multivariate-expectiles-regions} of multivariate expectile regions, is the following equivalence: for any $\bY \in \Lped{1}{\bP}{d}$,
\begin{equation*}
\bX \le_{\;\!\text{lcx}} \;\!\! \bY \;\!\!
\quad \Longleftrightarrow \quad
\cE_{\:\!\! \bX}^{\:\! d} \subseteq \cE_{\bY}^{\:\! d} \:\! .
\end{equation*}

Thus, the family ${ \cE_{\:\!\! \bX}^{\:\! d} }$ preserves the lower convex order structure, and each star-shaped set may be regarded as a measurable collection of sections parameterized by $t$. This property ensures that the Fubini\:\!--Tonelli theorem applies directly to the computation of volumes.

Specifically, exploiting the special radial structure of the star-shaped set~\eqref{eq:multivariate-expectile-starshaped-set}, applying the Fubini\:\!--Tonelli theorem, and performing the change of variables $\bx = \alpha \:\! \by$, we obtain
\begin{equation}
\label{eq:volume-multivariate-expectile-starshaped-set_0} 
\Vol_{\;\! d+1} \bracks[\big]{ \cE_{\:\!\! \bX}^{\:\! d} }
= \;\!\!
\int_{0}^{1/2} \:\!\! \Vol_{\;\! d}\bracks[\big]{ \alpha \:\! E_{\:\!\! \bX}^{\;\! d}\;\!\!(\alpha) } \;\! d\alpha
= \;\!\!
\int_{0}^{1/2} \:\!\! \alpha^d \;\!\! \Vol_{\;\! d}\bracks[\big]{ E_{\:\!\! \bX}^{\;\! d}\;\!\!(\alpha) } \;\! d\alpha \;\! .
\end{equation}

Combining identity~\eqref{eq:volume-multivariate-expectile-starshaped-set_0} with Proposition~\ref{prop:volume-multivariate-expectiles-regions} yields an explicit upper bound on the $(d + 1)$\:\!-\:\!dimensional volume of the expectile star-shaped set.
The resulting estimate, given below, provides a global deviation control encompassing all levels $\alpha$ simultaneously.

\begin{corollary}
\label{cor:volume-multivariate-expectile-starshaped-set} 
Let $\bX \in \Lped{1}{\bP}{d}$. Then
\begin{equation}
\label{eq:area-expectile-starshaped-set_bound_multivariate} 
\Vol_{\;\! d+1} \bracks[\big]{ \cE_{\:\!\! \bX}^{\:\! d} }
\:\! \le
\frac{\Gamma^{\;\! d-1}\big{(}\frac{d}{2}\big{)}}{d \:\! (d - 1)^{d \:\!\!} \:\! (d + 1) \:\! \Gamma^{\;\! d}\big{(}\frac{d - 1}{2}\big{)}} \;\! \bE\:\!\bracks[\Big]{\:\! \norm[\big]{\:\! \bX \:\!\! - \bE\:\!\bracks{ \bX } \:\!} \:\!} ^ {d \:\!\!} \:\!\! .
\end{equation}
\end{corollary}

To further analyze this estimate, it is natural to isolate the purely dimensional factor
\begin{equation}
\label{eq:tildeK_d} 
\widetilde{\mathsf{K}}_{\;\! \Gamma \:\!}(d)
\coloneqq
\frac{\mathsf{K}_{\;\! \Gamma \:\!}(d)}{2\:\!(d + 1)}
\equiv
\frac{\Gamma^{\;\! d-1}\big{(}\frac{d}{2}\big{)}}{d \:\! (d - 1)^{d \:\!\!} \:\! (d + 1) \:\! \Gamma^{\;\! d}\big{(}\frac{d - 1}{2}\big{)}} 
\end{equation}
(see~\eqref{eq:K_d}).
As in the previous subsection, understanding the behavior of $\widetilde{\mathsf{K}}_{\;\! \Gamma \:\!}(d)$ as the dimension $d$ grows is of intrinsic geometric interest: it quantifies the dimension-driven contraction of multivariate expectile star-shaped regions, independently of the distributional term.
In view of~\eqref{eq:tildeK_d}, the next proposition follows directly from Proposition~\ref{prop:K_d}, and no additional proof is required; see Appendix~\ref{sec:appendix_proof-of-K_d}. 

\begin{proposition}
\label{prop:tildeK_d} 
Let $\widetilde{\mathsf{K}}_{\;\! \Gamma} \colon \mo{]} 1 ,\:\!\! \infty \mc{[} \to \mo{]} 0 \:\!,\:\!\! \infty \mc{[}$ be the mapping defined by
\[
\widetilde{\mathsf{K}}_{\;\! \Gamma \:\!}(z) \coloneqq \frac{\Gamma^{\;\! z - 1}\big{(}\frac{z}{2}\big{)}}{z \:\! (z - 1)^{z \:\!\!} \:\! (z + 1) \:\! \Gamma^{\;\! z}\big{(}\frac{z - 1}{2}\big{)}} \;\! , \quad z \in \mo{]} 1 ,\:\!\! \infty \mc{[} \:\! . 
\]
Then $\widetilde{\mathsf{K}}_{\;\! \Gamma}$ strictly decreases in $\mo{[} 2 ,\:\!\! \infty \mc{[}$. Moreover, as $z \to \infty$, it decays to zero at a super-exponential rate; more precisely,
\begin{equation*}
\widetilde{\mathsf{K}}_{\;\! \Gamma \:\!}(z)
=
\cO\Big{(} e^{\:\! z/2} \;\! z^{\;\! -(z \:\! + \:\!  3/2)\:\!\!} \Big{)} \:\! , \quad z \to \infty \:\! .
\end{equation*}
In the opposite direction,
\begin{equation}
\label{eq:lim_1_tildeK_d} 
\widetilde{\mathsf{K}}_{\;\! \Gamma \:\!}(z) \uparrow \frac{1}{4} \quad \text{as} \quad z \downarrow 1 \;\! .
\end{equation}
\end{proposition}

The strict decrease of $\widetilde{\mathsf{K}}_{\;\! \Gamma \:\!}$ leads to the bound
\[
\widetilde{\mathsf{K}}_{\;\! \Gamma \:\!}(d) < \widetilde{\mathsf{K}}_{\;\! \Gamma \:\!}(2) = \frac{1}{6\:\!\pi} \approx 0.053 \:\! , \quad d > 2 \:\! .
\]
In addition, the limit~\eqref{eq:lim_1_tildeK_d} makes clear that the multiplicative constant $1/4$ occurring in~\eqref{eq:area-expectile-starshaped-set_bound} is recovered as the natural degenerate limit of its multivariate analogue~\eqref{eq:area-expectile-starshaped-set_bound_multivariate}.

In the nonnegative case, Proposition~\ref{prop:ball-rectangle} offers a different dimension-dependent upper bound on the volume of the expectile body $\cE_{\:\!\! \bX}^{\:\! d}$, as shown in the next result.

\begin{corollary}
\label{cor:ball-rectangle_cor} 
Let $\bX \in \Lped{1}{\bP}{d,+ \:\!\!}$. Then
\[
\Vol_{\;\! d+1} \bracks[\big]{ \cE_{\:\!\! \bX}^{\:\! d} }
\:\! \le \:\!
\braces*{\:\! \frac{1}{2} + d\:\!2^{\:\! d - 1} \bracks[\Bigg]{\;\! \sum_{k \:\! = \:\! 1}^d \frac{1}{k\:\!2^{\:\! k \:\!\!} } - \ln 2 \;\!} \:\!}
\prod_{i = 1}^d \bE\:\!\bracks{ X_{\:\! i} } \:\! .
\]
\end{corollary}

\begin{proof}
Combining~\eqref{eq:volume-multivariate-expectile-starshaped-set_0} with~\eqref{eq:volume-multivariate-expectiles-regions_ball-rectangle} yields
\begin{equation}
\label{eq:integral_0} 
\Vol_{\;\! d+1} \bracks[\big]{ \cE_{\:\!\! \bX}^{\:\! d} }
\:\! \le \:\!
\prod_{i = 1}^d \bE\:\!\bracks{ X_{\:\! i} } \:\! \cdot \:\!\! \int_{0}^{1/2} { \bigg{(}\;\!\! \frac{1 - 2\:\!\alpha}{1 - \alpha} \;\!\!} \bigg{)}^{\:\!\! d} \:\!\! d\alpha \:\! .
\end{equation}
We next prove the closed-form expression
\begin{equation}
\label{eq:integral} 
\int_{0}^{1/2} { \bigg{(}\;\!\! \frac{1 - 2\:\!\alpha}{1 - \alpha} \;\!\!} \bigg{)}^{\:\!\! d} \:\!\! d\alpha
=
\frac{1}{2} + d\:\!2^{\:\! d - 1} \bracks[\Bigg]{\;\! \sum_{k \:\! = \:\! 1}^d \frac{1}{k\:\!2^{\:\! k \:\!\!} } - \ln 2 \;\!} \:\! .
\end{equation}
Once the identity~\eqref{eq:integral} is established, the conclusion of the corollary follows immediately from~\eqref{eq:integral_0}.

To verify~\eqref{eq:integral}, observe first that a direct computation gives
\begin{equation}
\label{eq:integral_1} 
\int_{0}^{1/2} { \bigg{(}\;\!\! \frac{1 - 2\:\!\alpha}{1 - \alpha} \;\!\!} \bigg{)}^{\:\!\! d} \:\!\! d\alpha
=
2^{\:\! d - 1 \:\!\!} \:\!\! \int_{0}^{1} \:\!\! { \bigg{(}\;\!\! \frac{s}{1 + s} \;\!\!} \bigg{)}^{\:\!\! d} \:\!\! ds
=
2^{\:\! d - 1 \:\!\!} \:\!\! \int_{0}^{1/2} \:\!\! \frac{t^{\:\! d}}{(1 - t)^{2 \:\!\!}} \, dt \:\! ,
\end{equation}
via the successive substitutions $s = 1 - 2\:\!\alpha$ and $t = s/(1 + s)$. Expanding $(1 - t)^{2 \:\!\!}$ into its Taylor series on $\mo{[} 0 \:\!,\:\!\! 1/2 \mc{]}$ (in contrast with the far less convenient expansion of $(1 + s)^{d \:\!\!}$ on $\mo{[} 0 \:\!,\:\!\! 1 \;\!\!\mc{[}$ in the previous representation), and invoking termwise integration (justified by monotone convergence), we derive
\begin{equation}
\label{eq:integral_2} 
2^{\:\! d - 1 \:\!\!} \:\!\! \int_{0}^{1/2} \:\!\! \frac{t^{\:\! d}}{(1 - t)^{2 \:\!\!}} \, dt
=
\frac{1}{2} \:\! \sum_{n = 1}^\infty \frac{n}{(d + n)\:\!2^{\:\! n \:\!\!}} \:\! .
\end{equation}
Writing $\frac{n}{d + n} = 1 - \frac{d}{d + n}$ (which separates the contribution of $n$ in the numerator from the term $d + n$ in the denominator within the series), and using the geometric series with ratio $1/2$, we obtain
\begin{equation}
\label{eq:integral_3} 
\sum_{n = 1}^\infty \frac{n}{(d + n)\:\!2^{\:\! n \:\!\!}}
=
1 - d \:\! \sum_{n = 1}^\infty \frac{1}{(d + n)\:\!2^{\:\! n \:\!\!}}
\equiv
1 - d\:\!2^{\:\! d \:\!\!} \! \sum_{k = d + 1}^\infty \frac{1}{k\:\!2^{\:\! k \:\!\!}} \:\! .
\end{equation}
Finally, by the well-known Taylor expansion of the function $x \mapsto - \ln(1 - x)$ evaluated at $x = 1/2$, a straightforward rearrangement gives
\begin{equation}
\label{eq:integral_4} 
\sum_{k = d + 1}^\infty \frac{1}{k\:\!2^{\:\! k \:\!\!}}
\equiv
\sum_{k \:\! = \:\! 1}^\infty\frac{1}{k\:\!2^{\:\! k \:\!\!}} - \sum_{k \:\! = \:\! 1}^d \frac{1}{k\:\!2^{\:\! k \:\!\!}}
=
\ln 2 - \sum_{k \:\! = \:\! 1}^d \frac{1}{k\:\!2^{\:\! k \:\!\!}} \:\! .
\end{equation}
Gathering identities~\eqref{eq:integral_1}, \eqref{eq:integral_2}, \eqref{eq:integral_3}, and~\eqref{eq:integral_4}, and substituting into~\eqref{eq:integral}, we reach the desired closed form, thereby establishing the result. \qedhere
\end{proof}

As a particular instance of Corollary~\ref{cor:ball-rectangle_cor}, when $\bX \in \Lped{1}{\bP}{\:\! 2,+}$ (that is, for $d = 2$),
\[
\Vol_{\;\! 3} \bracks[\big]{ \cE_{\:\!\! \bX}^{\:\! 2} }
\:\! \le \:\!
\big{(}\:\! 3 - 4 \ln 2 \:\!\big{)} \bE\:\!\bracks{ X_1 } \:\!\! \bE\:\!\bracks{ X_{\:\! 2} }
\approx
0.227 \cdot \bE\:\!\bracks{ X_1 } \:\!\! \bE\:\!\bracks{ X_{\:\! 2} } \:\! .
\]

Once again, upper bounds formulated in terms of means rather than mean absolute deviations, prove particularly effective whenever the underlying random vector is markedly sparse.

An explicit radial representation of the volume of $\cE_{\bX}^{\, d}$, closely paralleling Proposition~\ref{prop:area-expectile-starshaped-set_rho} in the univariate case, is provided below.

\begin{proposition}
\label{prop:volume-multivariate-expectile-starshaped-set} 
Let $\bX \in \Lped{1}{\bP}{d}$. Then
\begin{equation*}
\Vol_{\;\! d+1} \bracks[\big]{ \cE_{\:\!\! \bX}^{\:\! d} }
=
\frac{1}{d + 1} \:\! \int\nolimits_{\:\! \R^{d \:\!\!}} \bigg{(} \inf_{\bu \:\!  \in \;\! \mathbb{S}^{\:\! d - 1 \:\!\!}} \bm{\rho}_{\:\! \cE_{\:\! \bracket{\bX\:\!\!,\bu}}}\:\!\!\big{(}1 , \bracket{\bx\:\!,\:\!\!\bu}\big{)} \bigg{)}^{\;\!\! d + 1 \:\!\!} \:\!\! d\bx \:\! ,
\end{equation*}
where $\bm{\rho}$ is given in~\eqref{eq:radial-function}.
\end{proposition}

\begin{proof}
From~\eqref{eq:volume-multivariate-expectile-starshaped-set_0} and Fubini\:\!--Tonelli theorem, we have
\[
\begin{split}
\Vol_{\;\! d+1} \bracks[\big]{ \cE_{\:\!\! \bX}^{\:\! d} } &= \int_{0}^{1/2} \:\!\! \alpha^{d \;\!\!} \Vol_{\;\! d}\bracks[\big]{ E_{\:\!\! \bX}^{\;\! d}\;\!\!(\alpha) } \;\! d\alpha \\
&= \int\nolimits_{\:\! \R^{d \:\!\!}} \:\!\ \bracks[\Bigg]{\;\! \int_{0}^{1/2} \:\!\! \alpha^d \1_{E_{\:\!\! \bX}^{\;\! d}\;\!\!(\alpha)} (\bx) \, d\alpha \:\! } \:\! d\bx \\[1ex]
&= \frac{1}{d + 1} \:\! \int\nolimits_{\:\! \R^{d \:\!\!}} \bigg{(}\:\!\! \sup \Set{\:\!\! \alpha \in \mo{]} 0 \:\!,\:\!\! 1/2 \mc{]} \:\!\ \Big{|} \:\!\ \bx \in E_{\:\!\! \bX}^{\;\! d}\;\!\!(\alpha) \!} \:\!\!\bigg{)}^{\;\!\! d + 1 \:\!\!} \:\!\! d\bx \;\! .
\end{split}
\]
Now, using the projection representation of the multivariate expectile regions,
\[
\sup \Set{\:\!\! \alpha \in \mo{]} 0 \:\!,\:\!\! 1/2 \mc{]} \:\!\ \Big{|} \:\!\ \bx \in E_{\:\!\! \bX}^{\;\! d}\;\!\!(\alpha) \!}
\:\! = \:\!\!
\inf_{\bu \:\!  \in \;\! \mathbb{S}^{\:\! d - 1 \:\!\!}} \bm{\rho}_{\:\! \cE_{\:\! \bracket{\bX\:\!\!,\bu}}}\:\!\!\big{(}1 , \bracket{\bx\:\!,\:\!\!\bu}\big{)} \:\! , \quad \bx \in \R^{d \:\!\!} ,
\]
which completes the proof. \qedhere
\end{proof}


\subsection{Multivariate expectile-based inequality indices}
\label{subsec:multivariate-inequality-indices}

The expectile-based approach to inequality measurement developed in Section~\ref{sec:inequality-index}
admits a natural extension to the multivariate framework.
Conceptually, this transition amounts to replacing expectile intervals~\eqref{eq:expectiles-regions} by their multidimensional counterparts~\eqref{eq:multivariate-expectile-regions_def}, reflecting joint deviations from the mean along multiple directions, with geometric size now quantified through volume rather than length.
All axiomatic properties discussed therein extend \emph{mutatis mutandis}, with convex-order consistency interpreted in terms of the linear-convex order $\le_{\;\!\text{lcx}}$ (cf. Section~\ref{subsec:multivariate-expectile-regions}), and the only structural difference being the domain of the functional; therefore, they will not be restated.

To this end, let $\psi \colon \mo{[} 0 \:\!,\:\!\! 1/2 \mc{]} \:\!\! \to \R_{\:\! +}$ be a measurable function and assume that
\begin{equation*}
\kappa_{\psi\:\!\!,d}^{+}
\coloneqq \;\!\!
\int_{0}^{1/2} \:\!\! \psi(\alpha) \:\! \bracks*{\;\!\!
\frac{1 - 2\:\!\alpha}{\alpha\:\!(1 - \alpha)}
\;\!\!}^{d \:\!\!} d\alpha < \infty \:\! .
\end{equation*}
In light of Proposition~\ref{cor:ball-rectangle}, we define the \emph{$\psi$-\:\!integrated expectile inequality index} on $\Lped{p}{\bP}{d,+ \:\!\!}$ by
\begin{equation}
\label{eq:multivariate-inequality-index} 
I^{\:\! e}_{\;\!\! \psi \:\!\! , d}(\bX)
\eqdef
\frac{1}{\kappa_{\psi\:\!\!,d}^{+}} \:\! \int_{0}^{1/2} \:\!\! \psi(\alpha) \;\!\! \Vol_{\;\! d}\bracks[\big]{ E_{\;\!\! \widetilde{\bX}\;\!}^{\;\! d}\;\!\!(\alpha) } \;\! d\alpha \:\! , \quad \bX \in \:\! \Big{(} \Lped{p}{\bP}{+} \setminus \Set{\! 0 \!} \! \Big{)}^{\;\!\! d \:\!\!} ,
\end{equation}
where, for each $i = 1 , \dots , d$,
\[
{ \big{(}\:\! \widetilde{\bX} \:\!\big{)} }_{i}
\;\!\! \coloneqq
\widetilde{X_{\:\! i}}
\]
(as in~\eqref{eq:X-std}) and, for any $\bX^{\:\! 0}\;\!\! \in \Lped{p}{\bP}{d,+ \:\!\!}$ with $X^{\;\! 0}_{\;\!\! j} \equiv 0$ for some $j \in \Set{\! 1, \dots, d \:\!\!}$,
\begin{equation*}
I^{\:\! e}_{\;\!\! \psi \:\!\! , d}\big{(} \bX^{\:\! 0} \:\!\big{)}
\coloneqq
0 \:\! .
\end{equation*}
This latter convention provides a natural extension of the index to boundary points of $\Lped{p}{\bP}{d,+ \:\!\!}$, where the normalization is not defined, and is consistent with the fact that the associated expectile regions have empty interior (and thus zero volume) in $\R^d$\;\!\!.

Several multivariate inequality measures rely on marginal aggregation or partial orders that become difficult to operationalize in high dimensions, as documented for instance in~\cite{marshall-olkin-arnold_2011}.
The expectile-based construction as in~\eqref{eq:multivariate-inequality-index} is instead compatible with the linear-convex order, which is characterized through one-dimensional projections and admits a transparent geometric interpretation in terms of the inclusion of multivariate expectile regions; see~\eqref{eq:cx-multivariate-expectiles-regions}.

The $\psi$-\:\!integrated expectile inequality index may be
interpreted as a scale-free measure of multivariate inequality obtained by aggregating, across expectile levels, the geometric size of suitably normalized expectile regions.
The normalization by the componentwise means removes the effect of global scale, while the weighting function
$\psi$ modulates the relative contribution of lower and upper expectile levels, allowing one to emphasize different portions of the joint distribution.
In this sense, $I^{\:\! e}_{\;\!\! \psi \:\!\! , d}$\;\!\! captures heterogeneity through the integration of directional tail imbalances into a single geometric quantity.

In contrast to marginal or componentwise approaches, the volume of the multivariate expectile region conveys intrinsically joint features of the distribution.
In particular, it accounts for the way tail imbalances interact across dimensions, thus encoding forms of dependence-driven heterogeneity that cannot be recovered from univariate summaries alone.

Even when the marginal distributions of the components are fixed, the geometry of the expectile regions, and hence the value of $I^{\:\! e}_{\;\!\! \psi \:\!\! , d}$, may vary substantially depending on the underlying dependence structure, underscoring the multivariate nature of the proposed index.
More precisely, the index does not merely reflect dependence in a probabilistic sense, but captures how tail deviations align or counterbalance across directions, quantifying a form of coordinated multivariate heterogeneity.
Consequently, two random vectors with identical marginals may still exhibit different inequality levels because the index captures genuinely joint distributional features.
This highlights one of the main advantages of the multivariate framework: the ability to detect forms of heterogeneity beyond the reach of purely univariate analyses.

Multivariate inequality may accordingly be viewed as a form of anisotropy of the standardized expectile regions: configurations that are balanced across directions give rise to nearly isotropic shapes, whereas pronounced joint inequalities manifest themselves through elongated or distorted geometries.

As a distinguished example, in line with the construction outlined in Section~\ref{subsec:multivariate-expectile-set}, consider
\[
\psi(\alpha) = \alpha^d \;\!\! , \quad \alpha \in \mo{[} 0 \:\!,\:\!\! 1/2 \mc{]} \:\! .
\]
Then, using also~\eqref{eq:integral}, the functional~\eqref{eq:multivariate-inequality-index} reduces to the explicit form
\begin{equation}
\label{eq:mult-I_dd} 
I^{\:\! e}_{\;\!\! d \;\!\! , d}(\bX)
\coloneqq
\braces*{\:\! \frac{1}{2} + d\:\!2^{\:\! d - 1} \bracks[\Bigg]{\;\! \sum_{k \:\! = \:\! 1}^d \frac{1}{k\:\!2^{\:\! k \:\!\!} } - \ln 2 \;\!} \:\!}^{\;\!\!\! -1} \:\!\!\! 
\Vol_{\;\! d+1} \bracks[\big]{ \cE_{ \widetilde{\bX}\;\!}^{\:\! d}\;\!\! } \:\! , \quad \bX \in \:\! \Big{(} \Lped{p}{\bP}{+} \setminus \Set{\! 0 \!} \! \Big{)}^{\;\!\! d \:\!\!} .
\end{equation}

The index~\eqref{eq:mult-I_dd}, as a rescaled volume of the standardized expectile star-shaped set, assigns greater weight to central expectile levels, thus privileging the structural shape of the joint distribution over extreme directional contributions.
As a result, it provides a multivariate analogue of classical volume-based measures of dispersion, where heterogeneity is quantified through the overall geometric extent of the expectile star-shaped set rather than through marginal summaries.

It is worth noting that the lack of decomposability identified in the univariate setting becomes even more pronounced in the multivariate case, where joint geometric features cannot, in general, be reconstructed from aggregate information on subpopulations or components.

\begin{remark}

In financial applications, the vector $\bX$ may represent the joint returns or losses of a portfolio of assets.
In this case, the index $I^{\:\! e}_{\;\!\! \psi \:\!\! , d}(\bX)$ quantifies a form of multivariate risk heterogeneity that is shaped not merely by marginal tail features, but also by cross-asset dependence in extreme regimes.
In particular, portfolios sharing identical marginal distributions yet exhibiting distinct patterns of tail co-movement may give rise to markedly different values of the index.

Similarly, in economic settings where $\bX$ gathers multiple dimensions of individual well-being (e.g., income, wealth, and consumption), the proposed index captures forms of multidimensional inequality arising from the joint concentration of disadvantages, which remain hidden to marginal inequality measures. 
Here as well, the geometry of multivariate expectile regions offers a natural framework for representing and quantifying such coordinated heterogeneity.

\end{remark}


\section{Conclusions}
\label{sec:conclusions}

The framework developed in this paper suggests several directions for further research. A first natural line concerns extremal problems for expectiles under distributional uncertainty. Since expectiles can be characterized as minimizers of asymmetric quadratic losses, they fit naturally into variational schemes where the functional value appears as an intermediate optimizer. In this perspective, it is of interest to study worst-case expectiles over suitable uncertainty sets of probability measures and to derive closed-form expressions or dual characterizations of the associated extremal distributions. Recent results on extreme-case distortion-type constructions provide a natural starting point for this analysis; see~\cite{shao-zhang_2024}.

A closely related direction is the robustification of expectile-based constructions in the presence of model ambiguity. One may define set-valued objects by robustifying integrated expectile functionals over uncertainty neighborhoods, aiming at systematic conditions ensuring convexity and closedness, and, when available, dual representability of the resulting mappings. This viewpoint aligns with recent developments on robustified and set-valued risk measures; see~\cite{righi-muller_2025}. Connections with multivariate and set-valued expectile-type constructions are also natural in this context; see~\cite{moresco-mailhot-pesenti_2025}. Moreover, duality-based approaches for set-valued convex risk measures suggest a principled route to representation results and stability properties in vector-valued settings; see~\cite{khanh.linh-hamel_2023}.

More generally, the robustification paradigm can be extended beyond expectiles to encompass convex risk and deviation measures in broad generality. Such robustified constructions may yield a unified theory linking convexity preservation, dual representations, and stability properties, both in scalar and in vector-valued frameworks. In particular, the axiomatic viewpoint on deviation measures and their links with risk-measure constructions provide a natural background for such extensions; see~\cite{rockafellar-uryasev-zabarankin_2006}. This program is also closely related to multivariate risk-measure constructions based on selections and their geometric interpretation; see~\cite{molchanov-cascos_2016}.

Another promising direction concerns the interplay between expectile-based ideas and selection-type approaches for random sets. Extending robustification procedures to selection risk measures, where random closed sets replace random vectors, may provide additional geometric insight and lead to new analytical tools for multivariate uncertainty modeling; see~\cite{molchanov-cascos_2016}. Such a viewpoint may also clarify how expectile-type constructions interact with set operations and selection mechanisms that are specific to random-set frameworks.

From an order-theoretic perspective, it would be of particular interest to investigate stochastic orderings in robustified set-valued settings. Establishing dominance and consistency principles under model ambiguity requires a careful understanding of how stochastic orderings interact with set-valued risk measures and their robustified variants. Recent developments on stochastic orderings for set-valued risk measures suggest a natural route to such a theory; see~\cite{mastrogiacomo-tarsia_2026}. Incorporating expectile-based constructions into this order-theoretic framework may yield comparison principles and monotonicity results that mirror those obtained in the law-invariant scalar setting.

Finally, the integrated expectile-based inequality indices introduced in this paper invite further investigation from the viewpoint of ranking and comparison of heterogeneous populations. Much like classical Lorenz--Gini measures, these indices can be used for inequality comparisons across groups or sectors, while additionally reflecting tail asymmetries and magnitude effects encoded by expectiles. A systematic study of their ordering properties, stability under perturbations, and potential advantages in capturing asymmetric or heavy-tailed phenomena constitutes a natural next step, together with an assessment of their empirical performance in realistic applications.




\appendix


\section{Empirical implementation via inverse expectiles} 
\label{sec:appendix_empirical-implementation}

This appendix develops empirical implementations of the integrated expectile functionals introduced throughout the paper.
The key observation is that, for empirical distributions, inverse expectile functions admit explicit piecewise linear--fractional representations, leading to closed-form expressions for the integrals appearing in the definitions of the associated risk, deviation, and geometric functionals.

Let $n \in \N^{\:\! *}\;\!\!$ and $x_1 < x_2 < \dots < x_n$ be distinct observations on $\R$. Denote by $X_n$ the empirical random variable associated with the finite sample $\mo{(} x_1 , \dots , x_n \mc{)}$ and,
for $k = 1 , \dots , n$, introduce the partial sums $S_k \;\!\! \coloneqq \sum_{\:\! i \:\! = \:\! 1}^{\:\! k} x_{\:\! i}$ ($S_n$ equals $n$ times the empirical mean of $X_n$). Then, according to~\eqref{eq:inverse-expectile}, the empirical inverse expectile function satisfies, for any $k = 1 , \dots , n - 1$,
\[
e_{X_n}^{-1}(x)
\equiv
\frac{\sum_{\:\! i \:\! = \:\! 1}^{\:\! k} \:\! ( x - x_{\:\! i} )}{\sum_{\:\! i \:\! = \:\! 1}^{\:\! n} \:\! \abs{\:\! x_{\:\! i} \;\!\! - x \:\!} }
=
\frac{k\:\!x - S_k}{(2\:\!k - n)\:\!x + S_n \;\!\! - 2\:\!S_k} \;\! ,
\quad
x \in \mo{]}\:\! x_k ,\;\!\! x_{k + 1} \;\!\!\mc{[} \:\! .
\]
The function is extended by $0$ for $x \in \mo{]} -\infty \:\!, x_1 \;\!\!\mc{]}$ and by $1$ for $x \in \mo{[} x_n ,\:\!\! \infty \mc{[}$. Consequently, whenever $2\:\!k \neq n$,
\[
\int_{x_k}^{x} \! e_{X_n}^{-1}(y) \, dy
=
\frac{a_k}{c_k} \;\! x - \frac{a_k\:\!d_k \;\!\! - b_k\:\!c_k}{c_k^{\:\! 2}} \;\! \ln\:\!\abs{\;\! c_k\:\!x + d_k \:\!} + C_k \:\! , \quad x \in \mo{]}\:\! x_k ,\;\!\! x_{k + 1} \;\!\!\mc{[} \:\! ,
\]
with $C_k \in \R$, where on each interval $\mo{]}\:\! x_k ,\;\!\! x_{k + 1} \;\!\!\mc{[}$ we set
\[
a_k \;\!\! \coloneqq k \:\! , \quad
b_k \;\!\! \coloneqq -\:\!S_k \:\! , \quad
c_k \;\!\! \coloneqq 2\:\!k - n \:\! , \quad
d_k \;\!\! \coloneqq S_n \;\!\! - 2\:\!S_k \:\! .
\]
In the remaining special case $2\:\!k = n$, the denominator is constant and $e_{X_n}^{-1}$ reduces to an affine function of $x$, yielding a purely polynomial primitive: for $C_k \in \R$,
\[
\frac{a_k}{2\:\!d_k} \;\! x^2 \;\!\! + \frac{b_k}{d_k} \;\! x + C_k \:\! , \quad
x \in \mo{]}\:\! x_k ,\;\!\! x_{k + 1} \;\!\!\mc{[} \:\! .
\]

The same property holds for $1 - e_{X_n}^{-1}(x)$, which on each interval $\mo{]}\:\! x_k ,\;\!\! x_{k + 1} \;\!\!\mc{[}$ can again be written as a rational function of the form
\[
x \longmapsto \frac{\tilde{a}_k\:\!x + \tilde{b}_k}{c_k\:\!x + d_k} \:\! ,
\]
with explicitly determined coefficients $\tilde{a}_k , \tilde{b}_k$.
Therefore, the empirical counterparts of the integrals appearing in the inverse representation of the dispersion functional reduce to finite sums of logarithmic or polynomial terms, obtained by evaluating these closed-form primitives at the grid points $x_k$.

The evaluation of the empirical dispersion functional requires at most $n - 1$ subintervals, corresponding to the linear--fractional pieces of $e_{X_n}^{-1}$.
Since, on each interval, the primitive is available in closed form and can be computed in constant time once the cumulative sums $S_k$ have been precomputed, the overall computational cost grows linearly with the sample size $n$.

Notice that the identity $2\:\!k = n$, giving rise to polynomial primitives, can occur only for even sample sizes.
For odd sample sizes, the denominator $c_k\:\!x + d_k$ never becomes constant on any subinterval.
This distinction has no practical implications: the empirical implementation remains uniform across all sample sizes, and the linear--fractional structure persists on every piece when $n$ is odd.

Building on the preceding empirical framework, the same piecewise structure also permits an explicit empirical evaluation of the geometric quantities associated with the expectile star-shaped set.
In particular, the representation~\eqref{eq:area-expectile-starshaped-set_rho} provides a direct route to the computation of its area through closed-form primitives.
Via~\eqref{eq:inequality-index}, these formulas also yield an empirical implementation of the associated expectile-based measure of inequality under the triangular weighting scheme.

Indeed, for any $k = 1 , \dots , n - 1$ such that $2\:\!k \neq n$ and for $x \in \mo{]}\:\! x_k ,\;\!\! x_{k + 1} \;\!\!\mc{[}$, one has
\[
\int_{x_k}^{x} \;\!\!\! { \big{(}\:\! e_{X_n}^{-1}(y) \big{)} }^{2} dy
=
\frac{a_k^{\:\! 2}}{c_k^{\:\! 2}}\;\!x - \frac{2\:\!a_k (a_k\:\!d_k \;\!\! - b_k\:\!c_k)}{c_k^{\:\! 3}} \:\! \ln\;\!\abs{\;\! c_k\:\!x + d_k \:\!} - \frac{{(a_k\:\!d_k \;\!\! - b_k\:\!c_k)}^2 \;\!\!}{c_k^3} \:\! \frac{1}{c_k\:\!x + d_k} + C_k \:\! ,
\]
with $C_k \in \R$. When $2\:\!k = n$, this integral instead reduces to
\[
\frac{a_k^{\:\! 2}}{3\:\!d_k^{\:\! 2}}\;\!x^3 \;\!\! + \frac{a_k\:\!b_k}{d_k^{\:\! 2}}\;\!x^2 \;\!\! + \frac{b_k^{\:\! 2}}{d_k^{\:\! 2}}\;\!x + C_k \:\! , \quad x \in \mo{]}\:\! x_k ,\;\!\! x_{k + 1} \;\!\!\mc{[} \:\! , \quad C_k \in \R \:\! .
\]

A completely analogous argument applies to the term involving ${ \big{(}\:\! 1 - e_{X_n}^{-1}(x) \big{)} }^{\;\!\! 2}$\;\!\!;
the same structural considerations remain valid in this case.

Consequently, both the area~\eqref{eq:area-expectile-starshaped-set_rho} and the corresponding inequality functional~\eqref{eq:inequality-index} reduce to the evaluation of closed-form primitives at finitely many breakpoints, yielding explicit empirical implementations.

Overall, the computational complexity remains $\cO(n)$, since the same piecewise linear--fractional structure underlies the entire procedure.


\section{On the convexity of the expectile star-shaped set} 
\label{sec:appendix_convexity}

The geometric profile of the expectile function $e_X$ and of its scaled version $\alpha \mapsto \alpha \:\! e_X(\alpha)$, $\alpha \in \mo{]} 0 \:\!,\:\!\! 1 \;\!\!\mc{[}$, for $X \in \Lped{1 \;\!\!}{\bP}{}$, underlies the interpretation of expectiles as asymmetric measures of location.
In particular, the convexity or concavity of the mapping
\begin{equation}
\label{eq:scaled-expectile} 
\alpha \longmapsto \alpha \:\! e_X(\alpha) \:\! , \quad \alpha \in \mo{]} 0 \:\!,\:\!\! 1/2 \mc{]} \:\! ,
\end{equation}
or, equivalently, of the associated expectile star-shaped set $\cE_{\:\! X \;\!\!}$ (cf. Section~\ref{sec:expectile-starshaped-set}), naturally suggest a geometric perspective closely aligned with Lorenz-type constructions and with the design of inequality-sensitive functionals.
This curvature structure provides an intrinsic explanation of how distributional asymmetry and tail behavior are reflected in the shape of the expectile curve and, consequently, in the associated expectile-based measures of inequality.

To the best of our knowledge, no general results are currently available in the literature characterizing the curvature properties of the transformation~\eqref{eq:scaled-expectile}.
This motivates the derivation of necessary and sufficient conditions for its convexity, which are established in the proposition below.

\begin{proposition}
\label{prop:appendix}
Let $X \in \Lped{1 \;\!\!}{\bP}{}$. Assume that the cumulative distribution function $F_{\:\! X}$ of $X$ is continuously differentiable on $\R$.
Then $e_X$ is twice continuously differentiable on $\mo{]} 0 \:\!,\:\!\! 1 \;\!\!\mc{[}$.
Moreover, the function~\eqref{eq:scaled-expectile} is convex on its domain $\mo{]} 0 \:\!,\:\!\! 1/2 \mc{]}$ if and only if, for every $\alpha \in \mo{]} 0 \:\!,\:\!\! 1/2 \mc{[}$,
\begin{equation}
\label{eq:convexity-derivative} 
\frac{d}{d\alpha} \:\! F_{\:\! X} \big{(} e_X(\alpha) \big{)}
\le
\frac{2 \:\! F_{\:\! X} \big{(} e_X(\alpha) \big{)}}{\alpha \:\! (1 - 2\:\!\alpha)} \:\! .
\end{equation}
\end{proposition}

\begin{proof}
By the classical differentiability results on expectile functionals (see, e.g.,~\cite{newey-powell_1987} and~\cite{bellini-klar-muller-r.gianin_2014}), and under the regularity conditions assumed here, the mapping $e_X$ is continuously differentiable on $\mo{]} 0 \:\!,\:\!\! 1 \;\!\!\mc{[}$ and, for any $\alpha \in \mo{]} 0 \:\!,\:\!\! 1 \;\!\!\mc{[}$, its (positive) derivative in $\alpha$ is given by
\begin{equation}
\label{eq:e'} 
e_{\;\!\! X}^{\:\! \prime}(\alpha)
=
\frac{\bE\:\!\bracks[\Big]{\:\! \abs[\big]{\:\! X - e_X(\alpha) \:\! } \:\!}}{\alpha + (1 - 2\:\!\alpha) \:\! t_X(\alpha)} \:\! ,
\end{equation}
where
\begin{equation}
\label{eq:t_X} 
t_X(\;\! \bm{\cdot} \;\!)
\coloneqq
F_{\:\! X} \big{(} e_X(\;\! \bm{\cdot} \;\!) \big{)}
\end{equation}
(which is itself continuously differentiable on $\mo{]} 0 \:\!,\:\!\! 1 \;\!\!\mc{[}$).

For completeness, we recall the following integral representation of $e_X$ (cf.~equation~(3) in~\cite{cascos-ochoa_2021}):
\[
e_X(\alpha) = \frac{(1 - \alpha)\:\!\!\int_{\:\! 0}^{\;\! t_{\:\! X}(\alpha)} \:\!\! F_{\:\!\! X}^{\;\! -1}(t) \:\! dt + \alpha\;\!\!\int_{\:\! t_{\:\! X}(\alpha)}^1 \:\!\! F_{\:\!\! X}^{\;\! -1}(t) \:\! dt}{\alpha + (1 - 2\:\!\alpha) \:\! t_X(\alpha)} \:\! , \quad \alpha \in \mo{]} 0 \:\!,\:\!\! 1 \;\!\!\mc{[} \:\! ,
\]
where $F_{\:\!\! X}^{\;\! -1}$ denotes the quantile function of $X$\;\!\!.

From~\eqref{eq:e'}, we first observe that
\[
\frac{d}{d\alpha} \bracks[\big]{ \alpha + (1 - 2\:\!\alpha) \:\! t_X(\alpha) }
=
1 - 2 \;\! t_X(\alpha) + (1 - 2\:\!\alpha) \:\! f_{\:\! X} \big{(} e_X(\alpha) \big{)} \:\! e_{\;\!\! X}^{\:\! \prime}(\alpha) \:\! , \quad \alpha \in \mo{]} 0 \:\!,\:\!\! 1 \;\!\!\mc{[} \:\! .
\]
On the other hand, let
\begin{equation}
\label{eq:f_X} 
f_{\:\! X} \;\!\!
\coloneqq
\frac{d\;\!\!\bP_{\:\! X}}{d\:\!\!\Leb}
\equiv
F_{\:\!\! X}^{\;\! \prime}
\end{equation} 
be the density of $X$\;\!\!. Then, for any $\alpha \in \mo{]} 0 \:\!,\:\!\! 1 \;\!\!\mc{[}$, we may write
\[
\bE\:\!\bracks[\Big]{\:\! \abs[\big]{\:\! X - e_X(\alpha) \:\! } \:\!}
= \;\!\!
\int\nolimits_{\:\! \R} \, \abs[\big]{\;\! x - e_X(\alpha) \:\!} \, f_{\:\! X}(x) \, dx
\]
and, by a straightforward application of the dominated convergence theorem, this term is continuously differentiable on $\mo{]} 0 \:\!,\:\!\! 1 \;\!\!\mc{[}$ with
\[
\frac{d}{d\alpha} \bE\:\!\bracks[\Big]{\:\! \abs[\big]{\:\! X - e_X(\alpha) \:\! } \:\!}
=
- \:\! \big{(} 1 - 2 \;\! t_X(\alpha) \big{)} \:\! e_{\;\!\! X}^{\:\! \prime}(\alpha) \:\! , \quad \alpha \in \mo{]} 0 \:\!,\:\!\! 1 \;\!\!\mc{[} \:\! .
\]
Consequently, the asserted regularity of the expectile function is established and, more precisely, a direct differentiation of~\eqref{eq:e'} using the two derivative computations obtained above yields
\begin{equation} 
\label{eq:second-derivative}
e_{\;\!\! X}^{\:\! \prime\prime}(\alpha)
=
- \;\! \frac{e_{\;\!\! X}^{\:\! \prime}(\alpha)}{\alpha + (1 - 2\:\!\alpha) \;\! t_X(\alpha)} \;\! \braces[\Big]{\;\! 2 \:\! \big{(} 1 - 2 \;\! t_X(\alpha) \big{)} + (1 - 2\:\!\alpha) \:\! f_{\:\! X} \big{(} e_X(\alpha) \big{)} \:\! e_{\;\!\! X}^{\:\! \prime}(\alpha) \:\!} \:\! .
\end{equation}

It remains to analyze the convexity of the mapping~\eqref{eq:scaled-expectile}, that is, $\alpha \mapsto \alpha \:\! e_X(\alpha)$, $\alpha \in \mo{]} 0 \:\!,\:\!\! 1/2 \mc{]}$, which is convex if and only if, for every $\alpha \in \mo{]} 0 \:\!,\:\!\! 1/2 \mc{[}$,
\[
2 \:\! e_{\;\!\! X}^{\:\! \prime}(\alpha) + \alpha \:\! e_{\;\!\! X}^{\:\! \prime\prime}(\alpha) \ge 0
\]
(by continuity, this extends to $\alpha = 1/2$).
Substituting~\eqref{eq:second-derivative} into the above inequality and dividing both sides by the positive quantity $e_{\;\!\! X}^{\:\! \prime}(\alpha)$, we deduce the equivalent condition
\[
\frac{1}{2} \;\! \alpha \:\! (1 - 2\:\!\alpha) \:\! f_{\:\! X} \big{(} e_X(\alpha) \big{)} \:\! e_{\;\!\! X}^{\:\! \prime}(\alpha) \le t_X(\alpha) \:\! , \quad \alpha \in \mo{]} 0 \:\!,\:\!\! 1/2 \mc{[} \:\! .
\]
Recalling~\eqref{eq:t_X} and~\eqref{eq:f_X}, this is precisely~\eqref{eq:convexity-derivative}. This completes the proof of the proposition. \qedhere
\end{proof}


The differentiability argument employed above relies on the regularity of the distribution function $F_{\:\! X}$ and on standard smoothness-preserving operations, among which repeated differentiation under the integral sign plays a central role.
In fact, the same line of reasoning may be iterated to infer higher-order regularity of the expectile function.
This yields the following immediate corollary.

\begin{corollary}
Let $X \in \Lped{1 \;\!\!}{\bP}{}$ and let $k \in \N^{\:\! *}\;\!\!$. If the cumulative distribution function $F_{\:\! X}$ is $k$-\:\!times continuously differentiable on $\R$, then $e_X$ is $(k + 1)$\:\!-\:\!times continuously differentiable on $\mo{]} 0 \:\!,\:\!\! 1 \;\!\!\mc{[}$.
\end{corollary}

Whenever $t_X(\;\! \bm{\cdot} \;\!) \;\!\! \equiv F_{\:\! X} \big{(} e_X(\;\! \bm{\cdot} \;\!) \big{)} > 0$, inequality~\eqref{eq:convexity-derivative} is equivalent to
\begin{equation}
\label{eq:convexity-derivative*} 
\frac{d}{d\alpha} \ln \bracks[\big]{ F_{\:\! X} \big{(} e_X(\alpha) \big{)} }
\le
\frac{2}{\alpha \:\! (1 - 2\:\!\alpha)} \:\! ,
\end{equation}
which provides a compact differential formulation of the convexity criterion, linking the rate of change of the cumulative distribution evaluated at the expectile to a purely parametric bound in~$\alpha$.

It is also worth remarking that, near the endpoints of $\mo{]} 0 \:\!,\:\!\! 1/2 \mc{[}$, condition~\eqref{eq:convexity-derivative*} tends to be automatically satisfied, since the right-hand side diverges to $\infty$ as $\alpha \downarrow 0$ or $\alpha \uparrow 1/2$, rendering the inequality non-restrictive in the vicinity of those points.

From a qualitative perspective, the inequality is typically satisfied for symmetric or light-tailed distributions, such as the Gaussian, uniform, or Laplace laws, for which the expectile varies smoothly with~$\alpha$.
By contrast, it may fail for highly skewed or heavy-tailed distributions, for instance Pareto or lognormal laws, where $t_X(\alpha)$ may exhibit excessive sensitivity to small perturbations in~$\alpha$, potentially causing the bound in~\eqref{eq:convexity-derivative*} to fail.




\begin{remark}

The regularity assumptions imposed in Proposition~\ref{prop:appendix} on the cumulative distribution function $F_{\:\! X}$ are mainly adopted for analytical convenience and may be relaxed under more refined hypotheses, such as piecewise or almost-everywhere continuous differentiability, provided that suitable conditions are assumed on the behavior of the expectile function $e_X$ and on the local integrability of the density $f_{\:\! X}$.
In such cases, the differentiability of $e_X$ and the associated convexity criterion~\eqref{eq:convexity-derivative} may still be recovered in a weaker sense.

At the same time, the two-point distribution analyzed in Section~\ref{subsec:two-point-expectile} illustrates that convexity of the scaled expectile mapping~\eqref{eq:scaled-expectile} may occur even when the underlying cumulative distribution function is not globally smooth.
Indeed, despite the lack of global differentiability of $F_{\:\! X}$, the expectile function remains sufficiently regular to yield a convex star-shaped configuration.
This example highlights that convexity of the expectile-based geometry is not intrinsically tied to global smoothness of the cumulative distribution function, but rather to the interplay between the corresponding distribution and its expectiles.

\end{remark}


\section{Maximization of expectile-based asymmetry measures} 
\label{sec:appendix_maximization}

We reconsider the maximization problem examined in Remark~\ref{rem:robin-hood} for two-point distributions, given by
\[
\sup_{\alpha \in \mo{]} 0 \:\!, 1/2 \mc{]}} \:\!\! \alpha \:\! \bracks[\big]{ e_X(1 - \alpha) - e_X(\alpha) } \:\! .
\]
This naturally suggests introducing the auxiliary objective function
\[
\alpha \longmapsto \alpha \:\! \bracks[\big]{ e_X(1 - \alpha) - 1 } \:\! , \quad \alpha \in \mo{]} 0 \:\!,\:\!\! 1/2 \mc{]} \:\! ,
\]
and studying the associated maximization problem, whose objective preserves the same qualitative behavior while being algebraically more tractable.
Invoking identity~\eqref{eq:e'} for $e_{\;\!\! X}^{\:\! \prime}$, the corresponding first-order optimality condition becomes
\[
\alpha \bE \:\! \abs[\big]{\:\! X - e_X(1 - \alpha) \:\!} = \bracks[\big]{ (1 - \alpha) + (2\:\!\alpha - 1) \:\! F_{\:\! X} \big{(} e_X(1 - \alpha) \big{)} } \bracks[\big]{ e_X(1 - \alpha) - 1 } \:\! .
\]

In general, this nonlinear equation does not yield any simple universal bound on the maximizer, such as
\[
\alpha^\ast \:\!\! < 1/4 \:\! .
\]
Even assumptions that might a priori be expected to regularize its behavior, including positivity, unit mean, or symmetry of the underlying law, do not provide such a guarantee. 
Rather, the maximizer may approach $1/2$ arbitrarily closely, or may even fail to be unique, as exemplified by the degenerate case $X \equiv 1$ (\:\!$\bP$-\:\!a.s.).

Under distributional symmetry, the relation
\[
e_X(1 - \alpha) - e_X(\alpha) = 2 \:\! \bracks[\big]{ e_X(1 - \alpha) - 1 }
\]
shows that this behavior directly carries over to the full expectile gap $\alpha \:\! \bracks[\big]{ e_X(1 - \alpha) - e_X(\alpha) }$.

To construct a nontrivial example with
\[
\alpha^\ast \:\!\! > 1/4 \:\! ,
\]
it is natural to consider a three-point distribution of the form
\[
X \coloneqq
\begin{cases}
\:\! 0 \:\! , & \text{with probability } p \:\! , \\
\:\! m \:\! , & \text{with probability } q \:\! , \\
\:\! M , & \text{with probability } 1 - p - q \:\! ,
\end{cases}
\qquad p,q \in \mo{]} 0 \:\!,\:\!\! 1 \;\!\!\mc{[},
\]
with $0 < m < 1 < M$, and $p , q \ll 1$, chosen so that  $\bE\:\!\bracks{ X } \equiv q\:\!m + (1 - p - q)\:\!M = 1$.
In this parameter regime, a large mass is concentrated at the upper value $M$\;\!\!, while a small fraction of probability is split between $0$ and $m$, producing a markedly right-skewed yet nondegenerate configuration in which the maximizer $\alpha^\ast \;\!\!$ may indeed lie strictly above $1/4$.
A direct numerical check would confirm this behavior in concrete instances (e.g., for $p = q \approx 0.05$ and $m \approx 0.1$), although, for clarity of exposition, we refrain from carrying out the full symbolic computation here.

However, the phenomenon discussed above is not confined to the simplest case of a two-point distribution.
Consider, for instance, the uniform model
\[
X \;\!\! \sim \cU \mo{[} 0 \:\!,\:\!\! 2 \mc{]} \:\! ,
\]
so that $X > 0$ and $\bE\:\!\bracks{ X } \equiv e_X(1/2\:\!) = 1$. An elementary computation shows that, for every $\alpha \in \mo{]} 0 \:\!,\:\!\! 1 \;\!\!\mc{[}$,
\[
(1 - 2\:\!\alpha)\:\!e_X^2(\alpha) + 4\:\!\alpha \:\! e_X(\alpha) - 4\:\!\alpha = 0 \:\! .
\]
Hence, for $\alpha \neq 1/2$,
\[
e_X(\alpha) = 2 \;\! \frac{-\:\!\alpha + \sqrt{\:\!\alpha\:\!(1 - \alpha)}}{1 - 2\:\!\alpha} \:\! , \quad
e_X(1 - \alpha) - 1 = \frac{1 - 2\:\!\sqrt{\:\!\alpha\:\!(1 - \alpha)}}{1 - 2\:\!\alpha} \:\! .
\]
To analyze the maximization problem for
\[
\alpha \longmapsto \alpha \:\! \bracks[\big]{ e_X(1 - \alpha) - 1 } \:\! , \quad \alpha \in \mo{]} 0 \:\!,\:\!\! 1/2 \mc{]} \;\! ,
\]
one readily verifies that the sign of its first derivative coincides with the sign of the function
\[
g(\alpha) \coloneqq -\:\!4\:\!\alpha^3\:\!\! + 6\:\!\alpha^2\:\!\! - 3\:\!\alpha + \sqrt{\:\!\alpha\:\!(1 - \alpha)} \:\! .
\]
A direct numerical evaluation gives $g(0.22) \approx 0.00205 > 0$ and $g(0.23) \approx -\:\!0.00044 < 0$, implying that $g$ changes sign strictly before $\alpha = 1/4$ and thus that $\alpha^\ast \:\!\! < 1/4$.

Similarly, for the exponential model
\[
X \;\!\! \sim \cE(1) \:\! ,
\]
one can find that, for any $\alpha \in \mo{]} 0 \:\!,\:\!\! 1 \;\!\!\mc{[}$,
\[
(1 - 2\:\!\alpha)\:\!\exp\:\!\!\big{(}\:\!\! -\:\!\!e_X(\alpha) \big{)} = (1 - \alpha)\:\!(1 - e_X(\alpha)) \:\! ,
\]
or, equivalently,
\[
(1 - e_X(\alpha))\:\!\exp\:\!\!\big{(} e_X(\alpha) \big{)} = \frac{1 - 2\:\!\alpha}{1 - \alpha} \:\! .
\]
Using the Lambert\:\!-$W$ function, defined by
\[
W(z) \exp\:\!\!\big{(} W(z) \big{)} = z \:\! , \quad z \in \R \:\! ,
\]
we obtain the explicit representations
\[
e_X(\alpha) = 1 + W\:\!\big{(}\! -\:\!\!\tfrac{1 \:\! - \:\! 2\:\!\alpha}{e\:\!(1 \:\! - \:\! \alpha)} \big{)} \:\! , \quad
e_X(1 - \alpha) - 1 = W\:\!\big{(} \tfrac{1 \:\! - \:\! 2\:\!\alpha}{e\:\!\alpha} \big{)} \:\! .
\]
Since $W^{\:\! \prime}\;\!\!(z) = \frac{W(z)}{z\:\!\big{(} 1 \:\! + \:\! W(z) \big{)}}$ for $z \neq 0$ ($\alpha \neq 1/2$), and $W(\;\! \bm{\cdot} \;\!) > 0$, a straightforward computation shows that the first derivative of the mapping
\[
\alpha \longmapsto \alpha \:\! \bracks[\big]{ e_X(1 - \alpha) - 1 } \:\! , \quad \alpha \in \mo{]} 0 \:\!,\:\!\! 1/2 \mc{]}
\]
has the same sign as
\[
h(\alpha) \coloneqq (1 - 2\:\!\alpha)\Big{(}1 + W\:\!\big{(} \tfrac{1 \:\! - \:\! 2\:\!\alpha}{e\:\!\alpha} \big{)} \Big{)} - 1 \:\! .
\]
Moreover, $h(0.19) \approx 0.0141 > 0$ and $h(0.20) \approx -\:\!0.0376 < 0$, which reveal that a sign change occurs strictly before $1/4$, and hence $\alpha^\ast \:\!\! < 1/4$ also in the exponential case.

It is worth emphasizing that, unlike the uniform model, the exponential law is highly right-skewed and supported on $\R_{\:\! +}$.
The fact that the maximizer $\alpha^\ast \;\!\!$ still lies below $1/4$ shows that the concentration of the maximizing asymmetry level in the lower part of $\mo{]} 0 \:\!,\:\!\! 1/2 \mc{[}$ is not a mere consequence of symmetry, but persists in genuinely skewed models as well.

The auxiliary formulation considered above relies on the single branch $e_X(1 - \alpha) - 1$, whose analytic structure remains tractable even when no closed-form expression for $e_X$ is available.
By contrast, working directly with the full expectile gap $e_X(1 - \alpha) - e_X(\alpha)$ would require handling two distinct expectile branches.
In models such as the exponential law, these branches involve Lambert\:\!-$W$ evaluations at different arguments, leading to substantially more intricate first-order conditions.
The simplified formulation therefore retains the essential qualitative behavior while remaining considerably more manageable.

In both models, the first-order condition exhibits a single sign change, ensuring existence and uniqueness of the maximizer.
This property relies critically on the smooth and unimodal structure of the underlying density (linear for the uniform distribution and decreasing for the exponential one), which induces a strictly monotone behavior in the associated expectile equation.

By contrast, uniqueness cannot be expected in more irregular settings.
Whenever the distribution exhibits multiple separated modes, pronounced mass imbalances, or abrupt changes in the slope of its distribution function, the mapping $\alpha \mapsto e_X(1 - \alpha) - 1$ may fail to preserve the strict monotonicity properties that underlie the single sign change observed for the uniform and exponential laws.
In such circumstances, the first-order condition may admit several distinct roots, allowing for the presence of multiple local maximizers of $\alpha \:\! \bracks[\big]{ e_X(1 - \alpha) - 1 }$.

From a geometric viewpoint, this reflects the fact that the associated star-shaped set $\cE_{\:\! X \;\!\!}$ need not be convex (see Appendix~\ref{sec:appendix_convexity}).
If the one-dimensional expectile profile oscillates sufficiently across scales, the induced radial function may lose unimodality, and convexity may break down.


\section{Proof of Proposition~\ref{prop:K_d}} 
\label{sec:appendix_proof-of-K_d}

In order to establish the first claim of Proposition~\ref{prop:K_d}, it is convenient to note that the monotonicity it asserts for $\mathsf{K}_{\;\! \Gamma}$ is equivalent to the strict decrease of its logarithm; that is,
\begin{equation}
\label{eq:log-K} 
\frac{d}{dz} \ln \mathsf{K}_{\;\! \Gamma \:\!}(z) < 0 \:\! , \quad z \in \mo{[} 2 ,\:\!\! \infty \mc{[} \:\! ,
\end{equation}
a fact which we derive by means of the digamma function (namely, the polygamma function of order~$0$).

To this end, let $\psi^{\:\!\! (0) \:\!\!} \colon \mo{]} 0 \:\!,\:\!\! \infty \mc{[} \to \R$ denote the digamma function, defined by
\begin{equation}
\label{eq:digamma} 
\psi^{\:\!\! (0) \:\!\!}(z) = \frac{d}{dz} \ln \Gamma(z) \:\! , \quad z \in \mo{]} 0 \:\!,\:\!\! \infty \mc{[} \:\! .
\end{equation}
By unfolding the expression of $\ln \mathsf{K}_{\;\! \Gamma \:\!}$, namely
\begin{equation}
\label{eq:expr-log-K} 
\ln \mathsf{K}_{\;\! \Gamma \:\!}(z)
=
(z - 1) \ln \Gamma\big{(}\tfrac{z}{2}\big{)} - z \ln \Gamma\big{(}\tfrac{z \:\! - \:\! 1}{2}\big{)} - z \ln (z - 1) - \ln z + \ln 2 \:\! , \quad z \in \mo{]} 1 ,\:\!\! \infty \mc{[} 
\end{equation}
(see~\eqref{eq:K_z}), and differentiating term by term, one obtains, for every $z \in \mo{]} 1 ,\:\!\! \infty \mc{[}$,
\begin{equation}
\label{eq:log-K_comp} 
\begin{split}
\frac{d}{dz} \ln \mathsf{K}_{\;\! \Gamma \:\!}(z) &= \frac{z - 1}{2}\;\!\psi^{\:\!\! (0) \:\!\!}\big{(}\tfrac{z}{2}\big{)} - \frac{z}{2}\;\!\psi^{\:\!\! (0) \:\!\!}\big{(}\tfrac{z \:\! - \:\! 1}{2}\big{)} \\[1.25ex]
&+ \ln \Gamma\big{(}\tfrac{z}{2}\big{)} - \ln \Gamma\big{(}\tfrac{z \:\! - \:\! 1}{2}\big{)} \\[1ex]
&- \ln\:\!(z - 1) - \frac{z}{z - 1} - \frac{1}{z} \:\! .
\end{split}
\end{equation}

To control the two differences occurring in the first two lines of~\eqref{eq:log-K_comp}, we appeal to the classical Gautschi-type bounds for $\psi^{\:\!\! (0) \:\!\!}$; specifically,
\begin{equation}
\label{eq:gautschi} 
\ln z - \frac{1}{z} < \psi^{\:\!\! (0) \:\!\!}(z) < \ln z - \frac{1}{2\:\!z} \;\! , \quad z \in \mo{]} 0 \:\!,\:\!\! \infty \mc{[} \:\! .
\end{equation}
Using both inequalities in~\eqref{eq:gautschi}, we first obtain, for every $z \in \mo{]} 1 ,\:\!\! \infty \mc{[}$,
\begin{equation}
\label{eq:I} 
\begin{split}
\frac{z - 1}{2}\;\!\psi^{\:\!\! (0) \:\!\!}\big{(}\tfrac{z}{2}\big{)} - \frac{z}{2}\;\!\psi^{\:\!\! (0) \:\!\!}\big{(}\tfrac{z \:\! - \:\! 1}{2}\big{)} &< \frac{z - 1}{2} \ln \frac{z}{2} - \frac{z}{2} \ln \frac{z - 1}{2} + \frac{z}{z - 1} - \frac{z - 1}{2\:\!z} \\[0.5ex]
&= \frac{z}{2} \ln \frac{z}{z - 1} - \frac{1}{2} \ln z + \frac{z}{z - 1} - \frac{z - 1}{2\:\!z} + \frac{1}{2} \ln 2 \;\! .
\end{split}
\end{equation}
On the other hand, combining~\eqref{eq:digamma} with the upper bound furnished by~\eqref{eq:gautschi},
\begin{equation}
\label{eq:II} 
\begin{split}
\ln \Gamma\big{(}\tfrac{z}{2}\big{)} - \ln \Gamma\big{(}\tfrac{z \:\! - \:\! 1}{2}\big{)}
\:\! &= \:\!\!
\int\nolimits_{\;\!\!\frac{z \:\! - \:\! 1}{2}}^{\;\!\!\frac{z}{2}} \psi^{\:\!\! (0) \:\!\!}(\zeta) \, d\zeta \\
&< \bracks[\Big]{\;\!\! \big{(}\zeta - \tfrac{1}{2}\big{)} \:\!\! \ln \zeta - \zeta}_{\zeta \:\! = \:\! \frac{z - 1}{2} \;\!\!}^{\zeta \:\! = \:\! \frac{z}{2}} \\
&= \frac{z}{2} \ln \frac{z}{z - 1} + \ln\:\!(z - 1) - \frac{1}{2} \ln z - \frac{1}{2} \ln 2 - \frac{1}{2} \;\! .
\end{split}
\end{equation}

Gathering the estimates in~\eqref{eq:I} and~\eqref{eq:II} with the remaining terms in~\eqref{eq:log-K_comp} yields, for every $z \in \mo{]}1 ,\:\!\! \infty\mc{[}$, 
\begin{equation}
\label{eq:log-K_comp*} 
\frac{d}{dz} \ln \mathsf{K}_{\;\! \Gamma \:\!}(z)
<
z \ln \frac{z}{z - 1} - \ln z - 1 - \frac{1}{2\:\!z} \;\! .
\end{equation}
At this point, by noting that, for every $z \in \mo{]}1 ,\:\!\! \infty\mc{[}$,
\[
\ln \frac{z}{z - 1} \equiv \ln\:\!\!\Big{(}1 + \frac{1}{z - 1}\Big{)} < \frac{1}{z - 1} \:\! ,
\]
we immediately find, for every $z \in \mo{[}2 ,\:\!\! \infty\mc{[}$,
\[
\begin{split}
z \ln \frac{z}{z - 1} - \ln z - 1 - \frac{1}{2\:\!z} &\equiv (z - 1) \ln \frac{z}{z - 1} - \ln\:\!(z - 1) - 1 - \frac{1}{2\:\!z} \\[0.75ex]
&< - \ln\:\!(z - 1) - \frac{1}{2\:\!z} \\[1.25ex]
& < 0 \:\! .
\end{split}
\]
In conjunction with~\eqref{eq:log-K_comp*}, this establishes~\eqref{eq:log-K}. 

We now turn to the asymptotic behavior of $\mathsf{K}_{\;\! \Gamma \:\!}$.  
Recall first Stirling's formula in the form
\[
\Gamma(z + 1) \sim \sqrt{2\:\!\pi} \, \frac{\;\! z^{\;\! z \:\! + \:\! 1/2}}{e^{\:\! z}} \:\! , \quad z \to \infty \;\! ,
\]
where the symbol $\sim$ denotes asymptotic equivalence (the ratio of the two sides tending to $1$).
As a standard consequence, for any fixed $r \in \R$,
\[
\Gamma(z + r) \sim z^{\:\! r} \:\! \Gamma(z) \:\! , \quad z \to \infty \;\! .
\]
Applying this relation to Gamma factors appearing in the expression~\eqref{eq:K_z} readily yields the explicit asymptotic upper bound stated in~\eqref{eq:lim_infty_K_d}.

Finally, turning to the complementary asymptotic regime, observe that $\Gamma(\,\bm{\cdot}\,)$ is continuous and satisfies
\begin{equation}
\label{eq:gamma-function} 
\Gamma(z + 1) = z \;\! \Gamma(z) \:\! , \quad z \in \mo{]} 0 \:\!,\:\!\! \infty\mc{[} \:\! .
\end{equation}
Hence $z \;\! \Gamma(z) \to \Gamma(1) = 1$ as $z \to 0$ (not monotonically), and the limit~\eqref{eq:lim_1_K_d} follows immediately.
It should be stressed, however, that this argument establishes the limit only, and does not provide any information on monotonicity.
The latter will be settled below, where we prove that the function $\ln \mathsf{K}_{\;\! \Gamma \:\!}$ is in fact strictly decreasing on a sufficiently small right neighborhood of~$1$.

With this aim, we return to the explicit logarithmic expansion~\eqref{eq:expr-log-K}.
It proves convenient to reparametrize the variable as $z = 1 + \zeta$, with $\zeta \in \mo{]} 0 \:\!,\:\!\! \infty\mc{[}$ sufficiently small, and thus to examine $\ln \mathsf{K}_{\;\! \Gamma \:\!}(1 + \zeta)$ in place of $\ln \mathsf{K}_{\;\! \Gamma \:\!}(z)$ in a neighborhood of~$1$.
Since differentiation is effected with respect to $z$, this change of variable does not alter the computation: indeed, for any $z^\ast \:\!\! \in \mo{]}1 ,\:\!\! \infty\mc{[}$ (in particular for $z^\ast \;\!\!$ sufficiently close to $1$),
\[
\frac{d}{dz} \ln \mathsf{K}_{\;\! \Gamma \:\!}(z) \, \bigg{|}_{\, z \:\! = \:\! z^\ast \:\!\!}
\:\!\!\! =
\frac{d}{d\zeta} \ln \mathsf{K}_{\;\! \Gamma \:\!}(1 + \zeta) \, \bigg{|}_{\, \zeta \:\! = \:\! z^\ast \;\!\! - \:\! 1} \;\!\! . 
\]
Accordingly, for $\zeta \in \mo{]} 0 \:\!,\:\!\! \infty\mc{[}$, the preceding identity yields
\begin{equation}
\label{eq:expr-log-K_*} 
\ln \mathsf{K}_{\;\! \Gamma \:\!}(1 + \zeta)
=
\zeta \ln \Gamma\big{(}\tfrac{1}{2} + \tfrac{\zeta}{2}\big{)} - (1 + \zeta) \ln \Gamma\big{(}\tfrac{\zeta}{2}\big{)} - (1 + \zeta) \ln \zeta - \ln (1 + \zeta) + \ln 2 \:\! . 
\end{equation}

The representation~\eqref{eq:expr-log-K_*} naturally leads to a local analysis of the two Gamma terms as $\zeta \downarrow 0$. Each of them admits a canonical expansion at the point approached by its argument, and the coefficients may be expressed through the digamma function~\eqref{eq:digamma}. Thus, using the Taylor formula for $\ln\Gamma$ at~$\frac{1}{2}$,
\begin{equation}
\label{eq:taylor_1} 
\ln \Gamma\big{(}\tfrac{1}{2} + \tfrac{\zeta}{2}\big{)}
=
\frac{1}{2} \ln \pi + \frac{\psi^{\:\!\! (0) \:\!\!}\big{(}\tfrac{1}{2}\big{)}}{2} \;\! \zeta + \cO\big{(}\:\!\zeta^{\:\! 2}\:\!\big{)} \:\! , \quad \zeta \downarrow 0 \:\! ,
\end{equation}
where we employed also the identity $\Gamma\big{(}\tfrac{1}{2}\big{)} = \sqrt{\pi}$; whereas the classical small-argument expansion of $\ln\Gamma$, in accordance with~\eqref{eq:gamma-function}, yields
\begin{equation}
\label{eq:taylor_2} 
\ln \Gamma\big{(}\tfrac{\zeta}{2}\big{)}
=
\ln 2 - \ln \zeta + \frac{\psi^{\:\!\! (0) \:\!\!}(1)}{2} \;\! \zeta + \cO\big{(}\:\!\zeta^{\:\! 2}\:\!\big{)} \:\! .
\end{equation}

At this stage, the Euler–Mascheroni constant $\gamma$ naturally appears. 
Indeed, recalling the well-known identity $\psi^{\:\!\! (0) \:\!\!}(1) = -\:\!\gamma$ (which may be established in various ways), one readily derives
\[
\psi^{\:\!\! (0) \:\!\!}\big{(}\tfrac{1}{2}\big{)} = -\:\!\gamma - 2 \ln 2 \:\! ,
\]
for instance, by differentiating the classical Legendre duplication formula.
Consequently,~\eqref{eq:taylor_1} and ~\eqref{eq:taylor_2} assume the more explicit form
\[
\begin{split}
\ln \Gamma\big{(}\tfrac{1}{2} + \tfrac{\zeta}{2}\big{)}
&=
\frac{1}{2} \ln \pi - \frac{1}{2} \:\! \big{(}\gamma + 2 \ln 2\big{)} \:\! \zeta + \cO\big{(}\:\!\zeta^{\:\! 2}\:\!\big{)} \:\! , \\[0.5ex]
\ln \Gamma\big{(}\tfrac{\zeta}{2}\big{)}
&=
\ln 2 - \ln \zeta - \frac{\gamma}{2} \;\! \zeta + \cO\big{(}\:\!\zeta^{\:\! 2}\:\!\big{)} \:\! .
\end{split}
\]
Observing that $\ln (1 + \zeta) = \zeta + \cO\big{(}\:\!\zeta^{\:\! 2}\:\!\big{)}$ as $\zeta \downarrow 0$, and inserting all these expansions into~\eqref{eq:expr-log-K_*}, we obtain
\[
\ln \mathsf{K}_{\;\! \Gamma \:\!}(1 + \zeta)
=
\frac{1}{2} \:\! \big{(}\gamma + \ln \tfrac{\pi}{4} - 2\big{)} \:\! \zeta + \cO\big{(}\:\!\zeta^{\:\! 2}\:\!\big{)} \:\! , \quad \zeta \downarrow 0 \:\! ,
\]
and therefore
\[
\frac{d}{d\zeta} \ln \mathsf{K}_{\;\! \Gamma \:\!}(1 + \zeta)
=
\frac{1}{2} \:\! \big{(}\gamma + \ln \tfrac{\pi}{4} - 2\big{)} + \cO(\:\!\zeta\:\!) \:\! , \quad \zeta \downarrow 0 \:\! .
\]
Since $\gamma + \ln \tfrac{\pi}{4} - 2 < 0$, the final assertion follows at once.









\end{document}